\documentclass[reqno]{amsart}
\usepackage{amssymb}
\usepackage{enumerate}
\usepackage{color}
\usepackage[colorlinks=false]{hyperref}
\hypersetup{citecolor=blue}

\numberwithin{equation}{section}

\newtheorem{theorem}{Theorem}
\newtheorem{lemma}{Lemma}[section]
\newtheorem{proposition}[lemma]{Proposition}
\theoremstyle{remark}
\newtheorem{remark}{Remark}[section]

\newcommand\MScN[1]{\href{http://www.ams.org/mathscinet-getitem?mr=#1}{\nolinkurl{(#1)}}}
\newcommand\DOI[1]{\href{http://dx.doi.org/#1}{(doi: \nolinkurl{#1})}}
\newcommand\LINK[1]{\href{#1}{(link: \nolinkurl{#1})}}

\newcommand\R{\mathbb{R}}
\newcommand\C{\mathbb{C}}
\newcommand\N{\mathbb{N}}
\newcommand\Z{\mathbb{Z}}
\newcommand\Loc{{\mathrm{loc}}}

\DeclareMathOperator\SPAN{span}
\DeclareMathOperator\SIGN{sign}

\title[Blow-up profiles for 1D critical NLS]
{Non--flat conformal blow--up profiles for the 1D critical nonlinear Schr\"odinger equation$^*$}

\author[Y. Martel]{Yvan Martel}
\address{Université Paris-Saclay, UVSQ, CNRS, Laboratoire de mathématiques de Versailles, 78035 Versailles, France}
\email{\href{mailto:yvan.martel@uvsq.fr}{yvan.martel@uvsq.fr}}

\author[I. Naumkin]{Ivan Naumkin}
\address{Departamento de Física Matemática, Instituto de Investigaciones
en Matemáticas Aplicadas y en Sistemas. Universidad Nacional Autónoma
de México, Apartado Postal 20-126, Ciudad de México, 01000, México.}
\email{\href{mailto:ivan.naumkin@iimas.unam.mx}{ivan.naumkin@iimas.unam.mx}}

\thanks{
$^*$This article is a product of a collaboration with Thierry Cazenave to whom 
the authors are deeply indebted.}

\subjclass[2010]{35Q41; 35B44; 37K40}

\begin{document}
\begin{abstract}
For the critical one-dimensional nonlinear Schr\"odinger equation, we construct blow-up solutions 
that concentrate a soliton at the origin at the conformal blow-up rate, with a non-flat blow-up profile.
More precisely, we obtain a blow-up profile that equals $|x|+i\kappa x^2$ near the origin,
where $\kappa$ is a universal real constant.
Such profile differs from the flat profiles obtained in the same context by Bourgain and Wang \cite{BourgainW}.
\end{abstract}

\maketitle

\section{Introduction}
We consider the mass critical Schrödinger equation in one dimension
\begin{equation}\label{nls}
i \partial_t u + \partial_x^2 u + |u|^4 u = 0 \quad (t,x)\in\R\times\R.
\end{equation}
Recall that the mass $M(u)$, the momentum $J(u)$ and the energy $E(u)$ of a solution $u$ of \eqref{nls} are formally conserved
\begin{align*}
M(u)&=\int_{\R} |u|^2 \, dx \\
J(u)&=\Im \int_{\R} (\partial_x u) \overline u \, dx\\
E(u)&=\frac12 \int_{\R} |\partial_xu|^2 \, dx-\frac16 \int_{\R} |u|^6 \, dx .
\end{align*}
The Cauchy problem for~\eqref{nls} is locally well-posed in $L^2(\R)$:
for any $u_0 \in L^2(\R)$, there exists a unique maximal solution $u$ of~\eqref{nls} in $\mathcal C([0,T), L^2(\R))$ satisfying $u(0,\cdot)=u_0$ (see e.g. \cite{CLN10} for details).
If in addition $u_0 \in H^1(\R)$, then the solution belongs to $\mathcal C([0,T), H^1(\R))$;
moreover, if the maximal time of existence $T>0$ is finite then 
\[
\liminf_{t \uparrow T} (T-t)^{\frac 12}\|\partial_x u(t)\|_{L^2}>0.
\]

Let $Q$ be the positive, even, $H^1$ solution of $- Q''+ Q= Q^5$ on $\R$ given by
\begin{equation}\label{def:Q}
Q(x) = 3^{ \frac{1}{4} }
 (\cosh ({2 } x ) )^{-\frac12}.
\end{equation}
It is well-known \cite{Weinstein83} that if the $H^1$ solution $u$ satisfies
 $\|u(t)\|_{L^2} < \|Q\|_{L^2}$, then it is global and bounded in $H^1$.

The pseudo-conformal symmetry of the equation implies that if $u(t,x)$ is a solution of \eqref{nls} then
\[
v(t,x)=\frac 1{|t|^\frac12} u\left( -\frac 1t,\frac{x}{|t|}\right) e^{i \frac{|x|^2}{4t}}
\]
is also a solution of \eqref{nls}.
By this symmetry, the solitary wave solution $u(t,x) = e^{it} Q(x)$ yields the
explicit blow-up solution
\begin{equation}\label{def:S}
S(t,x) = \frac 1{|t|^\frac12} Q\left( \frac{x}t \right) e^{-\frac{i x^2}{4t}} e^{\frac it}
\end{equation}
defined on $(-\infty,0)\times\R$.
This solution is called the minimal mass solution since
\[
\|S(t)\|_{L^2} = \|Q\|_{L^2},\quad \lim_{t\uparrow 0} \|\partial_x S(t)\|_{L^2}=+\infty.
\]
Moreover, for $t<0$ close to $0$,
\[
\|\partial_x S(t) \|_{L^2} \sim \frac C{|t|}
\]
and $|t|^{-1}$ is thus called the conformal blow-up rate.
In addition, it was proved in~\cite{Me93} that $S$ is the unique (up to the invariances of \eqref{nls}) blow-up solution with the minimal mass.

The general blow-up problem for \eqref{nls} is still largely open, but a lot of information is known for blow-up solutions close to the
ground state $Q$, more precisely under the following assumption on the mass of the initial data
\begin{equation}\label{eq:data}
\int Q^2 < \int u_0^2 < \int Q^2+\delta,
\end{equation}
where $\delta>0$ is small.
We recapitulate the main known results, in the context of the \emph{one-dimensional space variable} with \emph{even symmetry} on the initial data.
Suppose that $u_0\in H^1(\R)$ is even, satisfies \eqref{eq:data} and that \emph{the corresponding solution $u$ blows up in finite time $0<T<+\infty$}, then the following properties hold.
\begin{description}
\item[Orbital stability \cite{Weinstein83}]
There exist $\gamma(t)$ and $\lambda(t)$ such that
\[
\sup_{[0,T)}\|e^{-i\gamma(t)} \lambda^\frac12(t) u(t,\lambda(t) x) - Q\|_{H^1} \quad \mbox{is small as $\delta$ is small.}
\]
\item[Asymptotic stability \cite{MerleR1, MerleR2}]
There exist sequences $(\gamma_n)$, $(\lambda_n)$ and $(t_n)$ with $t_n\uparrow T$, such that
\[
\lim_{n\to +\infty} e^{-i\gamma_n} \lambda_n^\frac12 u(t_n,\lambda_n x) \to Q \quad \mbox{in $\dot H^1$ as $n\to +\infty$.}
\]

\item[Dichotomy and gap for the blow-up rate \cite{Raphael}] The solution $u$
\begin{itemize}
\item either blows up with the \emph{loglog rate}
\begin{equation}\label{loglog}
\lim_{t\uparrow T}
\sqrt{\frac {T-t}{\log|\log (T-t)|}} \left\|\partial_x u(t)\right\|_{L^2}  = \frac{\|Q'\|_{L^2}}{\sqrt{2\pi}},
\end{equation}
\item or blows up faster than the \emph{conformal rate}
\begin{equation}\label{conformal}
\liminf_{t\uparrow T} \left(T-t\right) \left\|\partial_x u(t)\right\|_{L^2}  >0.
\end{equation}
\end{itemize}
\item[Blow-up profile \cite{MerleR1}]
There exist $\gamma(t)$, $\lambda(t)$ and $r_*\in L^2(\R)$ such that
\[
u(t) - \frac{e^{i\gamma(t)}}{\lambda^\frac12(t)}Q\left(\frac{x}{\lambda(t)}\right) 
\to r_* \quad \hbox{in $L^2(\R)$ as $t\uparrow T$.}
\]
Moreover,
\begin{equation}\label{equivalent}
  \mbox{($r_*\in H^1(\R)$ and $r_*(0)=0$)}  \iff \mbox{ \eqref{conformal} holds}
\end{equation}
\item[Generic log-log blow-up rate \cite{MerleR2}]
The set of initial data in $H^1$ satisfying \eqref{eq:data} and such that the solution blows up 
with the loglog blow-up rate \eqref{loglog} is open in $H^1$. It contains all the initial data with negative energy
such that \eqref{eq:data} holds.
\item[Existence of non-generic blow-up \cite{BourgainW,KriegerS,MerleRS}] For $A>0$ large, if $r_*\in H^A(\R)$ satisfies
\begin{equation}\label{AAA}
(1+|x|^2)^\frac12 r_*\in L^2(\R) \quad \mbox{and}\quad r_*^{(k)}(0)=0 \quad \mbox{for any $0\leq k\leq A$,}
\end{equation}
then, there exist $\tau>0$ and a solution $u$ of \eqref{nls} on $[-\tau,0)$ such that
\[
\lim_{t\uparrow 0} u(t) -S(t) = r_* \quad \mbox{in $H^1(\R)$}.
\]
\end{description}

The minimal mass solution \eqref{def:S} and the non-generic blow-up solutions constructed in \cite{BourgainW}
have the conformal blow-up rate $\|\partial_x u(t)\|_{L^2} \sim C (T-t)^{-1}$ for $t\sim T$.
The  existence of solutions blowing up \emph{strictly faster} than the   conformal rate 
is an open problem, which is equivalent, by   the pseudo--conformal symmetry, to the existence of global solutions blowing up in infinite time. The only known example \cite{MartelR} contains at least two bubbles, which means that \eqref{eq:data} is not satisfied.

The $\log \log$ rate is physically more relevant than the conformal rate since it was proved to be stable under the assumption~\eqref{eq:data} (the set of blow-up solutions exhibited in \cite{MerleR2} is open in $H^1$). However, it is interesting mathematically to further investigate the conformal blow-up case. For example, the conformal blow-up rate is seen as a frontier separating solutions with the loglog rate
and scattering, as discussed in \cite{BourgainW,KriegerS,MerleRS}.

In this paper, we focus on the blow-up profile $r_*$ as defined in \cite{MerleR1}, in the conformal blow-up case.

\begin{theorem} \label{theorem1} 
There exist $\kappa\in \R$, $\delta_0>0$ and $C>0$ with the following property.
For any $\delta\in (0,\delta_0)$, there exist $\tau>0$ and a solution $u\in \mathcal C([-\tau,0),H^1(\R))$ of~\eqref{nls} which blows up at time $0$ with the asymptotics
\begin{align}
&\lim_{t\uparrow 0} |t| \, \|\partial_x u(t)\|_{L^2} = \|Q'\|_{L^2},\label{eq:thm0}\\
&\lim_{t\uparrow 0} \left\{u(t)-S(t)\right\}=
\lim_{t\uparrow 0} \biggl\{u(t)-\frac{e^{\frac it}}{|t|^{\frac12}} Q\left(\frac x t\right)\biggr\}=
 r_* \quad \mbox{in $L^2(\R)$},\label{eq:thm1}
\end{align}
where the function $r_*\in H^1(\R)$ satisfies $\|r_*\|_{L^2} \leq C \delta^\frac32$ and
\begin{equation}\label{eq:thm1bis}
r_*(x) = |x|+i \kappa x^2
\quad \mbox{for all $x\in (-\delta,\delta).$}
\end{equation}
\end{theorem}

\begin{remark}
This example illustrates the fact that a blow-up profile in the conformal case need not 
be flat at a blow-up point as one could have expected from previous constructions
of conformal blow-up solutions in \cite{BourgainW,KriegerS,MerleRS}; see \eqref{AAA}.
\end{remark}

\begin{remark}
The constant $\kappa$ is defined in \eqref{def:kappa} ; see also Remark~\ref{rk:kappa}.
\end{remark}

\begin{remark}
We expect that the proof of Theorem~\ref{theorem1} can be extended to more general blow-up profiles of the form
\begin{equation}\label{eq:profile}
r_*(x) = z_1 (|x|+i \kappa x^2) + z_2 |x|^3 + w_*(x)
\end{equation}
in the neighborhood of $0$ where
$z_1,z_2\in \C$ and $w_*$ is a term of order $x^4$ at $x=0$, with suitable regularity properties.
We do not pursue this issue here for the sake of simplicity.

It is an open problem to determine all possible blow-up profiles, in terms of regularity and behavior at the blow-up point.
On the one hand, in the one-dimensional case, from \cite{MerleR1} , the condition $r_*(0)=0$ is necessary for $r_*$ to be a blow-up profile
(see  \eqref{equivalent}).
On the other hand, we do not know whether the implicit relation $r_*''(0)=2i\kappa r_*'(0)$ in \eqref{eq:profile} is  necessary or is only
a technical assumption due to our approach. 
\end{remark}

\begin{remark}
The question of the blow-up rate and the blow-up profile has been studied for other nonlinear dispersive or wave PDEs ; see for example the following articles:
\cite{Jendrej,Jendrej2016,JendrejLR,KimK,KimK2, Kim,MartelMR1,MartelMR3,MartelP,RaphaelS}.
Analogies between the present work and \cite{MartelP}, devoted to the critical generalized Korteweg-de Vries equation, are discussed in the sketch of the proof.
\end{remark}

\subsection*{Sketch of the proof}

To analyse the blow-up phenomenon, we introduce the rescaled variable 
\begin{equation}\label{cov}
u (t,x)=\lambda^{-\frac 12}(s) e^{i\gamma(s)} U (s, y ),
\quad y = \frac {x}{\lambda},\quad ds = \frac {dt}{\lambda^2},
\end{equation}
where $y\in \R$ and $s\gg 1$ (which means that $t$ is close to the blow up time) and where
the scaling and phase parameters $\lambda>0$ and $\gamma$ are functions to be determined.
By direct computation, 
\begin{equation} \label{eq:resc1}
i \partial_t u + \partial_x^2 u +|u|^4 u =\lambda^{-\frac 52} e^{i\gamma} \mathcal E(U),
\end{equation}
where
\begin{equation}\label{def:EE}
\mathcal E(U) :=i \partial_s U + \partial_y^2 U - U+ |U|^4 U
- i \frac{\lambda_s}{\lambda} \Lambda U- (\gamma_s - 1) U .
\end{equation}
It is easily checked using $Q''+Q^5=Q$ that the function
$U(s,y) = e^{-ib(s) \frac {y^2}4} Q(y)$ satisfies $\mathcal E(U)=0$ for the special choice
of parameters
\begin{equation}\label{eq:pS}
\gamma(s) = s,\quad \lambda(s) = \frac 1s,\quad b(s)= \frac 1s,
\end{equation}
which corresponds to the minimal mass solution $S(t)$ defined in \eqref{def:S}, and thus
to $r_*=0$ in \eqref{eq:thm1}.
Any residual profile $r_*$ satisfying a strong flatness condition such as \eqref{AAA} will decouple from the soliton and in such case it is enough to consider an approximate solution containing only a rescaled soliton, as in~\cite{BourgainW,MerleRS}.
The main new ingredient of the proof of Theorem~\ref{theorem1} consists in constructing a suitable approximate solution $V$ of $\mathcal E(V)=0$.
 In contrast with the flat case, to generate the specific asymptotic behavior \eqref{eq:thm1}--\eqref{eq:thm1bis}, where the residual profile is expected to  interact strongly with the bubble, the approximate solution $V$ has to include additional terms.
We set
\begin{equation}\label{def:Vintro:2}
V(s,y) =  e^{-ib \frac {y^2}4} Q(y) + \theta(s,y) \lambda^\frac 32 \left( \varphi_1(y) \cos \gamma + i \psi_1(y)\sin \gamma\right),
\end{equation}
with a suitable cut-off function $\theta$ ($\theta\equiv 1$ close to the soliton) and 
where
\[
\gamma(s) \sim s,\quad \lambda(s) \sim \frac 1s,\quad b(s) \sim \frac 1s.
\]
We now discuss the choice of the pair of functions $(\varphi_1,\psi_1)$;
the form of $V$ in \eqref{def:Vintro:2} and the choice of the cut-off function $\theta$ are justified later.
By linearization, we obtain formally
\[
\mathcal E(V) \sim -\lambda^\frac32 \left[ (\psi_1 + L_+\varphi_1)\cos \gamma 
+ i (\varphi_1 + L_- \psi_1) \sin \gamma \right],
\]
where the operators $L_\pm$ are defined in \eqref{def:Lpm} and where we have discarded higher order terms.
It is thus natural to consider the system
\begin{equation}\label{eq:syst}
\begin{pmatrix} L_+ & 1 \\ 1 & L_- \end{pmatrix} \begin{pmatrix} \varphi \\ \psi\end{pmatrix}= 0.
\end{equation}
From \cite[Proposition~2.1.4]{Perelman}, the system \eqref{eq:syst} does not have any non zero bounded solution
$(\varphi,\psi)$.
However, it follows from ODE arguments (see Appendix \ref{appendix}) that this system has 
an even solution $(\varphi_1,\psi_1)$ with the following asymptotic behavior as $|y|\to +\infty$,
\begin{equation}\label{eq:asympp}
\varphi_1(y) \sim |y|  ,\quad
\psi_1 (y) \sim -|y| .
\end{equation}
(See Lemma \ref{lem1bis} for a precise statement.)
We now justify that the asymptotic behavior \eqref{eq:asympp} of $(\varphi_1,\psi_1)$ as $|y|\to +\infty$, the multiplicative factor $\lambda^\frac32$ in \eqref{def:Vintro:2} and the choice of a suitable cut-off function 
$\theta$ generate the first term of the profile $r_*$  in the original variables $(t,x)$. Indeed, from \eqref{def:Vintro:2},
the asymptotic behavior of $V(s,y)$ for $|y|$ large corresponds to the function $R_*(s,y)=\lambda^\frac 32 e^{-i \gamma}|y|$ and after the change of variable \eqref{cov}, we recognize the first term $|x|$ in the profile $r_*$ defined by \eqref{eq:thm1bis}.
For a different multiplicative factor $\lambda^\alpha$, $\alpha\neq \frac 32$, the ansatz would  be  either too large  (the approximate solution is not close to $Q$ in $L^2$) or too small (the solution is simply $S(t,x)$).
The second term $i \kappa x^2$ in \eqref{eq:thm1bis} is due to an additional correction term in the definition of $V$ (see Section \ref{S:3:2}) which is needed in our method to obtain manageable error terms.

Note that the idea of using a pair $(\varphi_1,\psi_1)$ solution of \eqref{eq:syst}
and the choice of a cut-off function $\theta$ related to the original variable (see Section \ref{S:3:2})
was inspired by the construction technique introduced in \cite{MartelP}
(see the definition of $A_1$ and formula (1.16) in \cite{MartelP}), devoted to the existence of a new blow-up rate for the critical generalized KdV equation. We also refer to \cite{MartelMR1,MartelMR3} and references therein for more blow-up results concerning that equation.

Once the approximate solution $V$ is constructed in Section~\ref{S:3} (see Lemma \ref{le:V} and Proposition \ref{PR:1} for its main properties), the proof scheme follows the ones in \cite{Me90,RaphaelS,MartelR}. 
For a sequence of solutions of \eqref{nls} constructed backwards in time close to the approximate solution, we establish uniform estimates for the parameters (by modulation and ODE techniques, see Sections \ref{S:4} and \ref{S:6}) and for the error term (by energy arguments inspired by \cite{RaphaelS}, see Section \ref{S:5}). In Section~\ref{S:7}, passing to the strong limit in $L^2$, we obtain a solution of \eqref{nls} as stated in Theorem~\ref{theorem1}. 

Compared to previous works on blowup for the nonlinear Schrödinger equation, a key difficulty comes from oscillations in time through the phase parameter $\gamma$, due to that fact that the solitonic bubble has a phase $\gamma\to +\infty$ at the blow-up time, while the profile $r_*(x)$ is time independent. 
Several time dependent oscillatory quantities  (like \eqref{eq:O} and \eqref{fdfnJ}) are introduced to use cancellations in mean value in time while estimating the parameters.
\subsection*{Notation}
Set
\begin{equation*} 
f(u) = |u|^4 u, \quad
F(u) = \frac {1} {6} |u|^6 .
\end{equation*} 
We denote 
\begin{equation*} 
\langle g , h \rangle = \Re \int g \overline{h}.
\end{equation*}
Set
\[
\Sigma=\{ g \in H^1(\R) ; xg\in L^2(\R)\}.
\]
Let $\Lambda$ be the generator of $L^2$-scaling in 1D:
\[ \Lambda g = \frac 12 g+x \partial_x g. \]
In particular,
\begin{equation} \label{fortho3b} 
\langle g_1, \Lambda g_2 \rangle + \langle g_2, \Lambda g_1 \rangle = 0,
\end{equation} 
for all $g_1, g_2\in \Sigma$, hence
\begin{equation} \label{fortho3} 
\langle g, \Lambda g \rangle = 0,
\end{equation} 
for all $g\in \Sigma$. 
In addition,
\begin{equation} \label{fortho3c} 
\langle \partial_x^2 g , \Lambda g \rangle = - \| \partial _x g \| _{ L^2 }^2,
\end{equation} 
for all $g \in H^2 (\R) $ with $\partial_x g \in\Sigma $. 
More generally, for $k\in \Z$, we let
\begin{equation} \label{fdfnlj} 
 \Lambda_k g=\frac {1-k}2 g+x \partial_x g,
\end{equation} 
so that $\Lambda = \Lambda_0$. 

We define the vector space
\begin{multline*}
\mathcal Y = \Bigl\{ g\in \mathcal C^\infty (\R, \R); g \text{ is even and } \\
\exists q\in \N, \forall p\in \N, \exists C>0, 
\forall x\in \R, | g^{(p)} (x) |\le C \langle x\rangle ^q e^{- |x|} \Bigr\}
\end{multline*} 
where $\langle x \rangle = (1+x^2)^{\frac 12}$.

Moreover, we fix an even function $\chi \in \mathcal C^\infty (\R, \R)$ such that $\chi \ge 0$, $\chi (x)= 0$ for $ |x|\le 1/2$ and $\chi (x)= 1$ for $ |x|\ge 1$, and we define the functions $\mu_k$, $k\ge 0$, by
\begin{equation} \label{fdfnmuk} 
\mu_k (x)= |x|^k \chi (x). 
\end{equation} 
We let
\begin{equation} \label{fdefzk} 
\mathcal Z_k = \mathcal Y + \SPAN\{ \mu_0, \ldots, \mu_k \} . 
\end{equation} 
Using the operators $\Lambda_j$ defined by~\eqref{fdfnlj}, it is not difficult to see that
for any $j, k\in \N$,
\begin{align} 
\Lambda_j \mathcal Z_k &\subset \mathcal Z_k \label{fcomm1} \\
\Lambda_{2k+3}  \mathcal Z_{k+1} &\subset \mathcal Z_{k} \label{fcomm2} \\
\Lambda_1  \mathcal Z_0 &\subset \mathcal Y \label{fcomm3} 
\end{align} 

\subsection*{Acknowledgements.}
The second author is Fellow of Sistema Nacional de Investigadores. His research was partially supported by the projects CONACYT, FORDECYT-PRONACES 429825/2020 and PAPIIT-DGAPA UNAM IA100422.

\section{Linearization around the soliton}

\subsection{Linearized operator}\label{S:2:1}

The linearization of~\eqref{nls} around the special solution $u(t,x)=e^{it}Q(x)$ involves the following Schr\"odinger operators:
\begin{equation}\label{def:Lpm}
L_+:=-\partial_x^2+1-5Q^4,\qquad
L_-:=-\partial_x^2+1-Q^4.
\end{equation}
Denote by $\rho\in \mathcal Y$ the unique radial $H^1$ solution of 
\[
L_+\rho =\frac{x^2}4Q.
\]
See~\cite[Appendix~B]{Weinstein}.
We recall the \emph{generalized null space relations} \cite[Appendix~B]{Weinstein}
\begin{equation} \label{fortho1} 
L_-Q=0,\quad L_+(\Lambda Q)=-2Q,\quad L_-(x^2Q)=-4\Lambda Q,\quad L_+\rho =\frac{x^2}4Q. 
\end{equation}
We check that
\begin{equation} \label{eq:rho}
\langle \rho, Q\rangle
=-\frac 12 \langle L_+ \rho, \Lambda Q\rangle
=-\frac 18 \langle x^2 Q, \Lambda Q\rangle = \frac18 \int x^2Q^2 >0.
\end{equation} 
For any $a\in \R$, we set
\[
Q_a = Q +a \rho
\]
so that
\begin{equation} \label{eq:Qa} 
Q_a'' - Q_a + Q_a ^5 + a \frac{x^2}{4} Q = Q_a^5-Q^5-5aQ^4\rho.
\end{equation}
We will use the following well-known coercivity properties.

\begin{lemma} \label{lem0}
There exists a constant $\zeta _0>0$ such that 
\begin{equation*} 
\langle L_+ g,g\rangle \ge \zeta _0 \|g \| _{ H^1 }^2 - \frac {1} {\zeta _0} \Bigl( \langle g, Q\rangle ^2 + \langle g, y^2 Q\rangle ^2 \Bigr)
\end{equation*} 
and
\begin{equation*} 
\langle L_- g,g\rangle \ge \zeta _0 \|g\| _{ H^1 }^2 - \frac {1} {\zeta _0} \Bigl( \langle g, \rho \rangle ^2 + \langle g, \Lambda Q\rangle ^2 \Bigr)
\end{equation*} 
for any real-valued function $g\in H^1 (\R) $.
\end{lemma}

\begin{proof} 
The estimates are immediate consequences of Propositions~2.9 and~2.10 in~\cite{Weinstein}.
\end{proof} 

\subsection{Linearized system}
For the construction of the approximate blow-up solution, we will use systems involving the 
operators $L_+$ and $L_-$, as introduced below.

\begin{lemma}\label{lem1bis}
The following properties hold. 

\begin{enumerate}[{\rm (i)}] 

\item \label{lem1bis:1} 
There exists a unique pair of even functions $( \varphi_1, \psi_1 ) \in \mathcal Z_1^2 $ 
such that
\begin{equation} \label{ellip0bis} 
\begin{cases} 
\psi_1 + L_+ \varphi_1 = 0\\
\varphi_1 + L_- \psi_1 = 0 
\end{cases} 
\end{equation} 
and satisfying
\[
\varphi_1(x) = \mu_1(x) + c_1 + v_1(x),\quad
\psi_1(x) = - \mu_1(x) - c_1 + w_1(x)
\]
where $c_1\in\R$ and $v_1,w_1\in \mathcal Y$.

\item \label{lem1bis:2} 
Given a pair of even functions $(g,h)\in \mathcal Y^2$, there exists a unique pair of even functions $(\varphi,\psi) \in \mathcal Z_0^2$ solution of the system
\begin{equation} \label{ellipbis} 
\begin{cases} 
\psi + L_+ \varphi = g \\
\varphi + L_- \psi = h 
\end{cases} 
\end{equation}
and satisfying
\[
\varphi(x) = c + v(x),\quad
\psi(x) = - c + w(x)
\]
where $c\in\R$ and $v,w\in \mathcal Y$.
\end{enumerate} 
\end{lemma}

The proof of Lemma~\ref{lem1bis} is given in Appendix \ref{appendix}.
The key ingredient is a result from \cite{Perelman} concerning the non existence of a resonance pair for \eqref{ellip0bis}, i.e. the non existence of a pair of non zero bounded functions satisfying
\eqref{ellip0bis}.

We define
$\alpha_1\in \R$ by
\begin{equation} \label{def:a1}
\langle \Lambda_3 \varphi_1 - x^2 Q^4 \varphi_1+ \alpha_1 \rho,Q\rangle =0
\end{equation}
and $\kappa\in \R$ by
\begin{equation} \label{def:kappa}
\kappa= - \frac {c_1}2 .
\end{equation}
\begin{remark}\label{rk:kappa}
The constant $\kappa$ appears in the statement of Theorem~\ref{theorem1}.
From the definition of $c_1$ in part (i) of Lemma~\ref{lem1bis}, it does not seem clear how to determine $\kappa$ explicitly.
\end{remark}

Now, we define another pair of functions that will be useful in the construction of the
approximate rescaled blow-up solution.

\begin{lemma} \label{lem2}
There exists a unique pair of even functions 
$(\varphi_2,\psi_2)\in \mathcal Z_2^2$ solution of the system
\begin{equation} \label{ellip2} 
\begin{cases} 
\psi_2 + L_+ \varphi _2 = -\Lambda_3 \psi_1- x^2 Q^4 \psi_1 \\
\varphi _2 + L_- \psi_2 = -\Lambda_3 \varphi_1 + x^2 Q^4 \varphi_1- \alpha_1 \rho
\end{cases}
\end{equation}
and satisfying
\begin{equation*}
\varphi_2(x) =\kappa x^2 + c_2 + v_2(x),\quad
\psi_2(x) = -\kappa x^2 - c_2 + w_2(x)
\end{equation*}
where $c_2 \in \R$ and $v_2, w_2\in \mathcal Y$.
\end{lemma}
\begin{proof}
Note from \eqref{fcomm2} with $k=0$ that $\Lambda_3 \varphi_1,\Lambda_3 \psi_1\in \mathcal Z_0$.
More precisely, from Lemma~\ref{lem1bis}, we have
\begin{align*}
-\Lambda_3 \psi_1- x^2 Q^4 \psi_1 & = c_1 + g_2,\\
-\Lambda_3 \varphi_1 + x^2 Q^4 \varphi_1- \alpha_1 \rho & = -c_1 + h_2 
\end{align*}
where $g_2$, $h_2\in \mathcal Y$.
In particular, $(\varphi_2,\psi_2)$ solves \eqref{ellip2} if and only if setting
$\check\varphi_2 = \varphi_2+\frac 12 c_1 x^2$, $\check \psi _2 = \psi _2 - \frac 12 c_1 x^2$, the pair
$(\check \varphi_2,\check \psi_2)$ satisfies
\begin{equation*} 
\begin{cases} 
\check\psi_2 + L_+ \check\varphi_2 = \check g_2 \\
\check\varphi_2 + L_- \check\psi_2 = \check h_2
\end{cases}
\end{equation*}
where
\[
\check g_2 = g_2 - \frac 52 c_1 x^2 Q^4,\quad 
\check h_2 = h_2 + \frac 12 c_1 x^2 Q^4.
\]
As $\check g_2$, $\check h_2\in \mathcal Y$, by~\eqref{lem1bis:2} of Lemma~\ref{lem1bis} there exists a unique pair of functions $(\check\varphi_2,\check\psi_2)$ solutions in $\mathcal Z_0^2$ of this system, satisfying for a constant $ c_2\in \R$
and $v_2$, $w_2\in \mathcal Y$,
\[
\check\varphi_2 = c_2 + v_2, \quad
\check\psi_2 = - c_2 + w_2.
\]
Thus, setting $\varphi_2 = \check\varphi_2-\frac 12 c_1 x^2$, $\psi _2 = \check \psi _2 + \frac 12 c_1 x^2$, the pair $(\varphi_2,\psi_2)$ solves \eqref{ellip2}.
The uniqueness of the solution $(\varphi_2,\psi_2)$ with the specified asymptotic behavior is deduced from part~\eqref{lem1bis:1} of Lemma~\ref{lem1bis}.
\end{proof}

Let $\alpha_2 \in \R$ be defined by (see \eqref{eq:rho})
\begin{equation} \label{def:a2}
\langle \Lambda_5 \varphi_2 - x^2 Q^4 \varphi_2+ \alpha_2 \rho,Q\rangle =0.
\end{equation}

\section{Approximate blow-up solution in rescaled variables}\label{S:3}

In the rest of this paper, $\delta\in (0,1)$ is a constant.
For the sake of clarity, we also introduce the bootstrap constant
\begin{equation}\label{def:K}
K = \delta^{-\frac 12}.
\end{equation}
This constant will be taken
large to close some bootstrap estimates by a continuity argument
in the proof of Proposition \ref{pr:boot}.

Let $s_0>1$ be a constant to be chosen sufficiently large, possibly depending on $K$.
Let $n\geq s_0$ be an integer.
Throughout this paper, except in Section \ref{S:7}, the notation $a \lesssim b$ means that $a \le C b$, where the constant $C>0$ is independent of $K$ and $n$.

\subsection{Rescaled variables}
Define
\begin{equation}\label{def:Tn}
T_n=-\frac 1n,\quad S_n = n.
\end{equation}
We introduce the change of variables
$(s,y)\mapsto (t,x)$ for $s\in [s_0,S_n]$,
\begin{equation} \label{chngvar} 
x = \lambda(s) y,\quad dt = \lambda^2(s) ds, \quad T_n -t = \int^{S_n}_s \lambda^2( \tau ) d\tau ,
\end{equation} 
where $(\gamma, \lambda,b,a)$ are $\mathcal C^1$ functions of $s\in [s_0,S_n]$ 
satisfying the following bootstrap estimates, for any $s\in[s_0,S_n]$
\begin{equation} \label{eq:BS1}
\left\{\begin{aligned}
|\gamma(s) -s |&\leq 2 s^{-\frac 74},\\
|b(s) - s^{-1}|&\leq 2 s^{-\frac 74},\\
\left|\frac{b(s)}{\lambda(s)} - 1\right|&\leq 2 K s^{-1},\\
|a(s)|&\leq (1+|\alpha_1|)  s^{-\frac 52},
\end{aligned}\right.
\end{equation}
where $\alpha _1$ is given by~\eqref{def:a1}.
It follows from \eqref{eq:BS1} that
\begin{equation}\label{eq:lambda}
| \lambda(s) - s^{-1} | \leq 3 s^{-\frac 74},
\end{equation}
for $s_0$ large. 
We also introduce the multiplier
\begin{equation*} 
M_b (s, y) = e^{ i b (s) \frac {y^2} {4} } .
\end{equation*} 
For future use, we also introduce the notation, for $k=1,2,3$,
\begin{equation}\label{def:j}
j_k(s) = \int_{s}^{S_n} \lambda^{k+1}( \tau ) d\tau.
\end{equation}
Using \eqref{eq:lambda}, we estimate, for $s$ large,
\begin{align*}
j_k(s) & = \int_{s}^{S_n} \lambda^{k+1}( \tau ) d\tau\\
&=\int_s^{S_n} \tau^{-k-1} d\tau + \int_s^{S_n} O(\tau^{-\frac{4k+7}4}) d\tau\\
&=\frac 1k \left( s^{-k} - n^{-k}\right)  + (s^{-1} - n^{-1}) O(s^{-\frac {4k-1}4}).
\end{align*}
Thus,
\begin{equation}\label{eq:jk}
\frac1{2k} ( s^{-k} - n^{-k} )  \leq j_k(s) \leq \frac 2k\left( s^{-k} - n^{-k}\right) .
\end{equation}
Note that $ T_n -t(s) = j_1(s)$ and thus
\[
T_n-t(s) = (s^{-1} - n^{-1}) + (s^{-1}-n^{-1}) O(s^{-\frac34}),
\]
so that using $T_n=-n^{-1}$,
\begin{equation}\label{eq:t}
|t(s) + s^{-1} | \lesssim s^{-\frac 74},\quad
s^{-1} \lesssim |t(s)| \lesssim s^{-1}.
\end{equation}

Now, we introduce bootstrap estimates on the time derivative of the parameters.
In order to determine $a$ at a sufficient order of precision, we define the function $\Omega=\Omega(s)$ by
\begin{equation} \label{eq:O}
\Omega=
\alpha_1 b \lambda^\frac32 \cos \gamma + \alpha_2 b\lambda^\frac52\sin \gamma
+ \alpha_3 \lambda^3 \sin 2 \gamma +\alpha_4 b^2 \lambda^\frac32\sin \gamma
\end{equation}
where $\alpha_1$ and $\alpha_2$ are defined by \eqref{def:a1}, \eqref{def:a2} and
\begin{equation}\label{def:a3a4}
\alpha_3 \int \rho Q = -2 \int Q^4 \varphi_1\psi_1,\quad
\alpha_4 \int \rho Q = -\frac 14 \int y^4 Q^5 \psi_1.
\end{equation}
Let
\begin{equation} \label{def:m} 
m_\gamma = \gamma_s - 1,\quad m_\lambda = \frac{\lambda_s}{\lambda}+b,
\quad m_b = b_s+b^2-a,\quad
m_a = a_s - \Omega,
\end{equation}
and
\begin{equation}\label{def:vecm}
\vec m = \begin{pmatrix} m_\gamma  \\ m_\lambda \\ m_b \\ m_a\end{pmatrix}.
\end{equation}
We impose the following bootstrap estimates
\begin{equation} \label{eq:BS2}
|\vec m(s)| \leq K s^{-3}.
\end{equation}

\begin{remark}\label{rk:omega}
The relations $m_\gamma=0$, $m_\lambda=0$ and $m_b=0$ give the standard
approximate equations for the parameters $\gamma$, $\lambda$ and $b$ for the conformal blow-up.
For example, for the minimal mass blow-up solution \eqref{def:S}, one has
\[
\gamma(s) = s,\quad \lambda(s) = \frac 1s,\quad b(s) = \frac 1s,\quad a(s)=0.
\]
 The function $\Omega$
 allows us to keep track of some oscillatory terms in the behavior of $a_s$ that cannot be estimated in absolute value, but that can be easily integrated in time.
This leads to a refined estimate of $a$ (see the last line of \eqref{eq:BS1}), 
which we use in the equation of $b$ (see the definition of $m_b$).
\end{remark}

\subsection{Definition of an approximate rescaled blow-up solution}\label{S:3:2}

First, let $\Theta_0:\R\to [0,1]$ be an even smooth function with compact support such that
\begin{equation}\label{def:Theta0}
\Theta_0(x)= \begin{cases}
1 & \mbox{if $|x|<1$}\\
0 & \mbox{if $|x|>2$}
\end{cases}
\end{equation}
and $C_0=\|\Theta_0\|_{L^2}^2>0$.
Define
\[
\Theta(x) = \Theta_0(x/\delta)
\]
so that
\begin{equation*}
\Theta(x)= \begin{cases}
 1 & \mbox{if $|x|<\delta$}\\
 0 & \mbox{if $|x|>2\delta$}
\end{cases}
\end{equation*}
and
\begin{equation*}
\|\Theta\|_{L^2}^2 = C_0 \delta .
\end{equation*}

We define 
\begin{equation} \label{eq:theta0}
\theta[\Gamma](y)= \Theta ( \lambda y).
\end{equation}
where
\[
\Gamma=(\gamma,\lambda,b,a,j_1,j_2,j_3)
\in \R^7 \cap \{\lambda >0\}.\]
When $\Gamma$ is a function of $s$, we will write by abuse of notation
\begin{equation} \label{eq:theta}
\theta(s,y)= \theta[\Gamma(s)](y)= \Theta ( \lambda (s) y).
\end{equation}

Now, we introduce
\begin{equation*}
W[\Gamma](y) = M_{ -b } Q_{a }(y) + \theta[\Gamma](y) (A[\Gamma](y)+ i B[\Gamma](y)) 
\end{equation*}
where
\begin{align*}
A[\Gamma](y) &= A_1[\Gamma](y)+A_2[\Gamma](y)= \lambda^\frac32 \varphi_1(y)\cos\gamma 
+ \lambda^\frac52  \varphi_2(y)\sin\gamma, \\
B[\Gamma](y) &= B_1[\Gamma](y)+B_2[\Gamma](y)=\lambda^\frac32 \psi_1(y) \sin\gamma 
- \lambda^\frac52  \psi_2(y) \cos\gamma.
\end{align*} 
Recall that the functions $\varphi_1, \psi_1\in \mathcal Z_1$ are defined in Lemma~\ref{lem1bis} and the functions
$\varphi_2,\psi_2\in \mathcal Z_2$ are defined in Lemma~\ref{lem2}.

When $\Gamma$ is a function of $s$, we write by abuse of notation
\begin{equation*}
W(s,y) = W[\Gamma(s)](y) =M_{ -b(s)} Q_{a(s)}(y) + \theta(s,y) (A(s,y)+ i B(s,y)) 
\end{equation*}
where
\begin{align*}
A(s,y) &= A_1(s,y)+A_2(s,y)= \lambda^\frac32(s) \varphi_1(y)\cos\gamma(s) 
+ \lambda^\frac52(s) \varphi_2(y)\sin\gamma(s), \\
B(s,y) &= B_1(s,y)+B_2(s,y)=\lambda^\frac32(s) \psi_1(y) \sin\gamma (s)
- \lambda^\frac52(s) \psi_2(y) \cos\gamma(s).
\end{align*} 

\begin{remark}
The function $W$ is the main part of the approximate solution. It contains the modified bubble $Q_a$, multiplied by the quadratic phase $M_{-b}$. It also contains the term $\theta (A+iB)$ which will generate the expected blow-up profile at $x=0$. That term is not multiplied by $M_{-b} $ which is consistent with the result of convergence of the residual profile in \cite{MerleR1}.
\end{remark}

We also introduce
\begin{align*} 
Z[\Gamma](y) = \lambda^\frac12 e^{-i\gamma } &\Bigl[ - i j_1 (\nu_1+\nu_2)(\lambda y) 
- \frac {1} {2} j_1^2 (\nu_1'' +\nu_2'' )(\lambda  y) \\
&\quad - ij_2 c_1 \Theta '' (\lambda  y) 
+ j_3 c_2 \Theta '' (\lambda y) \Bigr]
\end{align*} 
where the functions $\nu_1$ and $\nu_2$ are explicitly defined as
\begin{align*}
\nu_1(x) &= \Theta''(x) (|x| + i \kappa x^2) +2 \Theta'(x) (\SIGN(x) + 2i\kappa x),\\
\nu_2(x) &= f( \Theta (x)(|x|+ i \kappa x^2) ).
\end{align*}
Again, when $\Gamma$ is a function of $s$, we write by abuse of notation
\begin{equation*}
Z(s,y) = Z[\Gamma(s)](y).
\end{equation*}

\begin{remark}
The function $Z$ is a correction term in the approximate solution, to take into account error terms generated by $W$, as the nonlinear term $\nu_2$ for example.

Note that the expression of this correction term is simpler in the variable $(t,x)$. Setting
\[
z(t,x) = \lambda^{-\frac 12}(s) e^{i\gamma(s)} Z (s, y ) ,
\]
one has
\begin{equation} \label{def:z}
\begin{aligned} 
z(t,x) &= - i (T_n-t) ( \nu_1(x) + \nu_2(x)) - \frac {1} {2} (T_n-t)^2 (\nu_1 '' (x) + \nu_2 '' (x) ) \\
& \quad- i \Bigl( \int _{s(t)} ^{S_n} \lambda ^3 \Bigr) c_1 \Theta '' ( x ) + \Bigl( \int _{s(t)} ^{S_n} \lambda ^4 \Bigr) c_2 \Theta '' ( x ),
\end{aligned} 
\end{equation} 
where following \eqref{chngvar}, $s$ is the function of $t$ defined by
\begin{equation*} 
T_n-t = \int _{ s(t) }^{S_n} \lambda ^2 = j_1(s(t))
\end{equation*} 
(so that $s'(t) \lambda ^2(s(t))= 1$).
\end{remark}

Let $V$ be the approximate rescaled blow-up solution defined by
\begin{equation} \label{eq:blowup}
V[\Gamma](y) = W[\Gamma](y)+Z[\Gamma](y).
\end{equation}
It is immediate by the definitions of $W$ and $Z$ that $V$ is a locally Lipschitz function of $\Gamma$ with values in $L^2_y$.
The exact expressions of $W$ and $Z$ will be justified by the computations in the proof of Proposition~\ref{PR:1}.
As above, the notation 
\begin{equation*}
V(s,y) = W(s,y)+Z(s,y)
\end{equation*}
will be used when $\Gamma$ is a function of $s$.

We denote 
\begin{equation} \label{interval} 
I(s) = [-2\delta\lambda^{-1},2\delta\lambda^{-1}],\quad
J(s) = [-2\delta\lambda^{-1},-\delta\lambda^{-1}]\cup[\delta\lambda^{-1},2\delta\lambda^{-1}].
\end{equation} 
We now establish some pointwise bounds on $V$, $W$ and $Z$.

\begin{lemma}\label{le:V}
Under the bootstrap estimates \eqref{eq:BS1},
\begin{align}
| A|+ | B| & \lesssim s^{-\frac 32} (1+|y|) + s^{-\frac 52} (1+y^2),
\label{festW00}\\
|\theta A|+ |\theta B| & \lesssim s^{-\frac 32} (1+|y|) \mathbf{1}_{I(s)}
\lesssim \delta s^{-\frac 12} \mathbf{1}_{I(s)},
\label{festW0}\\
|\partial_y (\theta A)|+ |\partial_y(\theta B)| & \lesssim s^{-\frac 32}  \mathbf{1}_{I(s)},
\label{festdW0}\\
| W | + (1 + |y|) | \partial _y W| &\lesssim e^{- \frac { |y|} {2}} + \delta s^{-\frac 12} \mathbf{1}_{I(s)} ,
\label{festW1}\\
|Z| + (1 + |y|) |\partial _y Z| &\lesssim \delta^{-1} s^{-\frac 32} \mathbf{1}_{J(s)} + s^{-\frac{9}2} (|y|^3 + 1) \mathbf{1}_{I(s)},\label{festZ5} \\
| V-M_{-b} Q|+|V-Q|+||V|-Q| & \lesssim \delta s^{-\frac 12},\label{fxye1e2b}\\
|V| + (1 + |y|) |\partial _y V| &\lesssim e^{- \frac { |y|} {2}} + \delta s^{-\frac 12} \mathbf{1}_{I(s)} ,
\label{festV1}
\end{align}
Assuming in addition \eqref{eq:BS2}, it holds
\begin{equation}
 | \partial _s V | \lesssim s^{- \frac {3} {2}} (1+ |y| {\mathbf 1}_{I(s)} ) + \delta^{-1} s^{-\frac 32} \mathbf{1}_{J(s)} .
\label{dsV}  
\end{equation}
\end{lemma}
\begin{proof}
Since $\Theta (x) = \Theta _0 (x/ \delta )$, we see that 
\begin{equation} \label{fEstThera1} 
 |\Theta (x) | \lesssim \mathbf{1} _{ \{ |x|< 2\delta \} }, \quad x\in \R, 
\end{equation} 
and, for $ j\ge 1 $,
\begin{equation} \label{fEstThera2}
 |\Theta ^{(j)} (x) | \lesssim \delta ^{- j} \mathbf{1} _{ \{ \delta < |x|< 2\delta \} }, \quad x\in \R. 
\end{equation}
Thus, for $ j\ge 1 $, by \eqref{eq:BS1},
\[
|\theta|\lesssim \mathbf{1}_{I(s)},\quad
|\partial_y^j \theta|\lesssim \delta^{-j}s^{-j}\mathbf{1}_{J(s)}.
\]
Moreover, we observe for future reference that
\begin{equation} \label{festMbT} 
Q | M_b \theta -1| \lesssim Q |1- \theta | + Q | M_b -1| \lesssim e^{- \frac {1} {2} \frac {\delta } {\lambda }} + |b| \lesssim s^{-1}. 
\end{equation} 
Estimate \eqref{festW00} follows directly from the bootstrap assumptions \eqref{eq:BS1} and the properties of the functions
$(\varphi_1,\psi_1)$, $(\varphi_2,\psi_2)$ given in Lemmas \ref{lem1bis}, \ref{lem2}.
Thus,
\begin{align*}
|\theta A|+|\theta B| 
&\lesssim \left[ s^{-\frac 32} (1+|y|) + s^{-\frac 52} (1+y^2) \right] \mathbf{1}_{I(s)}\\
& \lesssim s^{-\frac 32} (1+|y|) \mathbf{1}_{I(s)}
\lesssim \delta s^{-\frac 12} \mathbf{1}_{I(s)},
\end{align*}
which is \eqref{festW0}.
We also check that
\[
|\partial_y(\theta A)|+|\partial_y(\theta B)|
\lesssim \left[ s^{-\frac 32} + \delta^{-1}s^{-\frac 52} (1+|y|)
+ \delta^{-1}s^{-\frac 72} (1+y^2)\right] \mathbf{1}_{I(s)}
\lesssim s^{-\frac 32}\mathbf{1}_{I(s)},
\]
which is \eqref{festdW0}.
Now, the pointwise estimate \eqref{festW1} for $W$ follows directly from the expression of $W$, the definition of $Q$ in \eqref{def:Q} and \eqref{festW0}. Besides,  \eqref{festdW0} implies the pointwise estimate \eqref{festW1} for $(1+|y|) |\partial_y W|$.

Using~\eqref{fEstThera1}-\eqref{fEstThera2}, it is not difficult to estimate from \eqref{def:z}
and \eqref{eq:BS1},
\begin{equation} \label{fEstz1} 
\begin{aligned} 
 | z(t,x) | \lesssim & (T_n-t) [ \delta ^{-1} \mathbf{1} _{ \{ \delta < |x|< 2\delta \} }+ |x| ^5 \mathbf{1} _{ \{ |x|< 2\delta \} } ] \\
 & + (T_n-t)^2 [ \delta ^{-3} \mathbf{1} _{ \{ \delta < |x|< 2\delta \} }+ |x| ^3 \mathbf{1} _{ \{ |x|< 2\delta \} } ] ,
\end{aligned} 
\end{equation} 
and
\begin{equation} \label{fEstz2} 
\begin{aligned} 
 | \partial _x z(t,x) | \lesssim & (T_n-t) [ \delta ^{-2} \mathbf{1} _{ \{ \delta < |x|< 2\delta \} }+ |x| ^4 \mathbf{1} _{ \{ |x|< 2\delta \} } ] \\
 & + (T_n-t)^2 [ \delta ^{-4} \mathbf{1} _{ \{ \delta < |x|< 2\delta \} }+ |x| ^2 \mathbf{1} _{ \{ |x|< 2\delta \} } ] .
\end{aligned} 
\end{equation} 
In particular,
\begin{equation} \label{fEstz3} 
\begin{aligned} 
 | z(t,x) | \lesssim & (T_n-t) [ \delta ^{-1} \mathbf{1} _{ \{ \delta < |x|< 2\delta \} }+ \delta ^5 \mathbf{1} _{ \{ |x|< 2\delta \} } ] \\
 & + (T_n-t)^2 [ \delta ^{-3} \mathbf{1} _{ \{ \delta < |x|< 2\delta \} }+ \delta ^3 \mathbf{1} _{ \{ |x|< 2\delta \} } ] ,
\end{aligned} 
\end{equation} 
and
\begin{equation} \label{fEstz4} 
\begin{aligned} 
 | \partial _x z(t,x) | \lesssim & (T_n-t) [ \delta ^{-2} \mathbf{1} _{ \{ \delta < |x|< 2\delta \} }+ \delta ^4 \mathbf{1} _{ \{ |x|< 2\delta \} } ] \\
 & + (T_n-t)^2 [ \delta ^{-4} \mathbf{1} _{ \{ \delta < |x|< 2\delta \} }+ \delta ^2 \mathbf{1} _{ \{ |x|< 2\delta \} } ] .
\end{aligned} 
\end{equation} 
Using~\eqref{fEstz1}-\eqref{fEstz2}, we obtain
\begin{align*} 
 | Z(s,y) | & \lesssim s^{- \frac {3} {2}} [ \delta ^{-1} \mathbf{1} _{J(s) }+ s^{-5} |y| ^5 \mathbf{1} _{I(s) } ] + s^{- \frac {5} {2}} [ \delta ^{-3} \mathbf{1} _{ J(s) }+ s^{-3} |y| ^3 \mathbf{1} _{ I(s)} ] \\
 & \lesssim s^{- \frac {3} {2}} \delta ^{-1} \mathbf{1} _{J(s) } + s^{-\frac {9} {2}} |y| ^3 \mathbf{1} _{I(s) } ,
\end{align*} 
and
\begin{equation*} 
 \begin{aligned} 
 | \partial _yZ(s,y) | &\lesssim s^{- \frac {5} {2}} [ \delta ^{-2} \mathbf{1} _{J(s) }+ s^{-4} |y| ^4 \mathbf{1} _{I(s) } ] + s^{- \frac {7} {2}} [ \delta ^{-4} \mathbf{1} _{ J(s) }+ s^{-2} |y| ^2 \mathbf{1} _{ I(s)} ] \\
 &\lesssim s^{- \frac {5} {2}} \delta ^{-2} \mathbf{1} _{J(s) }+ s^{-\frac {9} {2}} |y| ^2 \mathbf{1} _{I(s) } .
 \end{aligned} 
\end{equation*} 
This proves \eqref{festZ5} for $s$ large.
We obtain from~\eqref{festW0} and~\eqref{festZ5} that
\[| V-M_{-b} Q_a|\leq \delta s^{-\frac12},\]
and \eqref{fxye1e2b} follows from
\[
|Q_a- Q|+ |M_{-b} Q -  Q|\lesssim |a|+|b|\lesssim s^{-1}.
\]
By~\eqref{festW1} and~\eqref{festZ5}, we obtain \eqref{festV1}.

Finally, we prove \eqref{dsV}. By a direct computation,
\begin{equation*} 
 \partial_s W= Q_a \partial _s M _{ -b }+ a_s M _{ -b }\rho + (A+iB) \partial _s \theta + \theta (\partial _s A + i \partial _s B). 
\end{equation*} 
Using~\eqref{fusef1}, it follows from elementary calculations and the estimates $ |a_s |\lesssim s^{-\frac {5} {2}}$, $ |b_s| \lesssim s^{-2}$, $ |\frac {\lambda _s} {\lambda }| \lesssim s^{-1}$ 
and $|\partial_s \theta|\lesssim s^{-1}\mathbf{1}_{I(s)}$ (by assumptions~\eqref{eq:BS1} and \eqref{eq:BS2}) that 
\begin{equation*} 
 | Q_a \partial _s M _{ -b }+ a_s M _{ -b }\rho + (A+iB) \partial _s \theta | \lesssim s^{-\frac {3} {2}}. 
\end{equation*} 
The terms in the expansion of $\theta (\partial _s A + i \partial _s B)$ can all be estimated by similar arguments, so we give the details for one term only, $\theta \partial _s A_1$. We have
\begin{equation*} 
\theta \partial _s A_1 = \theta \Bigl( \frac {3} {2} \frac {\lambda _s} {\lambda } A_1 - \gamma _s \lambda ^{\frac {3} {2}} \varphi _1 \sin \gamma \Bigr) .
\end{equation*} 
Since $ |\theta A_1| \lesssim s^{-\frac {1} {2}}$, we see that 
\begin{equation*} 
 \Bigl| \frac {\lambda _s} {\lambda } A_1 \Bigr| \lesssim s^{- \frac {3} {2}}.
\end{equation*} 
Moreover, $ |\gamma _s| \lesssim 1$ and $ |\theta \varphi _1| \lesssim (1+ |y|) {\mathbf 1}_{I(s)}$, so that
\begin{equation*} 
 | \gamma _s \lambda ^{\frac {3} {2}} \theta \varphi _1 \sin \gamma |\lesssim s^{- \frac {3} {2}} (1+ |y|) {\mathbf 1}_{I(s)} ;
\end{equation*} 
and so
\begin{equation*} 
 | \partial _s W| \lesssim s^{- \frac {3} {2}} (1+ |y| {\mathbf 1}_{I(s)} ). 
\end{equation*} 
We now estimate $\partial _s Z$. Writing $\partial _s Z= - i \gamma _s Z + \widetilde{Z} $, where
\begin{align*}
\widetilde{Z} & = \frac{\lambda_s}{\lambda} \Lambda Z 
+  \lambda^\frac12 e^{-i\gamma } \Bigl[ i  \lambda^2  (\nu_1+\nu_2)(\lambda   y) \\
& +  \left( \int_s^{S_n} \lambda^2 \right) \lambda^2 (\nu_1'' +\nu_2'' )(\lambda   y) 
 + i \lambda ^3   c_1 \Theta '' (\lambda y) 
+   \lambda ^4  c_2 \Theta '' (\lambda y) \Bigr] .
\end{align*}
By~\eqref{festZ5}, 
\begin{equation*} 
 | \gamma _s Z | \lesssim |Z| \lesssim \delta^{-1} s^{-\frac 32} \mathbf{1}_{J(s)} + s^{-\frac 32}.
\end{equation*} 
Moreover, similarly as in the proof of \eqref{festZ5}, we see that $ | \widetilde{Z} |\lesssim  \delta ^{-1} s^{-\frac {5} {2}} \lesssim s^{-\frac {3} {2}}$. 
Thus,
\[
|\partial_s Z|\lesssim  \delta^{-1} s^{-\frac 32} \mathbf{1}_{J(s)} + s^{-\frac 32}.
\]
By summing the estimates of $\partial_s W$ and $\partial_s Z$, we obtain \eqref{dsV}.
\end{proof}

\subsection{Equation of the approximate rescaled blow-up solution}

The rest of this section is devoted to the proof of the following result.
Recall that the notation $\mathcal E$ is defined by \eqref{def:EE}.

\begin{proposition} \label{PR:1}
Under the bootstrap estimates \eqref{eq:BS1},
\begin{equation} \label{eq:V}
\mathcal E( V )
= \mathcal S_0 +\mathcal R
\end{equation}
where 
\begin{equation*} 
\mathcal S_0=\vec m \cdot \vec \rho_0, \quad    \vec \rho_0 =  M_{ -b }\begin{pmatrix}- Q_a \\ - i \Lambda Q_a - \tfrac b2{y^2} Q_a\\
\tfrac14 {y^2} Q_a \\ i \rho 
\end{pmatrix}
\end{equation*} 
and
\[
\mathcal R =  \vec m \cdot \vec \rho_m + \mathcal R_\Gamma,
\]
where
$\vec \rho_m=\vec \rho_m[\Gamma] \in (L^2)^4$ and $\mathcal R_\Gamma=\mathcal R_\Gamma[\Gamma] \in L^2$ are locally Lipschitz functions of $\Gamma$ with values in $L^2_y$ and 
\begin{equation}\label{eq:etoile}
\|\vec\rho_m[\Gamma]\|_{L^2} + \| \mathcal R_\Gamma[\Gamma]\|_{L^2} \to 0 \quad \mbox{as $(\lambda,b,a,j_1,j_2,j_3) \to 0$.}
\end{equation}
Moreover, assuming further \eqref{eq:BS2}, 
the error term $\mathcal R[\Gamma(s)](y)=\mathcal R(s,y)$ satisfies
\begin{gather}
\|\mathcal R(s)\|_{H^1} \lesssim s^{-3}, \label{eq:R1}\\
\| y \mathcal R(s)\|_{L^2} \lesssim s^{-2}, \label{eq:R1b}\\
| \langle M_{b(s)} \mathcal R(s) , iQ \rangle|\lesssim s^{-4}. \label{eq:R2}
\end{gather}
\end{proposition}

\subsection{Contribution of the main part of the approximate solution}

Using the equation \eqref{eq:Qa} of $Q_a$ and the definition of $\Omega$ in \eqref{eq:O}, we 
decompose $\mathcal E(W)$ as
\begin{equation} \label{eq:W}
\mathcal E(W) 
= \mathcal S_0 +\mathcal S_1 + \mathcal S_2 + \mathcal R_3 + \mathcal R_4+ \mathcal R_5.
\end{equation}
where
\begin{align*} 
\mathcal S_0 & =M_{ -b } \left( - m_\gamma Q_a + (m_b -2bm_\lambda)\tfrac14{y^2} Q_a + i m_a \rho  - i m_\lambda \Lambda Q_a \right),\\
\mathcal S_1 & = \theta_{ yy } A + 2 \theta_y A_y 
+i \bigl( \theta_{ yy } B + 2 \theta_y B_y \bigr), \\
\mathcal S_2 & = f(W)-f(M_{-b}Q_a) - 5 Q^4 \theta A +by^2 Q^4 \theta B - i Q^4 \theta B +iby^2 Q^4 \theta A\\
&\quad +i \theta \big( \alpha_3 \lambda^3 \sin 2 \gamma +\alpha_4 b^2 \lambda^\frac32\sin \gamma\big) \rho,\\
{\mathcal R}_3& =\theta (G+iH),\\
{\mathcal R}_4 & = M_{-b} \Bigl(Q_a^5-Q^5-5aQ^4\rho + a^2 \frac {y^2} {4} \rho \Bigr),\\
{\mathcal R}_5 & = i(M_{-b}-\theta)\Omega \rho,
\end{align*}
and 
\begin{align*}
G & = - \partial_s B -L_+ A
-b(\Lambda B + y^2 Q^4 B)- m_\gamma A +m_\lambda \Lambda B , \\
H & = \partial_s A - L_- B +b (\Lambda A -y^2 Q^4 A)-m_\gamma B -m_\lambda \Lambda A\\
&\quad +\bigl(\alpha_1 b \lambda^\frac32 \cos \gamma + \alpha_2 b\lambda^\frac52\sin \gamma\bigr)\rho.
\end{align*}
We refer to Appendix~\ref{appendixB} for the proof of \eqref{eq:W}.

\begin{remark}
In the expression \eqref{eq:W}, 
errors terms designated by $\mathcal R_k$ will be estimated directly in \S\ref{S:3.7}, while
errors terms designated by $\mathcal S_k$ will be corrected by some terms from the equation of $Z$,
before being estimated.

Note also that the contribution of $i\Omega \rho$ from the definition of $m_a$ has been split between $\mathcal S_2$ and $\mathcal R_3$ (with an error term $\mathcal R_5$ due to the quadratic phase and the localization $\theta$).
\end{remark}

\subsection{Contribution of the correction part of the approximate solution}
To write the equation of $Z$, it is simpler to derive first the equation of $z$ in the original variables $(t,x)$.
Indeed, by \eqref{def:z}, we have
\begin{equation*} 
\begin{aligned} 
i \partial_t z + \partial_x^2 z = & -\nu_1-\nu_2 - \frac {1} {2} (T_n-t)^2 (\nu_1^{(4)}+ \nu_2^{(4)}) \\
& - \lambda (s(t)) c_1 \Theta '' - i \Bigl( \int _{s(t)} ^{S_n} \lambda ^3 \Bigr) c_1 \Theta ^{(4)} \\
& -i \lambda ^2 (s(t)) c_2 \Theta '' + \Bigl( \int _{s(t)} ^{S_n} \lambda ^4 \Bigr) c_2 \Theta ^{(4)}.
\end{aligned} 
\end{equation*} 
Thus, by \eqref{eq:resc1}, one has
\begin{equation}\label{eq:Z}
\mathcal E(Z) = \mathcal T_1+\mathcal T_2+ \mathcal R_6 + \mathcal R_7,
\end{equation}
where
\begin{align*}
\mathcal T_1 &= -\lambda^\frac52 e^{-i\gamma} [ \nu_1(\lambda y) + \lambda c_1 \Theta '' (\lambda y) 
+ i \lambda ^2 c_2 \Theta '' ( \lambda y) ] , \\
\mathcal T_2 &= -\lambda^\frac52 e^{-i\gamma} \nu_2(\lambda y), \\
\mathcal R_6 &= \lambda^\frac52 e^{-i\gamma} \Bigl[ - \frac {1} {2} \Bigl( \int_s^{S_n} \lambda^2 \Bigr)^2 (\nu_1^{(4)} (\lambda y)+\nu_2^{(4)} (\lambda y)) \\
& \qquad+ \Bigl( \int _s ^{S_n} -i c_1 \lambda ^3 + c_2 \lambda ^4 \Bigr) \Theta ^{(4)} (\lambda y) \Bigr] ,\\
\mathcal R_7 &= f(Z).
\end{align*}

\subsection{Equation of the approximate rescaled blow-up solution}
We set 
\[
 \mathcal R_8 = f(W+Z)-f(W)-f(Z),
\]
so that, since $V= W+ Z$,
\[
\mathcal E(V) = \mathcal E(W)+\mathcal E(Z)+ \mathcal R_8.
\]
Setting
\[
\mathcal R_1=\mathcal S_1 + \mathcal T_1,\quad
\mathcal R_2=\mathcal S_2 + \mathcal T_2
\]
and
\[
\mathcal R= \sum_{k=1}^8 \mathcal R_k,
\]
from \eqref{eq:W} and \eqref{eq:Z}, we have obtained the expected form \eqref{eq:V} for $\mathcal E(V)$.
To complete the proof of Proposition~\ref{PR:1}, it remains to prove the estimates~\eqref{eq:R1}, \eqref{eq:R1b} and~\eqref{eq:R2} on each of the error term
$\mathcal R_k$, $k=1,\ldots, 8$. We also need to verify the structure and the estimates in \eqref{eq:etoile}.
We state and prove the following technical estimates to be used in the next subsection.

\begin{lemma} \label{LE:3.2}
With $\theta $ defined by~\eqref{eq:theta}, $\mathcal Z_k$ defined by~\eqref{fdefzk} and under the assumption~\eqref{eq:BS1}, the following properties hold. 
\begin{enumerate}[{\rm (i)}] 

\item \label{LE:3.2:1} 
For $R\in \mathcal Y$ and $s$ large,
\begin{align*}
\| \theta R \|_{L^2} & \lesssim 1,\\
\| (\partial_y^j \theta) R\|_{L^2} & \lesssim e^{- \sqrt s} \quad \mbox{if $j\geq 1$.}
\end{align*}

\item \label{LE:3.2:2} 
Let $j,k\in \N$. For any $R \in \mathcal Z_k$, for $s$ large,
\begin{align*}
\| (\partial_y^j \theta) R \|_{L^2} &\lesssim \delta^{\frac 12+k-j} s^{\frac12+k-j},\\
\| \partial_y( (\partial_y^j \theta) R) \|_{L^2} &\lesssim \delta^{-\frac 12+k - j} s^{-\frac12+k-j} \quad \mbox{if $j+k\geq 1$,}\\
\| \partial_y(\theta R) \|_{L^2} &\lesssim 1 \quad \mbox{if $k=0$.}
\end{align*}
In particular,
\begin{equation*} 
\| (\partial_y^j \theta) R \|_{H^1} \lesssim \delta^{\frac 12+k-j} s^{\frac12+k-j},
\end{equation*} 
for all $j, k \ge 0$ and $R \in \mathcal Z_k$.
\end{enumerate} 
\end{lemma}

\begin{proof}
Proof of~\eqref{LE:3.2:1}. 
Let $R\in \mathcal Y$, in particular $R^2(y) \lesssim e^{- \frac {3} {2} |y|}$. The first inequality is immediate since $0\le \theta \le 1$. If $j\geq 1$, then with the notation~\eqref{interval} 
\begin{equation*}
\int |\partial_y^j \theta|^2 R^2(y) dy 
\lesssim \lambda^{2j} \delta ^{- 2j} \int_{J(s) } e^{- \frac {3} {2} |y|} dy 
\lesssim e^{-\frac 32 \frac\delta\lambda} \lesssim e^{- 2 \sqrt s},
\end{equation*}
where we used $\lambda \le \delta $ for $s$ large, and $\lambda (s) \le \frac {3 \delta } {4 \sqrt s} $ for $s$ large (by~\eqref{eq:lambda}).

Proof of~\eqref{LE:3.2:2}. 
Let $j, k \ge 0$ and $R\in \mathcal Z_k$, in particular $R^2(y) \lesssim \langle y\rangle ^{2k}$. 
We calculate
\begin{equation*}
\int |\partial_y^j \theta|^2 R^2(y) dy \lesssim \lambda^{2j} \delta ^{- 2j} \int_{I(s) } \langle y\rangle ^{2k} dy 
\lesssim \lambda^{2j} \delta ^{- 2j} \Bigl( \frac {\delta } {\lambda } \Bigr)^{1+ 2k } \int_{-2 }^2 \langle x\rangle ^{2k} dx 
\end{equation*}
where we have used the change of variable $x= \frac {\lambda } {\delta } y$ and $\frac \lambda\delta\leq 1$ for $s$ large.
The first estimate then follows from~\eqref{eq:BS1}.
The remaining two estimates follow easily from the first one and from part~\eqref{LE:3.2:1}. 
\end{proof}

\subsection{Estimates of the error terms}\label{S:3.7}

\emph{Estimates of $\mathcal R_1$.}
We develop using the expression of $\nu_1$
\begin{align*}
\mathcal T_1 
&= -\lambda^\frac52 e^{-i\gamma} [ \nu_1(\lambda y) + \lambda c_1 \Theta '' (\lambda y) 
+ i \lambda ^2 c_2 \Theta '' ( \lambda y) ] \\
&= - \theta_{yy} \Bigl[ \lambda^\frac32 |y| \cos \gamma + \kappa \lambda^\frac52 y^2 \sin\gamma
- i \lambda^\frac32 |y| \sin \gamma + i \kappa \lambda^\frac52 y^2 \cos\gamma \\
& \qquad + \lambda^\frac32 c_1 \cos \gamma + \lambda^\frac52 c_2 \sin \gamma 
- i \lambda^\frac32 c_1 \sin \gamma +i \lambda^\frac52 c_2 \cos \gamma \Bigr]\\
&\quad - 2 \theta_y \left[\lambda^\frac32 (\SIGN y) \cos\gamma + 2 \kappa \lambda^\frac52 y\sin\gamma
- i \lambda^\frac32(\SIGN y)\sin\gamma + 2i\kappa\lambda^\frac52 y\cos\gamma\right].
\end{align*}
Thus,
\begin{align*}
\mathcal R_1
 = \mathcal S_1 + \mathcal T_1
& =\theta_{yy} \left[ \lambda^\frac32 (\varphi_1-|y| - c_1 ) \cos \gamma + \lambda^\frac52 (\varphi_2-\kappa y^2 - c_2) \sin\gamma\right]\\
&\quad +i\theta_{yy}\left[\lambda^\frac32 (\psi_1+|y| + c_1 ) \sin \gamma - \lambda^\frac52 (\psi_2+\kappa y^2 + c_2) \cos\gamma\right]\\
&\quad + 2 \theta_y \left[\lambda^\frac32(\varphi_1'-\SIGN y)\cos\gamma + \lambda^\frac52(\varphi_2'-2\kappa y)\sin\gamma\right]\\
&\quad + 2i \theta_y\left[\lambda^\frac32(\psi_1'+\SIGN y)\sin\gamma -\lambda^\frac52 (\psi_2'+2\kappa y)\cos\gamma\right].
\end{align*}
Observe that $\theta_y$ and $\theta_{yy}$ vanish outside $J (s) $, hence on $[-1, 1]$ for $s$ large.
In particular, $ \mathcal R_1 = \chi ( y) \mathcal R_1$. 
From the asymptotics in Lemmas~\ref{lem1bis} and~\ref{lem2}, it follows that the functions
\begin{gather*}
\chi (\varphi_1-|y| -c_1) ,\quad \chi (\psi_1+|y| + c_1) ,\quad \chi ( \varphi_2-\kappa y^2- c_2) ,\quad \chi (\psi_2+\kappa y^2+ c_2),\\
\chi (\varphi_1'-\SIGN y) , \quad \chi (\psi_1'+\SIGN y), \quad \chi (\varphi_2'-2\kappa y) ,\quad \chi (\psi_2'+2\kappa y) 
\end{gather*}
all belong to $\mathcal Y$. Therefore, it follows from Lemma~\ref{LE:3.2}~\eqref{LE:3.2:1} that 
\[
\|\mathcal R_1\|_{H^1}+\| y \mathcal R_1 \| _{ L^2 } +|\langle M_b \mathcal R_1,iQ\rangle|\lesssim e^{- \sqrt s},
\]
for $s$ large. Besides, it is immediate by its expression above that $\mathcal R_1[\Gamma] \in L^2$ is a locally Lipschitz function of $\Gamma$ with values in $L^2_y$ and 
$\| \mathcal R_1[\Gamma]\|_{L^2} \to 0$ as $\lambda\to 0$.

\emph{Estimates of ${\mathcal R}_2$.}
We decompose $\mathcal R_2 =\mathcal S_2 + \mathcal T_2=\mathcal R_{2,1}+\mathcal R_{2,2}$,
where
\begin{align*}
\mathcal R_{2,1} &= \mathcal S_2 - f(\theta (A+iB)\\
\mathcal R_{2,2} &= f(\theta (A+iB) + \mathcal T_2.
\end{align*}
Given $X, Y\in \R$, we expand $f(Q_a+X+iY) - f(Q_a)$. Since
\begin{align*}
|Q_a+ X+ iY |^4 & = \left[ (Q_a+X )^2 + Y^2 \right]^2 \\
& = (Q_a+X)^4 + 2 (Q_a+X)^2 Y^2 + Y^4\\
& = 4 Q_a^3 X + 6 Q_a^2 X^2 + 4 Q_a X^3 + 2 Q_a^2 Y^2 + 4 Q_a XY^2 \\
&\quad + Q_a^4 + |X+iY|^4,
\end{align*}
we have the general identity
\begin{equation} \label{eq:QXY}
\begin{aligned} 
f(Q_a& +X+iY) = 5 Q_a^4 X + i Q_a^4 Y \\ 
&\quad+ 10 Q_a^3 X^2 + 2 Q_a^3 Y^2 + 4i Q_a^3 XY \\
&\quad+ 10 Q_a^2 X^3 + 2i Q_a^2 Y^3 + 6i Q_a^2 X^2Y + 6 Q_a^2 XY^2 \\
&\quad+5 Q_a X^4 + 4i Q_a X^3Y + 6 Q_a X^2Y^2 + 4i Q_a XY^3 + Q_a Y^4 \\
&\quad+f(Q_a) + f(X+iY).
\end{aligned}
\end{equation} 
We apply the above identity to the real-valued functions $X,Y$ such that
\[
M_b \theta(A+iB) = X+iY,
\]
i.e.
\begin{equation}\label{eq:XY}
X = \Re (M_b\theta(A+iB)),\quad Y = \Im(M_b\theta(A+iB)),
\end{equation}
and so,
\begin{equation} \label{fmbw} 
M_b W= Q_a + X + i Y.
\end{equation} 
Thus, 
\begin{align*}
f(W)
& = M_{-b} \left[ f(Q_a+X+iY) - f(Q_a) \right]+ f(M_{-b} Q_a)\\
& =M_{-b} (5Q_a^4 X+i Q_a^4 Y) \\ 
&\quad + M_{-b} \left(10 Q_a^3 X^2 + 2 Q_a^3 Y^2 + 4i Q_a^3 XY \right)\\
&\quad + M_{-b} \left( 10 Q_a^2 X^3 + 2i Q_a^2 Y^3 + 6i Q_a^2 X^2Y + 6 Q_a^2 XY^2 \right)\\
&\quad + M_{-b} \left(5 Q_a X^4 + 4i Q_a X^3Y + 6 Q_a X^2Y^2 + 4i Q_a XY^3 + Q_a Y^4 \right)\\
&\quad + f(M_{-b} Q_a)+ f(\theta (A+iB)).
\end{align*}
We compute
\begin{align*}
{\mathcal R}_{2,1} 
&=f(W)-f(M_{-b}Q_a)-f(\theta (A+iB)) \\
&\quad - 5 Q^4 \theta A +by^2 Q^4 \theta B - i Q^4 \theta B+iby^2 Q^4 \theta A\\
&\quad+i \theta\big( \alpha_3 \lambda^3 \sin 2 \gamma +\alpha_4 b^2 \lambda^\frac32\sin \gamma\big) \rho\\
& =M_{-b} (5Q_a^4 X+i Q_a^4 Y) - 5 Q^4 \theta A +by^2 Q^4 \theta B - i Q^4 \theta B+iby^2 Q^4 \theta A\\
&\quad + M_{-b} \left(10 Q_a^3 X^2 + 2 Q_a^3 Y^2 + 4i Q_a^3 XY \right)\\
&\quad + M_{-b} \left( 10 Q_a^2 X^3 + 2i Q_a^2 Y^3 + 6i Q_a^2 X^2Y + 6 Q_a^2 XY^2 \right)\\
&\quad + M_{-b} \left(5 Q_a X^4 + 4i Q_a X^3Y + 6 Q_a X^2Y^2 + 4i Q_a XY^3 + Q_a Y^4 \right)\\
&\quad+i \theta\big( \alpha_3 \lambda^3 \sin 2 \gamma +\alpha_4 b^2 \lambda^\frac32\sin \gamma\big) \rho.
\end{align*}
We observe that
\begin{align*}
M_{-b} (5Q^4 X+i Q^4 Y)
& = Q^4 M_{-b} \left[ 3 M_b \theta(A+iB) + 2 \overline{M_b \theta(A+iB)}\right]\\
& = 3 Q^4 \theta(A+iB) + 2M_{-2b} Q^4 \theta(A-iB),
\end{align*}
thus
\begin{align*}
&M_{-b} (5Q_a^4 X+i Q_a^4 Y) - 5 Q^4 \theta A +by^2 Q^4 \theta B - i Q^4 \theta B+iby^2 Q^4 \theta A\\
&\quad = 2 \left(M_{-2b}-1+\tfrac i{2}b y^2\right) Q^4 \theta(A-iB) + M_{-b} (Q_a^4-Q^4) (5X+iY);
\end{align*}
and so,
\begin{align*}
{\mathcal R}_{2,1} 
& = 2 \left(M_{-2b}-1+\tfrac i{2}b y^2\right) Q^4 \theta(A-iB) + M_{-b} (Q_a^4-Q^4) (5X+iY)\\
&\quad + M_{-b} \left(10 Q_a^3 X^2 + 2 Q_a^3 Y^2 + 4i Q_a^3 XY \right)\\
&\quad + M_{-b} \left( 10 Q_a^2 X^3 + 2i Q_a^2 Y^3 + 6i Q_a^2 X^2Y + 6 Q_a^2 XY^2 \right)\\
&\quad + M_{-b} \left(5 Q_a X^4 + 4i Q_a X^3Y + 6 Q_a X^2Y^2 + 4i Q_a XY^3 + Q_a Y^4 \right)\\
&\quad+i \theta\big( \alpha_3 \lambda^3 \sin 2 \gamma +\alpha_4 b^2 \lambda^\frac32\sin \gamma\big) \rho.
\end{align*}
Therefore, by Lemma~\ref{le:V}, it is easy to estimate $\mathcal R_{2,1}$ in $H^1$ and $ y \mathcal R_{2,1}$ in $L^2 $
\[
\| \mathcal R_{2,1}\|_{H^1} + \| y \mathcal R_{2,1}\|_{L^2} \lesssim s^{-3}.
\]
Now, we write
\begin{align*}
M_b \mathcal R_{2,1} 
& =2 M_{b}\left(M_{-2b}-1+\tfrac i{2}b y^2 + \tfrac {b^2} {8} y^4\right) Q^4 \theta(A-iB) \\
&\quad + (Q_a^4-Q^4) (5X+iY) + 4i ( Q_a^3 - Q^3) XY \\
&\quad + 10 Q_a^2 X^3 + 2i Q_a^2 Y^3 + 6i Q_a^2 X^2Y + 6 Q_a^2 XY^2 \\
&\quad + 5 Q_a X^4 + 4i Q_a X^3Y + 6 Q_a X^2Y^2 + 4i Q_a XY^3 + Q_a Y^4\\
&\quad +i [ M_b\theta -1] \big( \alpha_3 \lambda^3 \sin 2 \gamma +\alpha_4 b^2 \lambda^\frac32\sin \gamma\big) \rho \\
&\quad - (M_{b} \theta - 1) \tfrac {b^2} {4} y^4 Q^4 (A-iB) +i \tfrac {b^2} {4} y^4 Q^4 B_2 \\
&\quad + 4i Q ^3 (XY -A_1 B_1 ) \\ 
&\quad + 10 Q_a^3 X^2 + 2 Q_a^3 Y^2 - \tfrac {b^2} {4} y^4 Q^4 A \\
& \quad + i \big( \alpha_3 \lambda^3 \sin 2 \gamma +\alpha_4 b^2 \lambda^\frac32\sin \gamma\big) \rho 
+ 4i Q ^3 A_1 B_1 +i \tfrac {b^2} {4} y^4 Q^4 B_1. 
\end{align*} 
We note that, by the bootstrap assumption, the first six lines are easily seen to be $O(s^{-4}) $ in $L^2$. 
For the seventh line, we first observe that
\begin{align*}
X = \{\Re [(M_b\theta-1)(A+iB)] + A_2 \}+ A_1 ,\\
Y = \{\Im[(M_b\theta-1) (A+iB)] + B_2 \} + B_1.
\end{align*}
Thus,
\begin{align*}
& XY -A_1 B_1  = \{\Re [(M_b\theta-1)(A+iB)] + A_2 \} \{\Im[(M_b\theta-1) (A+iB)] + B_2 \}\\
&\quad + A_1  \{\Im[(M_b\theta-1) (A+iB)] + B_2 \} + B_1 \{\Re [(M_b\theta-1)(A+iB)] + A_2 \},
\end{align*}
and from \eqref{festMbT}, one obtains
\[
|Q ^3 (XY -A_1 B_1 )| \lesssim  s^{-4}.
\]
The next-to-last line only contains real-valued terms, so its projection on $iQ$ vanishes. Last, the projection of the last line on $iQ$ vanishes because of the choice~\eqref{def:a3a4} of $\alpha _3$ and $\alpha _4$. 
Therefore,
\[
\langle M_b \mathcal R_{2,1} ,iQ\rangle = O(s^{-4})
\]
Now, we estimate $\mathcal R_{2,2}$.
We introduce a notation
\[
V_*(s,y)= e^{-i\gamma} \theta \left(\lambda^\frac32 |y| + i\kappa \lambda^\frac52 y^2\right)
= A_*+iB_*,
\]
where
\begin{align*}
A_* & = \lambda^\frac32 |y| \cos \gamma + \kappa \lambda^\frac52 y^2 \sin\gamma,\\
B_* & =-\lambda^\frac32 |y| \sin \gamma + \kappa \lambda^\frac52 y^2 \cos \gamma,
\end{align*}
so that
\[
\mathcal T_2 = -f( V_*)= -f(\theta (A_*+iB_*)).
\]
Using this notation, we estimate
\begin{align*}
|\mathcal R_{2,2}| 
&= |f(\theta (A+iB) - f (\theta (A_*+iB_*))|\\
&\lesssim \theta^5 \left(|A-A_*| +|B-B_*|\right) \left(A^4+B^4+A_*^4+B_*^4\right).
\end{align*}
By the asymptotics in Lemmas \ref{lem1bis} and \ref{lem2}, and $\theta(s,y)=0$ for $\lambda |y|\geq 2 \delta$,
we have
\[
|A-A_*| +|B-B_*|\lesssim \lambda^\frac32,\quad
\theta\left(A^4+B^4+A_*^4+B_*^4\right) \lesssim \lambda^6 (y^4+1).
\]
This implies the following pointwise estimate
\[
|\mathcal R_{2,2}| \lesssim \theta \lambda^\frac{15}2(y^4+1).
\]
Similarly, we have
\[
|\partial_y \mathcal R_{2,2}| \lesssim \theta \lambda^\frac{15}2(y^4+1).
\]
Thus, we obtain
\begin{align*} 
\| \mathcal R_{2,2}\|_{H^1} &\lesssim \delta^\frac12 s^{-3}, \\
\| y\mathcal R_{2,2}\|_{H^1} &\lesssim \delta^\frac32 s^{-2}, 
\end{align*} 
and
\[
|\langle M_b \mathcal R_{2,2} ,iQ\rangle|\lesssim s^{-\frac{15}2}.
\]
Besides, it is immediate  that all the terms in the expressions of $\mathcal S_2[\Gamma]$ and
$\mathcal T_2[\Gamma]$ are locally Lipschitz functions of $\Gamma$ with values in $L^2_y$.
We also observe that $f(W)-f(M_{-b} Q_a)$, as well as all the other terms in $\mathcal S_2[\Gamma]$ and
$\mathcal T_2[\Gamma]$ converge to $0$  in $L^2$ as $(\lambda,b,a)\to 0$.

\emph{Estimates of ${\mathcal R}_3$.}
We compute $G$ and $H$.
It follows from direct computations that
\begin{align*} 
&-\partial_s B_1 - L_+ A_1 - b (\Lambda B_1 + y^2 Q^4 B_1) - m_\gamma A_1 +m_\lambda \Lambda B_1\\
& \quad = 
- \lambda^\frac32 b (\Lambda_3 \psi_1 + y^2 Q^4 \psi_1) \sin \gamma 
 - m_\gamma \lambda^\frac32 (\varphi_1+\psi_1) \cos\gamma+ m_\lambda \lambda^\frac32 (\Lambda_3 \psi_1) \sin \gamma
\end{align*} 
and 
\begin{align*} 
&\partial_s A_1 - L_- B_1 + b ( \Lambda A_1 -y^2 Q^4 A_1 ) -m_\gamma B_1 -m_\lambda \Lambda A_1\\
& \quad= 
\lambda^\frac32 b (\Lambda_3 \varphi_1 - y^2 Q^4 \varphi_1 ) \cos \gamma 
- m_\gamma \lambda^\frac32(\varphi_1+\psi_1) \sin \gamma
-m_\lambda \lambda^\frac32 (\Lambda_3 \varphi_1) \cos \gamma.
\end{align*}
Similarly, using the definition of $(\varphi_2,\psi_2)$, it follows that 
\begin{align*} 
&-\partial_s B_2 - L_+ A_2 - b (\Lambda B_2 + y^2 Q^4 B_2) - m_\gamma A_2 +m_\lambda \Lambda B_2 \\
& \quad = \lambda^\frac52 (\Lambda_3 \psi_1+ y^2 Q^4 \psi_1) \sin \gamma 
+ \lambda^\frac52 b(\Lambda_5 \psi_2 + y^2 Q^4 \psi_2 ) \cos \gamma \\
&\qquad 
- m_\gamma \lambda^\frac52 (\varphi_2+\psi_2) \sin \gamma- m_\lambda \lambda^\frac52(\Lambda_5 \psi_2) \cos \gamma
\end{align*} 
and 
\begin{align*} 
&\partial_s A_2 - L_- B_2 + b ( \Lambda A_2 -y^2 Q^4 A_2 ) -m_\gamma B_2 -m_\lambda \Lambda A_2 \\
& \quad = -\lambda^\frac52 (\Lambda_3 \varphi_1- y^2 Q^4 \varphi_1+\alpha_1 \rho ) \cos \gamma 
 + \lambda^\frac52 b(\Lambda_5 \varphi_2 - y^2 Q^4 \varphi_2 ) \sin \gamma \\
& \qquad 
+ m_\gamma \lambda^\frac52(\varphi_2+\psi_2) \cos \gamma
-m_\lambda \lambda^\frac52 (\Lambda_5\varphi_2) \sin \gamma.
\end{align*} 
Summing up the previous identities, we obtain
\begin{equation}\label{eq:G}
\begin{aligned}
G & = 
- \lambda^\frac52 \Bigl(\frac b\lambda -1 \Bigr) ( \Lambda_3 \psi_1 + y^2 Q^4 \psi_1) \sin\gamma\\
&\quad + \lambda^\frac52 b ( \Lambda_5 \psi_2 + y^2 Q^4 \psi_2) \cos \gamma \\
&\quad - m_\gamma\lambda^\frac32 \left[ (\varphi_1+\psi_1)\cos \gamma
+ \lambda (\varphi_2+\psi_2) \sin \gamma \right]\\
&\quad + m_\lambda \lambda^\frac32 \left[(\Lambda_3 \psi_1)\sin \gamma
- \lambda (\Lambda_5 \psi_2)\cos \gamma \right]
\end{aligned}
\end{equation}
and
\begin{equation} \label{eq:H} \begin{aligned}
H & = 
\lambda^\frac52 \Bigl(\frac b\lambda -1 \Bigr) (\Lambda_3 \varphi_1 - y^2 Q^4 \varphi_1+\alpha_1\rho) \cos \gamma\\
&\quad + \lambda^\frac52b(\Lambda_5 \varphi_2 - y^2Q^4 \varphi_2+\alpha_2 \rho) \sin\gamma\\
&\quad - m_\gamma\lambda^\frac32 \left[ (\varphi_1+\psi_1)\sin \gamma
-\lambda (\varphi_2+\psi_2) \cos \gamma \right]\\
&\quad - m_\lambda \lambda^\frac32 \left[(\Lambda_3 \varphi_1)\cos \gamma
+ \lambda (\Lambda_5 \varphi_2) \sin \gamma \right].
\end{aligned}\end{equation}
We estimate $\theta H$ in $H^1$.
First, by \eqref{eq:BS1}, we have $|\frac b\lambda -1 |\leq Ks^{-1}$.
Moreover, by \eqref{fcomm2} with $k=0$, $\Lambda_3 \varphi_1 - y^2 Q^4 \varphi_1+\alpha_1\rho\in \mathcal Z_0$.
Thus, using Lemma~\ref{LE:3.2}, we have
\[
\lambda^\frac52 \Bigl|\frac b\lambda -1 \Bigr| \|\theta(\Lambda_3 \varphi_1 - y^2 Q^4 \varphi_1+\alpha_1\rho)\|_{H^1}
\lesssim \delta^\frac12 K s^{-3}.
\]
For the second term in $H$, we observe that by the asymptotics of $\varphi_2$ in Lemma~\ref{lem2},
$\Lambda_5 \varphi_2 - y^2 Q^4 \varphi_2 + \alpha_2\rho \in \mathcal Z_0$ and so using Lemma~\ref{LE:3.2},
\[
\lambda^\frac52 |b| \, \|\theta (\Lambda_5 \varphi_2 - y^2Q^4 \varphi_2+\alpha_2 \rho)\|_{H^1}
\lesssim \delta^{\frac 12} s^{-3} .
\]
By the asymptotics in Lemmas~\ref{lem1bis} and \ref{lem2}, and by \eqref{fcomm2}, we have
\[
\varphi_1+\psi_1 \in \mathcal Y,\quad \varphi_2+\psi_2\in \mathcal Y,\quad
\Lambda_3 \varphi_1 \in \mathcal Z_0,\quad \Lambda_5 \varphi_2\in \mathcal Z_0.
\]
Thus, by Lemma~\ref{LE:3.2}, using also the estimates
$|m_\gamma|+|m_\lambda|\leq K s^{-3}$ from \eqref{eq:BS2}, we have
\begin{equation} \label{eq:llH} \begin{aligned}
|m_\gamma| \lambda^\frac32 \left( \|\theta (\varphi_1+\psi_1)\|_{H^1}
+ \lambda \|\theta(\varphi_2+\psi_2)\|_{H^1} \right) &\lesssim K s^{-\frac92}\\
|m_\lambda| \lambda^\frac32 \left( \|\theta(\Lambda_3 \varphi_1)\|_{H^1} 
+ \lambda \|\theta(\Lambda_5 \varphi_2\|_{H^1}\right) &\lesssim \delta^\frac12 K s^{-4}.
\end{aligned}\end{equation}
Therefore, we have proved that $\|\theta H\|_{H^1} \lesssim \delta^\frac12 K s^{-3}$ for $s$ large, and the analogous estimate $\|\theta G \|_{H^1} \lesssim \delta^\frac12 K s^{-3}$ for $G$ is proved similarly.
Since $ \mathcal R_3 = \theta (G + i H) $, we obtain
\[
\|\mathcal R_3\|_{H^1}\lesssim \delta^\frac12 K s^{-3}.
\]
Now, we observe that by the definitions of $\alpha_1$ and $\alpha_2$ (see \eqref{def:a1}, \eqref{def:a2})
the terms in the first two lines in \eqref{eq:H} are orthogonal to $Q$. 
Thus, by \eqref{eq:BS1} and \eqref{eq:BS2} we have
\begin{equation*} 
| \langle H, Q \rangle| \lesssim  K s^{-\frac92} ,
\end{equation*} 
for $s$ large, so that using also the intermediate calculations in \eqref{festMbT} 
\begin{equation*} 
|\langle \mathcal R_3,iQ \rangle| = | \langle \theta H, Q \rangle|
= | \langle H, Q \rangle | + | \langle H, (1-\theta ) Q \rangle| \lesssim K s^{-\frac92} .
\end{equation*} 
Therefore,
\[
| \langle M_b \mathcal R_3,iQ \rangle |\le 
| \langle \mathcal R_3,iQ \rangle | + | \langle (M_b-1 ) \mathcal R_3,iQ \rangle |
\lesssim \delta^\frac12 K s^{-4}.
\]
Moreover,
\begin{equation*} 
 \| y \mathcal R_3 \| _{ L^2 }= \| y \mathcal R_3 \| _{ L^2 (I(s) )}\le 2\delta \lambda ^{-1} \| \mathcal R_3 \| _{ L^2 } \lesssim \delta^\frac32 K s^{-2} .
\end{equation*} 
Besides, we see from $\mathcal R_3=\theta (G+iH)$ and the expressions of $G$ and $H$ in
\eqref{eq:G}--\eqref{eq:H} that
$\mathcal R_3=\vec m \cdot \vec\rho_3 + \mathcal R_{\Gamma,3}$, where
the functions $\vec\rho_3[\Gamma]$ and
$\mathcal R_{\Gamma,3}[\Gamma]$ are locally Lipschitz functions of $\Gamma$ with values in $L^2_y$
and converge to $0$  in $L^2$ as $(\lambda,b)\to 0$.

\emph{Estimates of ${\mathcal R}_4$.}
Since $Q,\rho\in \mathcal Y$, using the definition of $\mathcal R_4$, and~\eqref{eq:BS1}, it follows that
\[
|\mathcal R_4|+|\partial_y \mathcal R_4|\lesssim a^2 e^{-\frac12|y|}\lesssim s^{-5} e^{-\frac12|y|}.
\]
Thus,
\[
\| {\mathcal R}_4 \|_{H^1} + \| y {\mathcal R}_4 \|_{L^2} \lesssim s^{-5}
\]
and
\[
|\langle M_b{\mathcal R}_4 , iQ \rangle |\lesssim s^{-5}.
\]
Besides, it is immediate by its definition that $\mathcal R_4[\Gamma]$ is a locally Lipschitz function of $\Gamma$ with values in $L^2_y$ and 
$\| \mathcal R_4[\Gamma]\|_{L^2} \to 0$ as $a\to 0$.

\emph{Estimates of ${\mathcal R}_5$.}
Since $\rho\in \mathcal Y$, using $|\Omega|\lesssim s^{-\frac 52}$ and~\eqref{eq:BS1} we have
\begin{align*}
|\mathcal R_5| + | \partial _y \mathcal R_5| & \lesssim s^{-\frac 52} e^{-\frac12|y|} \left( |1-\theta| + |\theta _y| + b\right)\\
&\lesssim s^{-\frac 52} e^{-\frac14|y|}e^{-\frac14\frac{\delta}\lambda} 
+ s^{-\frac 72} e^{-\frac12|y|}.
\end{align*}
Thus, 
\[
\| {\mathcal R}_5(s) \|_{H^1} + \| y {\mathcal R}_5(s) \|_{L^2} \lesssim s^{-\frac 72}.
\]
Next, we see that
\begin{align*}
|\Im (\mathcal R_5(s))|
&\lesssim s^{-\frac 52} e^{-\frac12|y|} \left( |1-\theta| + b^2\right)\\
&\lesssim s^{-\frac 52} e^{-\frac14|y|}e^{-\frac14\frac{\delta}\lambda} 
+ s^{-\frac 92} e^{-\frac12|y|}.
\end{align*}
Thus,
\[
|\langle M_b{\mathcal R}_5(s), iQ \rangle |\lesssim s^{-\frac 92}.
\]
Also, it is immediate by its definition that $\mathcal R_5[\Gamma]$ is a locally Lipschitz function of $\Gamma$ with values in $L^2_y$ and 
$\| \mathcal R_5[\Gamma]\|_{L^2} \to 0$ as $(\lambda,b)\to 0$.

\emph{Estimates of ${\mathcal R}_6$.}
It follows by easy calculations that
\begin{align*} 
 |\nu_1 ^{(4)} (x) | & \lesssim \delta ^{-5} \mathbf{1} _{ \{ \delta < |x|< 2\delta \} }, \\
 |\nu_1 ^{(5)} (x) | & \lesssim \delta ^{-6 } \mathbf{1} _{ \{ \delta < |x|< 2\delta \} }, \\
 |\nu_2 ^{(4)} (x) | & \lesssim \delta \mathbf{1} _{ \{ |x|< 2\delta \} }, \\
 |\nu_2 ^{(5)} (x) | & \lesssim \mathbf{1} _{ \{ |x|< 2\delta \} }. 
\end{align*} 
Next, by~\eqref{eq:BS1} 
\begin{equation*} 
s^2 \Bigl( \int _s ^{S_n} \lambda ^2 \Bigr)^2 + s^2 \int _s^{S_n} \lambda ^3 + s^3 \int _s ^{S_n} \lambda ^4 \lesssim 1. 
\end{equation*} 
Thus, using also \eqref{fEstThera1}, \eqref{fEstThera2} we see that
\begin{equation*} 
 |{\mathcal R}_6 (s,y) | \lesssim s^{- \frac {9} {2}} \delta ^{-5} \mathbf{1} _{ J(s) } + s^{- \frac {9} {2}} \delta \mathbf{1} _{ I(s) } , 
\end{equation*} 
and
\begin{equation*} 
 |\partial _y {\mathcal R}_6 (s,y) | \lesssim s^{- \frac {11} {2}} \delta ^{-6} \mathbf{1} _{ J(s) } + s^{- \frac {11} {2}} \mathbf{1} _{ I(s) } . 
\end{equation*} 
Therefore the error term $\mathcal R_6$ satisfies
\[
\| {\mathcal R}_6(s) \|_{H^1} \lesssim \delta^{-\frac{11}2} s^{-4},
\]
and
\[
|\langle M_b{\mathcal R}_6(s), iQ \rangle |\lesssim \delta s^{-\frac92},
\]
for $s$ large. 
Moreover,
\begin{equation*} 
 \| y \mathcal R_6 \| _{ L^2 }= \| y \mathcal R_6 \| _{ L^2 (I(s) )}\le 2\delta \lambda ^{-1} \| \mathcal R_6 \| _{ L^2 } \lesssim \delta^{-\frac{9}2} s^{-3} .
\end{equation*} 
Also, we rewrite $\mathcal R_6$ as
\[
\mathcal R_6[\Gamma] = \lambda^\frac52 e^{-i\gamma} \Bigl[ - \frac {1} {2} j_1 (\nu_1^{(4)} (\lambda y)+\nu_2^{(4)} (\lambda y)) 
+ (-i c_1  j_2  + c_2 j_3 ) \Theta ^{(4)} (\lambda y) \Bigr].
\]
Thus, it is immediate that $\mathcal R_6[\Gamma]$ is a locally Lipschitz function of $\Gamma$ with values in $L^2_y$ and 
$\| \mathcal R_6[\Gamma]\|_{L^2} \to 0$ as $(\lambda,j_1,j_2,j_3)\to 0$.

\emph{Estimates of ${\mathcal R}_7$.}
By change of variable, $x= \lambda y$,
the error term $\mathcal R_7$ satisfies
\begin{equation*} 
\|\mathcal R_7\|_{L^2} = \lambda^2 \|z^5\|_{L^2}, \quad \|\mathcal \partial _y R_7\|_{L^2} = \lambda^3 \|\partial _x (z^5)\|_{L^2}. 
\end{equation*} 
From \eqref{fEstz3} and \eqref{fEstz4},
\begin{equation*} 
 \| z^5 \| _{ L^2 } \lesssim (T_n-t)^5 \delta ^{- \frac {9} {2}} + (T_n-t)^{10} \delta ^{- \frac {29} {2}},
\end{equation*} 
and
\begin{equation*} 
 \| \partial _x (z^5) \| _{ L^2 } \lesssim (T_n-t)^5 \delta ^{- \frac {11} {2}} + (T_n-t)^{10} \delta ^{- \frac {31} {2}}.
\end{equation*} 
Since $(T_n-t) \lesssim (s(t))^{-1}$, it follows that
\[
\| {\mathcal R}_7 \|_{H^1} \lesssim \delta ^{- \frac {9} {2}} s^{-7} \lesssim s^{-6} ,
\]
for $s$ large, and so
\[
|\langle M_b {\mathcal R}_7(s), iQ \rangle |\lesssim s^{-6}
\]
for $s$ large. 
Moreover, since $\mathcal R_7= f(Z)$ and $Z$ is supported in $I(s)$,
\begin{equation*} 
 \| y \mathcal R_7 \| _{ L^2 }= \| y \mathcal R_7 \| _{ L^2 (I(s) )}\le 2\delta \lambda ^{-1} \| \mathcal R_7 \| _{ L^2 } \lesssim s^{-5} .
\end{equation*} 

\emph{Estimates of ${\mathcal R}_8$.}
Using \eqref{festW1}, \eqref{festZ5} and $e^{- \frac { |y|} {2}} \mathbf{1}_{J(s)} \lesssim e^{- \delta \frac {s} {4}}e^{- \frac { |y|} {4}}$, it follows that
\begin{equation*} 
( | W | + | \partial _y W| ) ( |Z| + |\partial _y Z| ) \lesssim s^{-\frac{9}2} e^{- \frac { |y|} {4}} + s^{-2 } \mathbf{1}_{J(s)} 
+ s^{- 5 } (|y|^3 + 1) \mathbf{1}_{I(s)}. 
\end{equation*} 
Moreover,
\begin{align*} 
[ | W | + | \partial _y W| + |Z| + |\partial _y Z| ]^3 
& \lesssim [ e^{- \frac { |y|} {2}} + s^{-\frac 12} \mathbf{1}_{I(s)} ]^3 \\
& \lesssim e^{- \frac {3 |y|} {2}} + s^{-\frac 32} .
\end{align*} 
Thus, by the definition of $\mathcal R_8$, it holds
\begin{align*}
|\mathcal R_8| + | \partial _y\mathcal R_8| & \lesssim ( | W | + | \partial _y W| ) ( |Z| + |\partial _y Z| ) [ | W | + | \partial _y W| + |Z| + |\partial _y Z| ]^3 \\
& \lesssim s^{-\frac{9}2} e^{- \frac { |y|} {4}} + s^{-\frac{7}2} \mathbf{1}_{J(s)} + s^{-\frac{13}2} (|y|^3 + 1) \mathbf{1}_{I(s)}. 
\end{align*}
Therefore,
\begin{align*} 
\| {\mathcal R}_8(s) \|_{H^1} & \lesssim \delta ^{\frac {1} {2}} s^{-3} ,\\
\| y {\mathcal R}_8(s) \|_{L^2} & \lesssim \delta ^{\frac {3} {2}} s^{-2} ,
\end{align*} 
and 
\[
|\langle M_b{\mathcal R}_8(s), iQ \rangle |\lesssim s^{-\frac {9}2},
\]
for $s$ large. 

Finally, we rewrite $\mathcal R_7+\mathcal R_8$ as
\[
(\mathcal R_7 +\mathcal R_8)[\Gamma]
= f(W+Z)-f(W)
\]
and it follows that $(\mathcal R_7+\mathcal R_8)[\Gamma]$ is a locally Lipschitz function of $\Gamma$ with values in $L^2_y$ and converges to $0$ as $\lambda\to 0$.

Therefore, Proposition~\ref{PR:1} is proved. 

\section{Modulation around the approximate blow-up solution}\label{S:4}

\subsection{Decomposition by modulation of the parameters}\label{S:4:1}

We recall that given $T\in \R$ and an initial data $u_T\in H^3\cap \Sigma$ such that $\partial_x u_T\in \Sigma$,
there exists a solution $u$ of \eqref{nls} defined on some interval $I\ni T$ with $u(T)=u_T$, in the same regularity class, i.e. $u\in \mathcal C(I,H^3\cap \Sigma)$, $\partial_x u\in \mathcal C(I,\Sigma)$,
$u\in \mathcal C^1(I,H^1)$.
This follows from standard arguments, see e.g. \cite[Proof of Lemma 5.6.2]{CLN10}.

We look for a decomposition of $u$ on $I$ of the form
\begin{equation}\label{def:u}\begin{aligned}
&u(t,x) =\lambda^{-\frac 12}(s) e^{i\gamma(s)} ( V[\Gamma(s)] (y)+ \varepsilon (s,y)),\\
&x = \lambda(s) y,\quad T -t = j_1(s)=\int^S_s \lambda^2
\end{aligned}
\end{equation}
where 
$V$ is defined in \eqref{eq:blowup},
$S>0$ and $(\gamma,\lambda,b,a)$ are $\mathcal C^1$ parameters such that 
\begin{equation} \label{fortho2} 
\left\{\begin{aligned} 
\langle \varepsilon, M_{ -b } Q \rangle & =0 \\
\langle \varepsilon, M_{ -b } y^2 Q \rangle & =0 \\
\langle \varepsilon, i M_{ -b } \Lambda Q \rangle & =0 \\
\langle \varepsilon, i M_{ -b } \rho \rangle & =0 
\end{aligned}\right. 
\end{equation} 

\begin{remark}\label{rk:ortho}
The choice of orthogonality relations \eqref{fortho2} is standard and corresponds to observations in \cite{Weinstein}
and \cite{RaphaelS}.
\end{remark}

Using \eqref{eq:resc1} and \eqref{eq:V} with $\vec m$ defined in \eqref{def:vecm},
we see that $\varepsilon$ and $(\gamma,\lambda,b,a)$ satisfy
\begin{equation} \label{eq:e}
\begin{aligned}
0 &= \lambda^{\frac 52} e^{-i\gamma} \left[ i\partial_t u + \partial_x^2 u +f(u)\right]= \mathcal E(V+\varepsilon) \\
& =\mathcal E(\varepsilon) + f(V+\varepsilon) - f(V) - f(\varepsilon)+ \mathcal E(V)\\
& = i \partial_s \varepsilon + \partial_y^2 \varepsilon - \varepsilon + f( V+\varepsilon) - f(V) 
+ ib\Lambda \varepsilon - i m_\lambda \Lambda \varepsilon- m_\gamma \varepsilon 
+\mathcal S_0 +\mathcal R.
\end{aligned}
\end{equation}
From its definition, it is easy to check that
$V\in \mathcal C(H^3\cap \Sigma)$, $\partial_x V\in \mathcal C(\Sigma)$, $V\in \mathcal C^1(H^1)$.
Thus, by formula~\eqref{def:u}, $\varepsilon \in \mathcal C(H^3\cap \Sigma)$, $\partial_x \varepsilon\in \mathcal C(\Sigma)$, $\varepsilon\in \mathcal C^1(H^1)$, and the above equation makes sense in $L^2$
and in $H^1_\Loc$.

\subsection{Bootstrap assumption}

The estimates of Section~\ref{S:3} on the approximate solution $V$,  given  in Lemma~\ref{le:V} and Proposition~\ref{PR:1}, were established assuming the bootstrap estimates~\eqref{eq:BS1} and \eqref{eq:BS2}.
Bootstrap estimates on the function $\varepsilon$ will also be necessary.
Below, we recapitulate all the bootstrap estimates to be used later

\begin{equation} \label{eq:BS}
\left\{\begin{aligned}
|\gamma(s) -s |&\leq s^{-\frac 74},\\
|b(s) - s^{-1} + \alpha_1 s^{-\frac52}\cos \gamma(s)|&\leq s^{-\frac 74},\\
\left|\frac{b(s)}{\lambda(s)} - 1\right|&\leq K s^{-1},\\
|a(s) - \alpha_1 s^{-\frac 52} \sin \gamma(s) |&\leq K s^{-3},\\
|\vec m(s)| & \leq K s^{-3},\\
\|\varepsilon(s)\|_{H^1}& \leq K^\frac14 s^{-2}, \\
 \| y \varepsilon(s) \| _{ L^2 } & \leq K^{\frac {1} {4}} s^{-1}.
\end{aligned}\right.
\end{equation}
Note that \eqref{eq:BS} implies \eqref{eq:BS1} and \eqref{eq:BS2} for $s$ large. The more refined estimates for $b$ and $a$ are needed to close the estimates of those parameters in Section~\ref{S:6}.

\subsection{Modulation equations}
For $g\in H^1$, we will use the notation
\[
d(g;\Gamma)=\bigl \| e^{-i\gamma}  \lambda^\frac12 g(\lambda y) - V[\Gamma](y)\bigr\|_{H^1}.
\]
We consider $u$ as in Section \ref{S:4:1}.
Note that 
\[u(t,x)=u(T-j_1,\lambda y)=u[\Gamma](\lambda y)\]
can be seen as a function of $\Gamma$ and $y$.
Moreover, $\varepsilon$ writes
\begin{equation}\label{eq:epseps}
\varepsilon(s,y)=\varepsilon[\Gamma](y)
=e^{-i\gamma} \lambda^\frac12 u[\Gamma](\lambda y)-V[\Gamma](y) .
\end{equation}
In particular, $\|\varepsilon\|_{H^1} = d(u[\Gamma];\Gamma)$.
We define
\begin{align*}
V &= V_1 + i V_2 \quad \mbox{where $V_1,V_2\in\R$}\\
\varepsilon &= \varepsilon_1 + i \varepsilon_2 \quad \mbox{where $\varepsilon_1,\varepsilon_2\in\R$}.
\end{align*}

\begin{lemma}\label{le:lipschitz}
Let $k\in \mathcal Y$.
For any nonnegative integers $p_1,p_2,q_1,q_2$, the function
\[
\Gamma \mapsto  \int \varepsilon_1^{p_1} \varepsilon_2^{p_2} 
V_1^{q_1}  V_2^{q_2}  k 
\]
is locally Lipschitz.
\end{lemma}
\begin{proof}
For any nonnegative integers $p_1',p_2',q_1',q_2'$, let
\[
\Gamma \mapsto \mathcal M [\Gamma] = \int u_1^{p_1'}[j_1](\lambda y) u_2^{p_2'}[j_1](\lambda y)
V_1^{q_1'}  V_2^{q_2'}  k ,
\]
where $u= u_1 + i u_2$, $u_1,u_2\in\R$ happens to depend only on $j_1$.
By the change of variable $x=\lambda y$, we have
\[
\mathcal M [\Gamma] =\frac 1\lambda \int u_1^{p_1'}[j_1](x) u_2^{p_2'}[j_1](x)
V_1^{q_1'}\left( \frac x\lambda\right)  V_2^{q_2'}\left( \frac x\lambda\right)  k\left( \frac x\lambda\right) dx.
\]
First, we observe that since $u\in \mathcal C^1_t(H_x^{-1})$, for any $\Phi \in H^1$,
the map $j_1 \mapsto \int  u [j_1](x) \Phi(x) dx $ is locally Lipschitz.
Second, since $V$ and $\partial_y V$ are bounded (see \eqref{festV1}), for any $\Phi \in H^1$,
the map $\lambda \mapsto \int \Phi(x)  V_1[\Gamma]\left( \frac x\lambda\right) k\left( \frac x\lambda\right)$ is locally Lipschitz and similarly for $V_2$.
It follows that $\mathcal M$ is locally Lipschitz.
Since $p_1',p_2',q_1',q_2'$ are arbitrary, the result follows by using \eqref{eq:epseps}.
\end{proof}

Let $g,h\in \mathcal Y$ be time-independent functions and let
\[
 {\mathcal I}[g,h](s) = \langle \varepsilon(s) , M_{-b(s)} (g+ih)\rangle
 = \langle \varepsilon[\Gamma(s)] , M_{-b(s)} (g+ih)\rangle.
\]
We give a general estimate on $\frac {d\mathcal I}{ds}$ that will be used to establish the modulation equations.
\begin{lemma} \label{le:dI}
On its interval of definition, the function $\mathcal I$ is $\mathcal C^1$ and satisfies
\begin{equation} \label{fmodul2} 
\begin{aligned}
\frac{d{\mathcal I[g,h]}}{ds} & = \langle \varepsilon, i M_{-b} (L_- g + i L_+ h)\rangle \\
&\quad - m_\gamma \langle Q,h\rangle + \tfrac14 (m_b -2bm_\lambda) \langle {y^2} Q,h\rangle
 - m_a \langle \rho, g\rangle + m_\lambda \langle \Lambda Q,g\rangle\\
&\quad +\vec m \cdot \vec \rho_{\mathcal I} +\mathcal R_{\mathcal I},
\end{aligned}
\end{equation} 
where $\vec\rho_{\mathcal I}(s)=\rho_{\mathcal I}[\Gamma(s)]$, $\mathcal R_{\mathcal I}(s)=\mathcal R_{\mathcal I}[\Gamma(s)]$ are Lipschitz functions of $\Gamma$ converging to $0$
as $(\lambda, b,a,j_1,j_2,j_3)\to 0$ and $d(u[\Gamma];\Gamma)\to 0$.

Moreover, assuming \eqref{eq:BS}, 
\begin{equation}\label{fmodul3}
|\vec \rho_{\mathcal I}|\lesssim s^{-1},\quad 
|\mathcal R_{\mathcal I}(s)|\lesssim s^{-3},
\end{equation}
and 
\begin{equation} \label{fmodul1} 
\begin{aligned}
\frac{d{\mathcal I}}{ds} & = \langle \varepsilon, i M_{-b} (L_- g + i L_+ h)\rangle \\
&\quad +4 \lambda^\frac32 \left[ \langle \varepsilon_1, 5 Q^3\varphi_1h\rangle \cos \gamma
+ \langle \varepsilon_2,Q^3\psi_1h\rangle \sin \gamma\right]\\
&\quad - 4 \lambda^\frac32\left[ \langle \varepsilon_1 ,Q^3\psi_1g\rangle\sin\gamma
+ \langle \varepsilon_2, Q^3\varphi_1g\rangle\cos\gamma \right]\\
&\quad - m_\gamma \langle Q,h\rangle + \tfrac14 (m_b -2bm_\lambda) \langle {y^2} Q,h\rangle
 - m_a \langle \rho, g\rangle + m_\lambda \langle \Lambda Q,g\rangle\\
&\quad + \langle M_b \mathcal R, h-ig\rangle +O(K^\frac12 s^{-4}).
\end{aligned}
\end{equation} 
\end{lemma}
\begin{proof}
Denote ${\mathcal I}(s) = {\mathcal I}[g,h](s)$.
We differentiate ${\mathcal I}$ using \eqref{eq:e}.
First,
\begin{equation}\label{page21}
\begin{aligned}
\frac {d{\mathcal I}}{ds} & = \langle \partial_s \varepsilon , M_{-b} (g+ih)\rangle
- \frac 14 b_s \langle \varepsilon , i y^2 M_{-b} (g+ih)\rangle\\
& = \langle -\partial_y^2 \varepsilon +\varepsilon -f(V+\varepsilon)+ f(V), i M_{-b} (g+ih) \rangle\\
&\quad - b \langle\Lambda \varepsilon, M_{-b} (g+ih)\rangle - \frac 14 b_s \langle \varepsilon , i y^2 M_{-b} (g+ih)\rangle
 \\
&\quad 
+ \langle m_\lambda \Lambda \varepsilon - i m_\gamma \varepsilon, M_{-b} (g+ih)\rangle\\
&\quad - \langle \mathcal S_0,i M_{-b} (g+ih) \rangle - \langle \mathcal R,i M_{-b} (g+ih)\rangle.
\end{aligned}
\end{equation}
Integrating by parts and using~\eqref{fusef1}, we have
\begin{align*}
\langle -\partial_y^2 \varepsilon, i M_{-b} (g+ih) \rangle & =
- \langle \varepsilon, i M_{-b} (g''+ih'') \rangle
- b \langle \varepsilon, y M_{-b} (g'+ih') \rangle\\
& \quad - \frac 12 b \langle \varepsilon, M_{-b} (g+ih) \rangle + \frac 14 b^2\langle \varepsilon, i y^2 M_{-b} (g+ih) \rangle
\end{align*}
and
\begin{align*}
- b \langle\Lambda \varepsilon, M_{-b} (g+ih)\rangle &=
\frac 12 b \langle \varepsilon, M_{-b} (g+ih)\rangle 
+ b \langle \varepsilon, y M_{-b} (g'+ih') \rangle\\
&\quad - \frac 12 b^2 \langle \varepsilon, iy^2M_{-b} (g+ih)\rangle .
\end{align*}
Thus, for $s$ large,
\begin{align}
& \langle -\partial_y^2 \varepsilon , i M_{-b} (g+ih) \rangle - b \langle\Lambda \varepsilon, M_{-b} (g+ih)\rangle - \frac 14 b_s \langle \varepsilon , i y^2 M_{-b} (g+ih)\rangle\nonumber\\
& \quad =
- \langle \varepsilon, i M_{-b} (g''+ih'') \rangle
- \frac 14 (b_s+b^2) \langle \varepsilon , i y^2 M_{-b} (g+ih)\rangle\nonumber\\
& \quad =
- \langle \varepsilon, i M_{-b} (g''+ih'') \rangle
- \frac 14 a \langle \varepsilon , i y^2 M_{-b} (g+ih)\rangle
- \frac 14 m_b \langle \varepsilon , i y^2 M_{-b} (g+ih)\rangle \nonumber\\
& \quad = - \langle \varepsilon, i M_{-b} (g''+ih'') \rangle + O(K^\frac14 s^{-\frac 92})
\label{ond2g}
\end{align}
where we have estimated, using \eqref{eq:BS}, 
\begin{align*}
|a \langle \varepsilon , i y^2 M_{-b} (g+ih)\rangle| \lesssim |a| \, \| \varepsilon \| _{ L^2 } \lesssim K^\frac14 s^{-\frac 92} ,\\
|m_b \langle \varepsilon , i y^2 M_{-b} (g+ih)\rangle| \lesssim |m_b| \, \| \varepsilon \| _{ L^2 } \lesssim K^\frac54 s^{-5}.
\end{align*}
Now, we decompose the term
\begin{equation*}
- \langle f(V+\varepsilon)- f(V), i M_{-b} (g+ih) \rangle
= \langle f(M_bV+M_b\varepsilon)- f(M_b V), h - ig \rangle .
\end{equation*}
We extend the notation~\eqref{eq:XY} by setting
\begin{equation} \label{fdfnxt} 
\begin{aligned} 
 &\widetilde{X} = X+ \Re (M_b Z) = \Re (M_b V - Q_a)\\ 
 &\widetilde{Y} = Y+ \Im (M_b Z) =\Im (M_b V - Q_a) \\
 &E_1 = \Re(M_b \varepsilon)  ,\quad E_2 = \Im(M_b \varepsilon),
\end{aligned} 
\end{equation} 
so that
\begin{equation} \label{fdfnxt2} 
\begin{aligned} 
M_b V &= Q_a + \widetilde{X} + i \widetilde{Y} ,\\
M_b V + M_b \varepsilon & = Q_a + \widetilde{X} + E_1 + i (\widetilde{Y} + E_2).
\end{aligned} 
\end{equation} 
Using the identity~\eqref{eq:QXY} twice (once with $( \widetilde{X} +E_1, \widetilde{Y} +E_2)$ and once with $( \widetilde{X} , \widetilde{Y} )$), we obtain
\begin{equation}\label{eq:NL}
\begin{aligned}
f(M_bV+ & M_b\varepsilon)- f(M_b V)\\ 
& = f(Q_a+ \widetilde{X} +E_1+i \widetilde{Y} +iE_2) - f(Q_a+ \widetilde{X} +i \widetilde{Y} )\\
&= 5 Q_a^4 E_1 + i Q_a^4 E_2 \\ 
&\quad+ 10 Q_a^3 (( \widetilde{X} +E_1)^2- \widetilde{X} ^2) + 2 Q_a^3 (( \widetilde{Y} +E_2)^2- \widetilde{Y} ^2)\\
&\quad + 4i Q_a^3 (( \widetilde{X} +E_1)( \widetilde{Y} +E_2)- \widetilde{X} \widetilde{Y} ) + \widetilde{\mathcal R}_1,
\end{aligned}
\end{equation}
where 
\begin{equation*}
\begin{aligned}
& \widetilde{\mathcal R}_1 = 10 Q_a^2 (( \widetilde{X} +E_1)^3- \widetilde{X} ^3) + 2i Q_a^2 (( \widetilde{Y} +E_2)^3- \widetilde{Y} ^3) \\
&\quad + 6i Q_a^2 (( \widetilde{X} +E_1)^2( \widetilde{Y} +E_2)- \widetilde{X} ^2 \widetilde{Y} ) + 6 Q_a^2 (( \widetilde{X} +E_1)( \widetilde{Y} +E_2)^2- \widetilde{X} \widetilde{Y} ^2) \\
&\quad +5 Q_a (( \widetilde{X} +E_1)^4- \widetilde{X} ^4)+ 4i Q_a (( \widetilde{X} +E_1)^3( \widetilde{Y} +E_2)- \widetilde{X} ^3 \widetilde{Y} ) \\
&\quad + 6 Q_a (( \widetilde{X} +E_1)^2( \widetilde{Y} +E_2)^2- \widetilde{X} ^2 \widetilde{Y} ^2)\\
&\quad + 4i Q_a (( \widetilde{X} +E_1)( \widetilde{Y} +E_2)^3- \widetilde{X} \widetilde{Y} ^3) + Q_a (( \widetilde{Y} +E_2)^4- \widetilde{Y} ^4) \\
&\quad+f( \widetilde{X} +E_1+i \widetilde{Y} +iE_2)-f( \widetilde{X} +i \widetilde{Y} ).
\end{aligned}
\end{equation*}
Recall that by \eqref{eq:BS}, $ |E_1| + |E_2| \lesssim \|\varepsilon \| _{ L^\infty } \lesssim K^{\frac {1} {4}} s^{-2}$
and thus by~\eqref{festW0}, \eqref{festZ5}, we have
\begin{equation} \label{fxye1e2} 
\begin{aligned} 
 | \widetilde{X} |+ | \widetilde{Y} | + |E_1| + |E_2| 
 & \lesssim s^{-\frac {3} {2}} + s^{-\frac 32} ( |y| + s^{-1}  |y|^2 )\mathbf{1}_{I(s)}+\delta^{-1} s^{-\frac 32} \mathbf{1}_{J(s)}\\
 & \lesssim s^{-\frac {3} {2}} (1+ |y|^2) +\delta^{-1} s^{-\frac 32} \mathbf{1}_{J(s)}.
\end{aligned} 
\end{equation} 
Therefore, 
\begin{equation*} 
\begin{aligned} 
 | \widetilde{\mathcal R}_1 | & \lesssim ( | \widetilde{X} |+ | \widetilde{Y} | + |E_1| + |E_2| )^3 + ( | \widetilde{X} |+ | \widetilde{Y} | + |E_1| + |E_2| )^5 \\
 & \lesssim   (s^{-\frac {3} {2}} (1+ |y|^2) )^3 +  (s^{-\frac {3} {2}} (1+ |y|^2) )^5
 + (\delta^{-3} s^{-\frac 92}+\delta^{-5} s^{-\frac{15}2}) \mathbf{1}_{J(s)}\\
 & \lesssim  s^{-\frac {9} {2}} (1+ |y|^{10}) +  \delta^{-3} s^{-\frac 92}\mathbf{1}_{J(s)}.
\end{aligned} 
\end{equation*} 
Since $g,h\in \mathcal Y $, it follows that for $s$ large
\begin{equation*} 
 |\langle \widetilde{\mathcal R}_1 , h - ig \rangle | \lesssim s^{- \frac {9} {2}} .
\end{equation*} 
For the first line in the last identity in~\eqref{eq:NL} (recalling  $ |Q_a - Q| \lesssim |a\rho|\lesssim s^{-\frac 52}|\rho|$),
we have
\[
|\langle 5 Q_a^4 E_1 + i Q_a^4 E_2,h-ig\rangle
- \langle 5 Q^4 E_1 + i Q^4 E_2,h-ig\rangle|\lesssim |a| \|\varepsilon\|_{L^2} \lesssim K^\frac14 s^{-\frac92} ;
\]
and (using the definition of $E_1, E_2$)
\begin{align*}
\langle 5 Q^4 E_1 + i Q^4 E_2,h-ig\rangle
&= \langle Q^4 E_1 , 5h \rangle - \langle Q^4 E_2, g\rangle = \langle Q^4 M_b \varepsilon, 5h-ig\rangle \\
&= \langle \varepsilon, M_{-b}(5Q^4 h - i Q^4g)\rangle = - \langle \varepsilon, i M_{-b}( Q^4g + 5 i Q^4 h )\rangle.
\end{align*}
We now estimate the first quadratic term in the last identity in~\eqref{eq:NL}. We first estimate
\begin{equation*} 
 | (Q_a^3- Q^3 ) (( \widetilde{X} +E_1)^2- \widetilde{X} ^2) | \lesssim |a\rho| ( \widetilde{X} ^2 + E_1^2) \lesssim s^{-\frac {11} {2}} 
\end{equation*} 
by~\eqref{fxye1e2}. Next, we write
\begin{equation*} 
10 Q^3 (( \widetilde{X} +E_1)^2- \widetilde{X} ^2) =  \widetilde{\mathcal R}_2+20 Q^3 A_1 \varepsilon _1 ,
\end{equation*} 
where
\begin{equation*} 
\widetilde{\mathcal R}_2 = 20 Q^3 ( \widetilde{X} E_1 - A_1 \varepsilon _1) + 10 Q^3 E_1^2.
\end{equation*} 
On the one hand, using  \eqref{festW00} and \eqref{festMbT}
\begin{equation*} 
Q^2 |M_b \theta (A+iB ) - (A+iB ) | \lesssim s^{-1} Q | A+iB| \lesssim s^{- \frac {5} {2}} ,
\end{equation*} 
and in particular 
\begin{equation*} 
Q^2 | {X} -A | \lesssim s^{- \frac {5} {2}} . 
\end{equation*} 
Hence, using \eqref{festZ5}
\begin{equation*} 
Q^2 | \widetilde{X} -A_1 | \lesssim s^{- \frac {5} {2}}+ Q^2 |A_2|+Q^2|Z| \lesssim s^{- \frac {5} {2}} . 
\end{equation*} 
Moreover, by \eqref{eq:BS} and \eqref{festMbT} 
\begin{equation}\label{festE1}
|E_1|\lesssim K^\frac 14 s^{-2},\quad Q |E_1-\varepsilon_1|
=Q|\Re((M_b-1) \varepsilon)|\lesssim s^{-1} \|\varepsilon\|_{L^\infty}\lesssim K^\frac14 s^{-3}.
\end{equation}
Therefore, using also \eqref{festW00}, we obtain
\begin{equation*} 
Q^3 | \widetilde{X} E_1 - A_1 \varepsilon _1| \lesssim Q^3 |( \widetilde{X} -A_1) E_1 | + Q^3 |A_1 (E_1 - \varepsilon _1) | \lesssim K^{\frac {1} {4}} s^{- \frac {9} {2}} .
\end{equation*} 
Since also
\begin{equation*} 
Q^3 E_1^2 \lesssim K^{\frac {1} {2}} s^{-4} 
\end{equation*} 
we estimate
\begin{equation*} 
 | \langle \widetilde{\mathcal R}_2 , h-ig \rangle | \lesssim K^{\frac {1} {4}} s^{- \frac {9} {2}} + K^{\frac {1} {2}} s^{-4} \lesssim K^{\frac {1} {2}} s^{-4} .
\end{equation*} 
On the other hand,
\[
\langle 20 Q^3 A_1 \varepsilon_1,h-ig\rangle
=\langle 20 Q^3 A_1 \varepsilon_1,h\rangle = 20 \lambda^\frac32 \langle \varepsilon_1, Q^3\varphi_1h\rangle \cos \gamma .
\]
Combining the above two estimates, we obtain
\[
\langle 10 Q_a^3 (( \widetilde{X} +E_1)^2- \widetilde{X} ^2),h-ig\rangle = 20 \lambda^\frac32 \langle \varepsilon_1, Q^3\varphi_1h\rangle \cos \gamma
+O(K^\frac12 s^{-4} ).
\]
Similarly, for the other two quadratic terms in the last identity in~\eqref{eq:NL}, we obtain 
\begin{align*}
 \langle 2 Q_a^3 (( \widetilde{Y} +E_2)^2- \widetilde{Y} ^2),h-ig\rangle& = 4 \lambda^\frac32 \langle \varepsilon_2,Q^3\psi_1h\rangle \sin \gamma+O(K^\frac12 s^{-4} ) , \\
\langle 4i Q_a^3 ( \widetilde{X} +E_1)( \widetilde{Y} +E_2)- \widetilde{X} \widetilde{Y} ),h-ig\rangle & = - 4 \lambda^\frac32 \langle \varepsilon_1 ,Q^3\psi_1g\rangle\sin\gamma
\\&\quad -4 \lambda^\frac32 \langle \varepsilon_2, Q^3\varphi_1g\rangle\cos\gamma 
+O(K^\frac12 s^{-4} ) .
\end{align*}
Summing up the above calculations, we see that
\begin{align*}
 \langle  - f(V+\varepsilon) + f(V),&  i M_{-b} (g+ih)\rangle 
 = - \langle \varepsilon, i M_{-b} (Q^4 g + 5 i Q^4 h)\rangle\\
& \quad  +4 \lambda^\frac32 \left[ \langle \varepsilon_1, 5 Q^3\varphi_1h\rangle \cos \gamma
+ \langle \varepsilon_2,Q^3\psi_1h\rangle \sin \gamma\right]\\
& \quad - 4 \lambda^\frac32\left[ \langle \varepsilon_1 ,Q^3\psi_1g\rangle\sin\gamma
+ \langle \varepsilon_2, Q^3\varphi_1g\rangle\cos\gamma \right] +O(K^\frac12 s^{-4}).
\end{align*}
Therefore, combining this estimate with \eqref{ond2g},
we obtain for the first two lines of the right-hand side of \eqref{page21}
\begin{equation} \label{finter12} 
\begin{aligned}
 \langle -\partial_y^2& \varepsilon +\varepsilon -f(V+\varepsilon)+ f(V), i M_{-b} (g+ih) \rangle\\
&\quad - b \langle\Lambda \varepsilon, M_{-b} (g+ih)\rangle - \frac 14 b_s \langle \varepsilon , i y^2 M_{-b} (g+ih)\rangle
 \\
 & = \langle \varepsilon, i M_{-b} (L_- g + i L_+ h)\rangle \\ 
 & \quad+4 \lambda^\frac32 \left[ \langle \varepsilon_1, 5 Q^3\varphi_1h\rangle \cos \gamma
+ \langle \varepsilon_2,Q^3\psi_1h\rangle \sin \gamma\right]\\
& \quad - 4 \lambda^\frac32\left[ \langle \varepsilon_1 ,Q^3\psi_1g\rangle\sin\gamma
+ \langle \varepsilon_2, Q^3\varphi_1g\rangle\cos\gamma \right] +O(K^\frac12 s^{-4}).
\end{aligned}
\end{equation} 

For the next term in the right-hand side of \eqref{page21}, we use \eqref{eq:BS},
\[
| \langle m_\lambda \Lambda \varepsilon - i m_\gamma \varepsilon, M_{-b} (g+ih)\rangle |
\lesssim ( |m_\lambda | + |m_\gamma |) \| \varepsilon \| _{ H^1 }
\lesssim K^\frac54 s^{-5}.
\]
Last, using the expression of $\mathcal S_0$, we estimate using \eqref{eq:BS} 
and $|a|\lesssim s^{-\frac52}$,
\begin{align*}
- \langle \mathcal S_0,i M_{-b} (g+ih) \rangle
& = - m_\gamma \langle Q_a,h\rangle + \tfrac14 (m_b -2bm_\lambda) \langle {y^2} Q_a,h\rangle\\
&\quad - m_a \langle \rho, g\rangle + m_\lambda \langle \Lambda Q_a,g\rangle\\
& = - m_\gamma \langle Q,h\rangle + \tfrac14 (m_b -2bm_\lambda) \langle {y^2} Q,h\rangle\\
&\quad - m_a \langle \rho, g\rangle + m_\lambda \langle \Lambda Q,g\rangle
+O(Ks^{-\frac{11}2}).
\end{align*}
As a first consequence of these calculations,
we see that $\frac {d\mathcal I}{ds}$ has the form \eqref{fmodul2} where
$\vec \rho_{\mathcal I}$ and $\mathcal R_{\mathcal I}$ 
converge to $0$ as $(\lambda, b,a,j_1,j_2,j_3)\to 0$ and $\|\varepsilon\|_{H^1}=d(u[\Gamma];\Gamma)\to 0$.
Moreover, we obtain \eqref{fmodul1}  for $s$ large,
and using
$\lambda ^{\frac {3} {2}} \| \varepsilon \| _{ L^2 }\lesssim K^{\frac {1} {4}} s^{- \frac {7} {2}}$ and~\eqref{eq:R1}, we check that $\vec \rho_{\mathcal I}$ and $\mathcal R_{\mathcal I}$ defined in \eqref{fmodul2}  satisfy the estimates \eqref{fmodul3}.

Finally, the Lipschitz property of $\vec \rho_{\mathcal I}$ and $\mathcal R_{\mathcal I}$ follows from
Lemma~\ref{le:lipschitz} through a tedious but elementary examination of all the error terms.
As a first example, we observe that
\[
m_b \langle \varepsilon , i y^2 M_{-b} (g+ih)\rangle
=\vec m \cdot \vec \rho
\]
where 
\[
\vec \rho = \begin{pmatrix} 0\\ 0 \\ \langle \varepsilon , i y^2 M_{-b} (g+ih)\rangle  \\0  \end{pmatrix} 
\]
and $\Gamma\mapsto \langle \varepsilon , i y^2 M_{-b} (g+ih)\rangle$ is clearly locally Lipschitz by
Lemma~\ref{le:lipschitz}.
To conclude, we treat a typical nonlinear term from \eqref{eq:NL}
\[
\langle Q_a^3 (( \widetilde{X} +E_1)^2- \widetilde{X} ^2) , g+ih \rangle
=\langle Q_a^3 E_1^2 , g+ih \rangle
+ 2 \langle Q_a^3 \widetilde{X} E_1 , g+ih \rangle.
\]
The first term on the right-hand side is immediate using Lemma~\ref{le:lipschitz}.
For the second term, we use \eqref{fdfnxt} to replace $\widetilde X$ and then apply again
Lemma~\ref{le:lipschitz}.
\end{proof}

We are now in a position to actually prove the existence of a decomposition \eqref{def:u}
on some time interval for initial data close $V[\Gamma]$.
 
\begin{lemma}\label{le:CL}
There exists $\omega_0>0$ such that for any $S>0$, $\omega\in (0,\omega_0)$
and $(u^{in},\Gamma^{in})\in H^1 \times (\R^6\cap \{\lambda >0\})$, if 
\[
|(\lambda^{in},b^{in},a^{in},j_1^{in},j_2^{in},j_3^{in})|+|d(u^{in};\Gamma^{in})| < \omega,
\]
and $u$ is the solution of \eqref{nls} with $u(j_1^{in})=u^{in}$,
then there exist $\bar s\in (-\infty,S)$ and  $\Gamma\in \mathcal C^1([\bar s,S])$
such that $\Gamma(S) =  \Gamma^{in}$,
\[
|(\lambda ,b ,a ,j_1 ,j_2 ,j_3 )|+|d(u(j_1^{in}-j_1) ;\Gamma )| < \omega
\]
and which satisfies the following system of ODEs
\begin{equation} \label{fortho33} 
\begin{cases} 
\partial_s\langle \varepsilon, i M_{ -b } \rho \rangle =0 \\
\partial_s\langle \varepsilon, M_{ -b } y^2 Q \rangle =0 \\
\partial_s\langle \varepsilon, i M_{ -b } \Lambda Q \rangle =0 \\
\partial_s\langle \varepsilon, M_{ -b } Q \rangle =0 \\
\partial_s j_1 = -\lambda^2\\
\partial_s j_2 = -\lambda^3\\
\partial_s j_3 = -\lambda^4
\end{cases} 
\end{equation} 
where $\varepsilon=\varepsilon[\Gamma]$ is defined in \eqref{eq:epseps}.
\end{lemma}
\begin{proof}
It follows from Lemma \ref{le:dI} applied with the following choice of $(g,h)$
\[
(0,\rho),\quad  (y^2 Q,0),\quad (0,\Lambda Q),\quad (Q,0), 
\]
that the system \eqref{fortho33} is equivalent to
\[
\mathbf{D} [\Gamma] \partial_s \Gamma = \mathbf{F}[\Gamma]
\]
where
\[
\mathbf{D} = \mathbf{D}_0 + \mathbf{D}_\Gamma,
\]
\[
\mathbf{D}_0 = \begin{pmatrix} 

\langle Q,\rho\rangle & 0 & -\frac 14 \langle y^2 Q,\rho\rangle & 0 &0 &0 &0  \\
0 & - \langle \Lambda Q,y^2 Q\rangle & 0 & \langle \rho,y^2 Q\rangle & 0&0&0\\
4 \langle Q, \Lambda Q\rangle & 0 & -\langle y^2 Q, \Lambda Q\rangle & 0  & 0&0&0\\
0 & - \langle \Lambda Q,  Q\rangle & 0 & \langle \rho,Q\rangle   & 0&0&0\\
0 & 0 & 0 & 0 & 1 & 0 & 0 \\
0 & 0 & 0 & 0 & 0 & 1 & 0 \\
0 & 0 & 0 & 0 & 0 & 0 & 1 \\
\end{pmatrix}
\]
and
\[
\mathbf{D}_\Gamma \to 0 \quad \mbox{and}\quad \mathbf{F}[\Gamma]\to 0
\]
as 
$|(\lambda ,b ,a ,j_1 ,j_2 ,j_3 )|+\|\varepsilon\|_{H^1}\to 0$.
By \eqref{fortho3} and \eqref{eq:rho}, the matrix $\mathbf{D}_0$ is invertible.
Moreover, using also Lemma \ref{le:lipschitz}, ${\mathbf D}_\Gamma$ and ${\mathbf F}$ are locally Lipschitz. For small
$|(\lambda ,b ,a ,j_1 ,j_2 ,j_3 )|$ and $\|\varepsilon\|_{H^1}$, the matrix 
$\mathbf{D}$ is therefore invertible and 
the result follows from the Cauchy-Lipschitz theorem.
\end{proof}

The estimates in Lemma \ref{le:dI} allow us to 
improve the estimates in \eqref{eq:BS} related to the equations satisfied by the parameters. Recall that
 $m_\gamma$, $m_\lambda$, $m_b$ and $m_a$ are defined in \eqref{def:m}.
 We also prove a refined estimate on $m_a=a_s-\Omega$ (see Remarks~\ref{rk:omega} and~\ref{rk:Omegab}).
\begin{lemma} \label{eestm1} 
Under the assumption~\eqref{eq:BS}, for $s$ large
\begin{equation} \label{festm1}
 |\vec m| \lesssim s^{-3}.
\end{equation} 
In particular,
\begin{equation} \label{eq:S0} 
 \| \mathcal S_0 \| _{ H^1 } + \| y\mathcal S_0 \| _{ L^2 } \lesssim s^{-3} . 
\end{equation} 
Moreover,
\begin{equation} \label{festm2} 
m_a = - \alpha _5 \lambda^\frac32\left[ \langle \varepsilon_1 ,Q^4\psi_1\rangle\sin\gamma
+ \langle \varepsilon_2, Q^4\varphi_1 \rangle\cos\gamma \right]
+ O( K^\frac12 s^{-4})
\end{equation} 
for $s$ large, where $\alpha _5= 32 \bigl( \int y^2 Q^2 \bigr)^{-1}$.
\end{lemma}

\begin{proof}
All the estimates in the proof hold for $s$ large (possibly depending on $K$). 
By the orthogonality relations~\eqref{fortho2}, we have
\begin{equation} \label{fortho4} 
{\mathcal I} [Q, 0] = {\mathcal I}[y^2 Q, 0] = {\mathcal I}[ 0, \Lambda Q] = {\mathcal I} [0,\rho ] = 0.
\end{equation}
We use the simplified estimate \eqref{fmodul2} to derive the estimates \eqref{festm1} on $m_a$, $m_\lambda$, $m_b$ and $m_\gamma$.
Using first~\eqref{fmodul2} for ${\mathcal I}[Q, 0] = 0$, i.e. $g=Q$ and $h=0$, together with the identities $L_- Q=0$, $\langle Q, \Lambda Q\rangle =0$ and $\langle \rho ,Q\rangle = \frac {1} {8} \int y^2Q^2 \not = 0$ (by~\eqref{fortho1}, \eqref{fortho3} and~\eqref{eq:rho}) we obtain
\begin{equation} \label{festma} 
 |m_a | \lesssim  s^{-3}. 
\end{equation} 
Next, using\eqref{fmodul2} for ${\mathcal I}[y^2Q, 0] = 0$, together with the identities $L_- (y^2Q)=-4 \Lambda Q$, $\langle \varepsilon, i M_{ -b } \Lambda Q \rangle =0 $, $\langle \Lambda Q ,y^2Q\rangle = - \int y^2Q^2 \not = 0$ (by~\eqref{fortho1}, \eqref{fortho2} and~\eqref{eq:rho}), and~\eqref{festma} we obtain
\begin{equation} \label{festml} 
 |m_\lambda | \lesssim  s^{-3}. 
\end{equation} 
Similarly, using\eqref{fmodul2} for ${\mathcal I} [0, \Lambda Q] = 0$, together with the identities $L_+ (\Lambda Q)=-2 Q$, $\langle \varepsilon, M_{ -b } Q \rangle =0 $, $\langle \Lambda Q ,y^2Q\rangle = - \int y^2Q^2 \not = 0$, $\langle Q, \Lambda Q\rangle =0$ (by~\eqref{fortho1}, \eqref{fortho2}, \eqref{eq:rho} and~\eqref{fortho3}), and~\eqref{festml} we obtain
\begin{equation} \label{festmb} 
 |m_b | \lesssim s^{-3}. 
\end{equation}
Last, using \eqref{fmodul2} for ${\mathcal I}[0, \rho ] = 0$, the identities $L_+ \rho = \frac {y^2} {4} Q$, $\langle \varepsilon, M_{ -b } y^2 Q \rangle =0 $ and $\langle \rho ,Q\rangle = \frac {1} {8} \int y^2Q^2 \not = 0$ (by~\eqref{fortho1}, \eqref{fortho2} and~\eqref{eq:rho}), and~\eqref{festml}-\eqref{festmb} we obtain
\begin{equation} \label{festmg} 
 |m_\gamma | \lesssim  s^{-3}. 
\end{equation}
Thus we have proved~\eqref{festm1}. 

Finally, using~\eqref{fmodul1} for ${\mathcal I}[Q, 0] = 0$, together with the identities $L_- Q=0$, $\langle Q, \Lambda Q\rangle =0$ and $\langle \rho ,Q\rangle = \frac {1} {8} \int y^2Q^2 \not = 0$ (by~\eqref{fortho1}, \eqref{fortho3} and~\eqref{eq:rho}) and with the estimate~\eqref{eq:R2}, we obtain~\eqref{festm2}. 
\end{proof}

\begin{remark}\label{rk:Omegab}
The terms in $\varepsilon_1$ and $\varepsilon_2$ in the right-hand side of \eqref{festm2}
are of size $O(K^{\frac14} s^{-\frac 72})$, in particular time integration does not
 improve the estimate of $a$ in \eqref{eq:BS}.
 Therefore, we need to introduce another functional $\mathcal J$ in order to control these terms
 at this order of $s^{-1}$.
 \end{remark}

We introduce some notation. 
Since $Q^4 \varphi _1\in \mathcal Y$ and $Q^4 \psi _1\in \mathcal Y$, it follows from Lemma~\ref{lem1bis}~\eqref{lem1bis:2} that
there exists a unique pair of even functions $(\varphi _0,\psi _0) \in \mathcal Z_0^2$ such that
\begin{equation} \label{fedfpsiz} 
\begin{cases} 
\psi _0 + L_+ \varphi _0 = - Q^4 \psi _1 \\
\varphi _0 + L_- \psi _0 = Q^4 \varphi _1 
\end{cases} 
\end{equation}
and satisfying
\[
\varphi_0 (x) = c + v(x),\quad
\psi _0(x) = -c + w(x)
\]
where $c\in \R$ and $v, w\in \mathcal Y$.
We define 
\begin{equation} \label{fdfnJ} 
\mathcal J = \cos \gamma \int \varepsilon_1 \psi _0 \Theta_0 ( y{\sqrt{\lambda}} )
+ \sin \gamma \int \varepsilon_2 \varphi _0 \Theta_0 ( y{\sqrt{\lambda}} ).
\end{equation}
Since $\varphi_0$ and $\psi_0$ need not be in $L^2$, we introduce the cut-off function $\Theta_0 ( y{\sqrt{\lambda}})$ in the definition of $\mathcal J$.
(Recall that the function $\Theta_0$ is defined in \eqref{def:Theta0}.)

\begin{lemma} \label{eestJ} 
Under the assumption~\eqref{eq:BS}, for $s$ large
\begin{equation} \label{eq:bdJ}
|\mathcal J| \lesssim K^\frac14 s^{-\frac {7}4}
\end{equation}
and
\begin{equation} \label{eq:bdJs}
\mathcal J_s = \langle \varepsilon_1 ,Q^4\psi_1\rangle\sin\gamma
+ \langle \varepsilon_2, Q^4\varphi_1 \rangle\cos\gamma + O( K^\frac14 s^{-\frac 52}).
\end{equation}
\end{lemma}

\begin{proof}
By~\eqref{eq:BS}, 
\begin{equation*} 
 |\mathcal J|  \lesssim \|\Theta_0 ( \cdot{\sqrt{\lambda}} )\|_{L^2} \| \varepsilon \| _{ L^2 }
 \lesssim \lambda ^{- \frac {1} {4}} \| \varepsilon \| _{ L^2 } \lesssim K^\frac14 s^{-\frac {7}4} ,
\end{equation*} 
which proves the first statement.

Next, setting $ \widetilde{ \Theta }_1 (s, y)= \Theta_0 ( y{\sqrt{\lambda}} ) $ and $ \widetilde{ \Theta }_2 (s, y)= y{\sqrt{\lambda}} \Theta_0 '( y{\sqrt{\lambda}} ) $ (so that it holds $y\partial _y \widetilde{ \Theta }_1 (y) = \widetilde{ \Theta }_2 (y)$),
we obtain by direct differentiation
\begin{equation*} 
\begin{aligned} 
\mathcal J _s = & \cos \gamma \int \Bigl( \partial _s \varepsilon_1 \psi _0 \widetilde{ \Theta }_1 + \frac {1} {2} \varepsilon_1 \frac {\lambda _s} {\lambda } \psi _0 \widetilde{ \Theta }_2 + \gamma _s \varepsilon_2 \varphi _0 \widetilde{ \Theta }_1 \Bigr) \\
& + \sin \gamma \int \Bigl( \partial _s \varepsilon_2 \varphi _0 \widetilde{ \Theta }_1 + \frac {1} {2} \varepsilon_2 \frac {\lambda _s} {\lambda } \varphi _0 \widetilde{ \Theta }_2 - \gamma _s \varepsilon_1 \psi _0 \widetilde{ \Theta }_1 \Bigr) \\
= & \cos \gamma \mathcal J _s^1 + \sin \gamma \mathcal J _s^2 .
\end{aligned} 
\end{equation*} 
Using~\eqref{eq:e}, \eqref{fortho3b}, the identity $\Lambda ( \psi _0 \widetilde{ \Theta }_1 )= (\Lambda \psi _0 ) \widetilde{ \Theta }_1 + \psi _0 \widetilde{ \Theta }_2$ and integration by parts, we deduce that 
\begin{equation*} 
\begin{aligned} 
\mathcal J _s^1 = & - \int \varepsilon_2 ( \partial _y^2 \psi _0 \widetilde{ \Theta }_1 - \psi _0 \widetilde{ \Theta }_1 + 2 \partial _y \psi _0 \partial _y \widetilde{ \Theta }_1 + \psi _0 \partial _y^2 \widetilde{ \Theta }_1 )\\ 
& - \int \Im ( f( V+\varepsilon) - f(V) ) \psi _0 \widetilde{ \Theta }_1 \\
& + \int (b-m_\lambda ) \varepsilon _1 [ (\Lambda \psi _0 ) \widetilde{ \Theta }_1 + \psi _0 \widetilde{ \Theta }_2 ] + \int m_\gamma \varepsilon_2 \psi _0 \widetilde{ \Theta }_1 \\
& - \int \Im S_0 \psi _0 \widetilde{ \Theta }_1 - \int \Im \mathcal R \psi _0 \widetilde{ \Theta }_1 + \frac {1} {2} \int \varepsilon_1 \frac {\lambda _s} {\lambda } \psi _0 \widetilde{ \Theta }_2 + \int \gamma _s \varepsilon_2 \varphi _0 \widetilde{ \Theta }_1 .
\end{aligned} 
\end{equation*} 
Using~\eqref{fedfpsiz}, $b-m_\lambda + \frac {\lambda _s} {\lambda } =0$, and $\gamma _s -1 = m_\gamma $,
we obtain
\begin{equation*} 
\begin{aligned} 
\mathcal J _s^1 = & \int \varepsilon_2 Q^4 \varphi _1 \widetilde{ \Theta }_1 - \int \Bigl( \Im ( f( V+\varepsilon) - f(V) )- Q^4 \varepsilon _2 \Bigr) \psi _0 \widetilde{ \Theta }_1\\ 
& - \int \varepsilon _2 ( 2 \partial _y \psi _0 \partial _y \widetilde{ \Theta }_1 + \psi _0 \partial _y^2 \widetilde{ \Theta }_1 ) \\
& + \int (b-m_\lambda ) \varepsilon _1 \Bigl( (\Lambda \psi _0 ) \widetilde{ \Theta }_1 + \frac {1} {2} \psi_0 \widetilde{ \Theta }_2 \Bigr) + \int m_\gamma \varepsilon_2 \psi _0 \widetilde{ \Theta }_1 \\
& - \int \Im \mathcal S_0 \psi _0 \widetilde{ \Theta }_1 - \int \Im \mathcal R \psi _0 \widetilde{ \Theta }_1 + \int m_\gamma \varepsilon_2 \varphi _0 \widetilde{ \Theta }_1 .
\end{aligned} 
\end{equation*}
We estimate, using \eqref{eq:BS} and $\varphi _0, \psi _0\in \mathcal Z_0$,
\[
\left|\int \varepsilon _2 \partial _y \psi _0 \partial _y \widetilde{ \Theta }_1 \right|
\lesssim \|\varepsilon\|_{L^2}\|\partial _y \psi _0\|_{L^2}\|\partial _y \widetilde{ \Theta }_1\|_{L^\infty}\lesssim \sqrt{\lambda} \|\varepsilon\|_{L^2}\lesssim K^\frac14 s^{-\frac52},
\]
and
\[
\left|\int\varepsilon _2   \psi _0 \partial _y^2 \widetilde{ \Theta }_1 \right|
\lesssim \|\varepsilon\|_{L^2}\|\psi _0\|_{L^\infty}\|\partial _y^2 \widetilde{ \Theta }_1\|_{L^2}
\lesssim \lambda^\frac34 \|\varepsilon\|_{L^2}\lesssim K^\frac14 s^{-\frac{11}4}.
\]
By similar estimates,
\[
\left|\int (b-m_\lambda ) \varepsilon _1 \Bigl( (\Lambda \psi _0 ) \widetilde{ \Theta }_1 + \frac {1} {2} \psi_0 \widetilde{ \Theta }_2 \Bigr) +\int m_\gamma \varepsilon_2 \psi _0 \widetilde{ \Theta }_1 +\int m_\gamma \varepsilon_2 \varphi _0 \widetilde{ \Theta }_1\right|
\lesssim K^{\frac {1} {4}} s^{- \frac {11} {4}}.
\]
Last, by~\eqref{eq:BS}, \eqref{eq:S0} and \eqref{eq:R1}
\[
\left| \int \Im S_0 \psi _0 \widetilde{ \Theta }_1\right| +\left| \int \Im \mathcal R \psi _0 \widetilde{ \Theta }_1 \right|
\lesssim (\|\mathcal S_0\|_{L^2}+ \|\mathcal R\|_{L^2}) \|\psi _0\|_{L^\infty}\|\widetilde{ \Theta }_1\|_{L^2}\lesssim s^{-\frac {11}4}.
\]
Thus,
\begin{equation*} 
\mathcal J _s^1 = \int \varepsilon_2 Q^4 \varphi _1 \widetilde{ \Theta }_1 - \int \Bigl( \Im ( f( V+\varepsilon) - f(V) )- Q^4 \varepsilon _2 \Bigr) \psi _0 \widetilde{ \Theta }_1 + O (K^{\frac {1} {4}} s^{- \frac 52}). 
\end{equation*} 
Similarly, 
\begin{equation*} 
\mathcal J _s^2 = \int \varepsilon_1 Q^4 \psi _1 \widetilde{ \Theta }_1 + \int \Bigl( \Re ( f( V+\varepsilon) - f(V) ) - 5Q^4 \varepsilon _1 \Bigr) \varphi _0 \widetilde{ \Theta }_1 + O (K^{\frac {1} {4}} s^{- \frac 52}). 
\end{equation*} 
We claim that
\begin{equation*} 
\int | f(V+\varepsilon ) - f(V) - [5 Q^4 \varepsilon _1 + i Q^4 \varepsilon _2] | \widetilde{ \Theta }_1 \lesssim K^{\frac {1} {4}} s^{-3},
\end{equation*} 
and the conclusion follows. The claim is an immediate consequence of the definition of $ \widetilde{ \Theta }_1 $ and the pointwise estimate
\begin{equation} \label{fTaylor1} \begin{aligned}
| f(V+\varepsilon ) - f(V) - [5 Q^4 \varepsilon _1 + i Q^4 \varepsilon _2] | 
&\lesssim K^{\frac {1} {4}} s^{-3} e^{- \frac { |y|} {2}} \\
&\quad + s^{-\frac {15} {2}} (\delta^{-5} + |y|^5 + s^{-5} |y|^{10} ) .
\end{aligned}\end{equation} 
Since $f(V+\varepsilon ) -f(V)= M _{ -b } [f(M_b V + M_b \varepsilon )- f(M_b V) ]$, it follows from formula~\eqref{eq:NL} that
\begin{equation*} 
\begin{aligned} 
f(V+\varepsilon ) - f(V) - [5 Q^4 \varepsilon _1 + i Q^4 \varepsilon _2] = & 5Q^4 [M _{ -b }E_1 - \varepsilon _1] + i Q^4 [M _{ -b }E_2 - \varepsilon _2 ] \\ & + M _{ -b } (Q_a^4 - Q^4 ) (5E_1 + i E_2) + \widetilde{\mathcal R}_3, 
\end{aligned} 
\end{equation*} 
where, from~\eqref{eq:NL},
\begin{align*}
M_b \widetilde{\mathcal R}_3 & = 10 Q_a^3 (( \widetilde{X} +E_1)^2- \widetilde{X} ^2) + 2 Q_a^3 (( \widetilde{Y} +E_2)^2- \widetilde{Y} ^2)\\
&\quad + 4i Q_a^3 (( \widetilde{X} +E_1)( \widetilde{Y} +E_2)- \widetilde{X} \widetilde{Y} ) + \widetilde{\mathcal R}_1
\end{align*}
and thus, using the expression of $\widetilde{\mathcal R}_1$,
\begin{equation*} 
 | \widetilde{\mathcal R}_3 | \lesssim e^{- |y|} ( | \widetilde{X} | + | \widetilde{Y} | + |E_1| + |E_2|)^2 + ( | \widetilde{X} | + | \widetilde{Y} | + |E_1| + |E_2|)^5. 
\end{equation*} 
Using the first estimate in~\eqref{fxye1e2}, we deduce that
\begin{equation*} 
 | \widetilde{\mathcal R}_3 | \lesssim s^{-3} e^{- \frac { |y|} {2}} + s^{-\frac {15} {2}} (\delta^{-5} + |y|^5 + s^{-5} |y|^{10} ) .
\end{equation*} 
Moreover, using~\eqref{festMbT} and~\eqref{festE1}, we see that
\begin{equation*} 
Q^4[M _{ -b }E_1 - \varepsilon _1] + Q^4[M _{ -b }E_2 - \varepsilon _2 | \lesssim s^{-1} \|\varepsilon \| _{ L^\infty } e^{- |y|} \lesssim K^{\frac {1} {4}} s^{-3} e^{-  |y| } .
\end{equation*} 
Finally, $ | Q_a^4 - Q^4 |\lesssim |a| e^{- |y|} \lesssim s^{-\frac {5} {2}} e^{- |y|} $, so that
\begin{equation*} 
| (Q_a^4 - Q^4 ) (5E_1 + i E_2)| \lesssim s^{-\frac {5} {2}} e^{- |y|} \|\varepsilon \| _{ L^\infty } \lesssim K^{\frac {1} {4}} s^{-\frac {9} {2}} e^{-  |y| }.
\end{equation*} 
This proves the estimate~\eqref{fTaylor1} and completes the proof.
\end{proof}

\begin{lemma} \label{eestma} 
Under the assumption~\eqref{eq:BS}, for $s$ large
\begin{equation} \label{eq:new}
 \Bigl| \frac {d} {ds} [ a + \alpha _5 \lambda ^{\frac {3} {2}} \mathcal J] -\Omega \Bigr| \lesssim  K^\frac12 s^{-4}.
\end{equation} 
\end{lemma} 

\begin{proof} 
We have
\begin{equation*} 
 \frac {d} {ds} [ a + \alpha _5 \lambda ^{\frac {3} {2}} \mathcal J ] = \Omega + m_a + \frac {3} {2} \alpha _5 \frac {\lambda _s} {\lambda } \lambda ^{\frac {3} {2}} \mathcal J + \alpha _5 \lambda ^{\frac {3} {2}} \mathcal J_s. 
\end{equation*} 
Since $ |\frac {\lambda _s} {\lambda }| \lesssim |b|+ |m_\lambda | \lesssim s^{-1}$, the result follows from Lemmas~\ref{eestm1} and~\ref{eestJ}. 
\end{proof} 

\section{Energy estimates}\label{S:5}
Consider a solution $u$ of \eqref{nls} as in Section \ref{S:4:1}, with $T=T_n=-\frac 1n$ as defined in \eqref{def:Tn} and
$S=S_n=n$. Assume 
that $u$ can be decomposed as in Section~\ref{S:4} and that
the bootstrap estimates \eqref{eq:BS} hold for $s\leq S_n$ close to $S_n$.
In particular \eqref{eq:BS1} and \eqref{eq:BS2} also hold, hence the estimates deduced in Lemma~\ref{le:V} and Proposition~\ref{PR:1} apply to the solution $u$. 
In this section, we derive energy estimates in the original variables $(t,x)$ rather than the rescaled variables
$(s,y)$. The main reason is the simplicity of equation~\eqref{eq:eta} below, compared to \eqref{eq:e}
which involves the operator $\Lambda$.
However, we will sometimes reintroduce the notation $\varepsilon$ to take advantage of the spectral properties of the operators $L_\pm$ and of the related orthogonality conditions \eqref{fortho2}; see Lemma~\ref{lem0}.

We let
\begin{align} 
v (t,x)&=\lambda^{-\frac 12}(s) V (s, y )\label{def:v}\\
\eta (t,x)&=\lambda^{-\frac 12}(s) \varepsilon (s, y ). \label{def:eta} 
\end{align} 
Note that
\begin{equation} \label{def:uveta} 
u(t,x)= e^{i \gamma (s(t))}(v(t,x) + \eta (t,x)). 
\end{equation} 
The equation for $\eta$ is now 
\begin{equation} \label{eq:eta} 
0 = i \partial _t \eta +\partial _{ xx } \eta - \frac {1} {\lambda ^2 (s(t))} \eta + f(v+ \eta ) - f(v) - \frac {m_\gamma (s(t))} {\lambda ^2 (s(t))} \eta + \mathcal Q_0 + \mathcal P,
\end{equation} 
where
\begin{equation} \label{def:PQ} 
\mathcal P (t,x)= \lambda ^{-\frac {5} {2}} \mathcal R (s,y) ,\quad \mathcal Q_0 (t,x) = \lambda ^{-\frac {5} {2}} \mathcal S_0 (s,y) .
\end{equation} 
Since $\eta\in \mathcal C(H^3)\cap \mathcal C^1(H^1)$, the above equation makes sense in $H^1$.

It follows from the change of variable~\eqref{def:eta}, 
the estimate on $t(s)$ written in \eqref{eq:t} and
the bootstrap assumption~\eqref{eq:BS} on $\varepsilon$ that
\begin{equation} \label{eq:BS3} 
\begin{aligned} 
 \| \eta \| _{ L^2 } & \lesssim K^{\frac {1} {4}} |t|^2 , \\
 \| \partial _x \eta \| _{ L^2 } & \lesssim K^{\frac {1} {4}} |t| , \\
 \| x \eta \| _{ L^2 } & \lesssim K^{\frac {1} {4}} |t|^2 . 
\end{aligned} 
\end{equation} 
In addition, it follows from~\eqref{def:v} and~\eqref{festV1} that
\begin{equation} \label{festv2} 
\begin{aligned} 
 |v| + |x| \, |\partial _x v| & \lesssim \lambda ^{-\frac {1} {2}}e^{- \frac { |x|} {2 \lambda }} + \delta \mathbf{1}_{ \{ |x| \le 2\delta \} } , \\
 |\partial _x v| & \lesssim \lambda ^{-\frac {3} {2}}e^{- \frac { |x|} {2 \lambda }} + \delta \lambda ^{-1 } \mathbf{1}_{ \{ |x| \le 2\delta \} } .
\end{aligned} 
\end{equation} 
In particular,
\begin{equation*} 
\lambda ^{\frac {1} {2}} \|v\| _{ L^\infty } +\lambda ^{\frac {1} {2}} \|x \partial_x v\| _{ L^\infty } + \| \, |x|^{\frac {1} {2}} v \| _{ L^\infty } + \|v\| _{ L^2 } + \| x \partial _x v\| _{ L^2 } \lesssim 1.
\end{equation*} 
Moreover, \eqref{def:PQ} and~\eqref{eq:R1}-\eqref{eq:R1b} yield
\begin{equation} \label{eq:RR} 
\begin{aligned} 
 \| \mathcal P\| _{ L^2 } + \| x\mathcal P\| _{ L^2 } &\lesssim |t| ,\\
 \| \partial _x \mathcal P\| _{ L^2 } &\lesssim 1 ,
\end{aligned} 
\end{equation} 
while~\eqref{def:PQ} and~\eqref{eq:S0} yield
\begin{equation} \label{eq:QQ} 
\begin{aligned} 
 \| \mathcal Q_0\| _{ L^2 } &\lesssim |t| ,\\
 \| x\mathcal Q_0\| _{ L^2 } &\lesssim |t|^2 ,\\
 \| \partial _x \mathcal Q_0\| _{ L^2 } &\lesssim 1 .
\end{aligned} 
\end{equation} 
We define
\begin{align*} 
\mathcal N &= \int x^2 |\eta| ^2 \\
\mathcal H & = \int \left\{ |\partial_x \eta |^2 +  {\lambda ^{-2}} |\eta |^2 
- 2 \left[ F(v+\eta )-F(v) - \Re (f(v) \overline{\eta } ) \right]\right\}\\
\mathcal K & = \Im \int x (\partial_x \eta )\bar \eta
\end{align*} 
and
\[
\mathcal G = \lambda ^2 \mathcal H + b \mathcal K + \frac {1} {4} \frac{b^2}{\lambda^2} \mathcal N.
\]
First, we provide direct upper bounds and coercivity estimates on the quantities $\mathcal N$,
$\mathcal H$ and $\mathcal K$. Second, we give estimates on the time derivatives of these quantities
using \eqref{eq:eta}.

\begin{lemma}\label{le:coercivity}
Assuming~\eqref{eq:BS},
\begin{equation}\label{eq:bounds}\begin{aligned}
\mathcal N (t) &\lesssim K^\frac12 |t|^4, \\
|\mathcal H (t)| &\lesssim K^\frac12 |t|^2,\\
|\mathcal K (t)| &\lesssim K^\frac12 |t|^3,
\end{aligned}\end{equation}
 for $t<T_n$ close to $T_n$.
Moreover, there exist a constant $\zeta>0$ such that
\begin{align}
 \zeta \left( \| \partial _x \eta(t) \| _{ L^2 }^2
 +\lambda^{-2}(t) \| \eta (t)\| _{ L^2 }^2\right) &\leq \mathcal H (t), \label{eenergy2:1} \\
\zeta \| \eta (t)\| _{ L^2 }^2 &\leq \mathcal G (t), \label{eenergy1:1} 
\end{align} 
 for $t<T_n$ close to $T_n$.
\end{lemma}

\begin{proof}
The estimates \eqref{eq:bounds} for $\mathcal N$ and $\mathcal K$ follow directly from \eqref{eq:BS3}.
Next, we note that
\[
\mathcal H = \lambda^{-2} \int \left\{ |\partial_y \varepsilon |^2 + |\varepsilon |^2 
- 2 \left[ F(V+\varepsilon )-F(V) - \Re (f(V) \overline{\varepsilon } ) \right]\right\}
\]
We claim that 
\begin{equation} \label{fspl1} 
\begin{aligned} 
 |V + \varepsilon |^6 &- |V|^6 - 6 \Re ( |V|^4 V \overline{\varepsilon } ) \\ &= 15 [\Re ( |V| V \overline{\varepsilon } )]^2 
+ 3 [\Im ( |V| V \overline{\varepsilon } )]^2 + O ( |\varepsilon |^3 |V|^3) + O ( |\varepsilon |^6). 
\end{aligned} 
\end{equation} 
This is immediate if $ |V| \le 2 |\varepsilon |$, and follows by developing $ |1 + \frac {\varepsilon } {V}|^6$ if $ |V| > 2\varepsilon $. 
Next, by~\eqref{fxye1e2b},
\begin{equation*} 
\begin{aligned} 
 |V| V \overline{\varepsilon } & = Q^2 \overline{\varepsilon } + |V| (V-Q) \overline{\varepsilon } + Q ( |V|-Q) \overline{\varepsilon } \\
& = Q^2 \overline{\varepsilon } + O(s^{-\frac {1} {2}} |\varepsilon |).
\end{aligned} 
\end{equation*} 
It follows that
\begin{equation*} 
\lambda^2 \mathcal H = \langle L_+\varepsilon _1, \varepsilon _1\rangle + \langle L_-\varepsilon _2, \varepsilon _2\rangle + O(s^{-\frac {1} {2}} \|\varepsilon \| _{ H^1 }^2) + O ( \|\varepsilon \| _{ H^1 }^3).
\end{equation*} 
This proves in particular the estimate~\eqref{eq:bounds} for $\mathcal H$.
Next, by~\eqref{fortho2},
\begin{equation*} 
\langle \varepsilon , Q\rangle = \langle \varepsilon , M _{ -b } Q \rangle + O( |b| \, \|\varepsilon \| _{ L^2 }) = 0 + O(s^{-1} \|\varepsilon \| _{ L^2 }).
\end{equation*} 
Similarly,
\begin{equation*} 
 |\langle \varepsilon , Q\rangle |+ |\langle \varepsilon , y^2 Q\rangle | + |\langle \varepsilon , i\rho \rangle | + |\langle \varepsilon , i\Lambda Q\rangle | \lesssim s^{-1} \|\varepsilon \| _{ L^2 } .
\end{equation*} 
By Lemma~\ref{lem0} and \eqref{eq:BS},
\[
\lambda^2 \mathcal H \geq \zeta_0 \|\varepsilon\|_{H^1}^2 + O (s^{-\frac 12} \|\varepsilon\|_{H^1}^2),
\]
and estimate~\eqref{eenergy2:1} follows for $s$ large, letting $\zeta=\frac 12 \zeta_0$.

We finally prove~\eqref{eenergy1:1}.
We observe first that
\[
\lambda ^2 \int |\partial_x \eta |^2 
+b \Im \int x (\partial_x \eta )\bar \eta + \frac 14 \frac {b^2}{\lambda^2} \int x^2 |\eta| ^2
=\lambda^2 \int \left| \partial_x \Bigl( \eta e^{i \frac{b}{4\lambda^2} x^2} \Bigr)\right|^2.
\]
Thus,
\begin{align*} 
\mathcal G
&= \lambda^2 \mathcal H + b \mathcal K + \frac {1} {4} \frac{b^2}{\lambda^2} \mathcal N\\
&=\lambda^2 \int \left\{\left| \partial_x \left( \eta e^{i \frac{b}{4\lambda^2} x^2} \right)\right|^2 
+ \lambda^{-2}|\eta |^2 - 2 \left[ F(v+\eta )-F(v) - \Re (f(v) \overline{\eta } ) \right]\right\}.
\end{align*} 
We let $\varepsilon _b= M_b \varepsilon $ and $V_b= M_b V$, so that 
\begin{align*} 
v(t,x)&=\lambda^{-\frac 12}(s) V (s, y)= \lambda^{-\frac 12}(s) M_{-b}(s,y) V_b (s, y)\\
\eta (t,x)&=\lambda^{-\frac 12}(s) \varepsilon (s, y )=\lambda^{-\frac 12}(s) M_{-b}(s,y)\varepsilon_b (s, y ), 
\end{align*}
and
\[
\mathcal G
= \int \left\{ | \partial_y \varepsilon_b|^2 
+ |\varepsilon_b |^2 - 2 \left[ F(V_b+\varepsilon_b )-F(V_b) - \Re (f(V_b) \overline{\varepsilon}_b )\right]\right\}.
\]
Since
\begin{equation*} 
 |\langle \varepsilon_b , Q\rangle |+ |\langle \varepsilon _b, y^2 Q\rangle | + |\langle \varepsilon _b, i\rho \rangle | + |\langle \varepsilon_b , i\Lambda Q\rangle | =0 
\end{equation*} 
by~\eqref{fortho2}, we deduce, following the proof of~\eqref{eenergy2:1}, that
\[
\mathcal G \geq \zeta \|\varepsilon_b\|_{H^1}^2 \geq \zeta \|\varepsilon_b\|_{L^2}^2
=\zeta \|\eta\|_{L^2}^2.
\]
This proves~\eqref{eenergy1:1}. (Note that $\mathcal G$ does not control $ \|\partial _x \eta \| _{ L^2 }^2$.)
\end{proof}

\begin{proposition} \label{eenergy1} 
Assuming~\eqref{eq:BS},
\begin{align} 
\Bigl|\frac {d} {dt} \mathcal N(t) \Bigr| & \lesssim \|\partial_x \eta(t)\|_{L^2} \|x\eta(t)\|_{L^2} + K^{\frac {1} {4}}|t|^3,\label{eenergy1:2t}\\
\Bigl| \frac {d} {dt} \mathcal H (t) \Bigr| & \lesssim |t|^{-3} \| \eta(t) \| _{ L^2 }^2 + K^{\frac {1} {4}}|t| ,\label{eenergy1:2} \\
\Bigl| \frac {d} {dt} \mathcal G (t) \Bigr| & \lesssim K^{\frac 14} |t|^3 ,
\label{eenergy1:2b} 
\end{align} 
 for $t<T_n$ close to $T_n$. 
\end{proposition} 

To prove Proposition \ref{eenergy1},
we will use the next lemmas.
In the following lemmas, we assume that $\varepsilon $ is sufficiently smooth, namely $\varepsilon \in \mathcal C([S_1, S_2], H^2 (\R ) ) \cap \mathcal C^1 ([S_1, S_2], L^2 (\R ) ) $ and $\varepsilon \in \mathcal C([S_1, S_2], L^2 (\R , |y|^2 dy) ) $, so that the calculations in the proof are valid. Of course, we will apply the lemma to appropriately smooth solutions. 

\begin{lemma} \label{lemma:n} 
Assuming~\eqref{eq:BS}, 
\begin{equation} \label{eenergy3:0}
\frac {d} {dt} \mathcal N = 4 \mathcal K + O (K^{\frac {1} {4}}|t|^3),
\end{equation} 
 for $t<T_n$ close to $T_n$. 
\end{lemma} 

\begin{lemma} \label{eenergy3} 
Assuming~\eqref{eq:BS}, 
\begin{equation} \label{eenergy3:1} 
\begin{aligned} 
\frac {d} {dt} \mathcal H (t) = & 2 \frac {b} {\lambda ^4} \| \eta \| _{ L^2 }^2 - 2 \frac {b} {\lambda ^2} \langle f( v + \eta ) - f(v) - 3 |v|^4 \eta - 2 |v|^2 v^2 \overline{\eta } , \Lambda v \rangle \\ & + O(K^{\frac {1} {4}} |t|),
\end{aligned} 
\end{equation}
 for $t<T_n$ close to $T_n$. 
In particular,
\begin{equation} \label{eenergy3:1b} 
 \Bigl| \frac {d} {dt} \mathcal H (t) \Bigr| \lesssim |t|^{-3} \| \eta \| _{ L^2 }^2 + K^{\frac {1} {4}} |t|.
\end{equation} 
\end{lemma}

\begin{lemma} \label{eenergy4}
Assuming~\eqref{eq:BS}, 
\begin{equation} \label{eenergy3:2} 
\begin{aligned} 
 \frac {d} {dt} \mathcal K& = 2 \int |\partial _x \eta |^2 -4 \Re \int [F(v+\eta ) - F(v) - f(v) \overline{\eta }] \\
& \quad + 2 \langle f( v + \eta ) - f(v) - 3 |v|^4 \eta - 2 |v|^2 v^2 \overline{\eta } , \Lambda v \rangle \\
& \quad + O ( K^{\frac {1} {4}} |t|^2), 
\end{aligned} 
\end{equation}
 for $t<T_n$ close to $T_n$.
\end{lemma} 

\begin{proof}[Proof of Lemma~$\ref{lemma:n}$]
We give the formal argument, the complete proof requires truncation of $x^2$, see for instance~\cite[Proof of Proposition~6.5.1]{CLN10}.
Multiplying the equation~\eqref{eq:eta} by $x^2 \overline{\eta} $ and taking the imaginary part, we obtain
\begin{equation*} 
\frac {1} {2} \frac {d} {dt} \mathcal N= 2 \mathcal K - \Im \int [f(v+ \eta ) - f(v)] x^2 \overline{\eta } - \Im \int [ \mathcal Q_0 + \mathcal P] x^2 \overline{\eta} ;
\end{equation*} 
and so, since $ |f(v+ \eta ) - f(v)|\lesssim ( |v|^4 + |\eta |^4) |\eta | $,
\begin{equation*} 
 \Bigl| \frac {d} {dt} \mathcal N - 4 \mathcal K \Bigr| \lesssim \| \, |x|^{\frac {1} {2}}v\| _{ L^\infty }^4 \|\eta \| _{ L^2 }^2 + \| \eta \| _{ L^\infty }^4 \mathcal N + ( \| x \mathcal Q_0 \| _{ L^2 } + \| x \mathcal P \| _{ L^2 } ) \sqrt{ \mathcal N } .
\end{equation*} 
Using~\eqref{eq:BS3}--\eqref{eq:QQ}, we deduce 
\begin{equation*} 
 \Bigl| \frac {d} {dt} \mathcal N - 4 \mathcal K \Bigr| \lesssim K^{\frac {1} {2}} |t|^4 + K^\frac32 |t|^8 + K^{\frac {1} {4}} |t|^3,
\end{equation*} 
which completes the proof.
\end{proof}

\begin{proof}[Proof of Lemma~$\ref{eenergy3}$]
Using the identities
\begin{equation} \label{fpreenergy3} 
\begin{aligned} 
 \partial_t F(v) & = \Re ( |v|^4 v \partial _t \overline{v} ) \\
 \partial_t F(v+ \eta ) & = |v+\eta |^4 \Re ( ( v+ \eta ) \partial _t (\overline{v} + \overline{\eta } ) )\\
 \partial_t (\Re (f(v) \overline{\eta } )) &= \Re [ (3 |v|^4 \eta + 2 |v|^2 v^2 \overline{\eta } ) \partial _t \overline{v} + |v|^4 v \partial _t \overline{\eta } ], 
\end{aligned} 
\end{equation} 
it follows from elementary calculations that (recall that $dt = \lambda^2 ds$)
\begin{equation*} 
\begin{aligned} 
\frac {d} {dt} \mathcal H = & - 2 \langle f( v + \eta ) - f(v) - 3 |v|^4 \eta - 2 |v|^2 v^2 \overline{\eta } , \partial _t v \rangle \\ & - 2 \langle \partial _{ xx } \eta - \lambda ^{-2} \eta + f(v+\eta ) - f(v) , \partial _t \eta \rangle \\
& -2 \frac {\lambda _s} {\lambda ^5} \| \eta \| _{ L^2 }^2 \\
& = : \mathcal H_1 + \mathcal H_2 + \mathcal H_3. 
\end{aligned} 
\end{equation*} 
We first estimate $\mathcal H_1$. 
Note that
\begin{equation*} 
 \partial_t v (t,x) = \lambda^{-\frac {5} {2}} \partial_sV (s,y) - \frac {\lambda _s} {\lambda } \lambda ^{-\frac {5} {2}} \Lambda V (s,y).
\end{equation*} 
Since $\Lambda V (s,y)= \lambda ^{\frac {1} {2}} \Lambda v(t,x)$, we deduce that
\begin{equation*} 
 \partial_t v (t,x) = \lambda^{-\frac {5} {2}} \partial _sV (s,y) - \frac {\lambda _s} {\lambda } \lambda ^{-2 } \Lambda v (t,x);
\end{equation*} 
and so (recall that $dx=\lambda dy$)
\begin{equation*} 
\mathcal H_1 = \mathcal H_4 +  2 \frac {\lambda _s} {\lambda } \lambda ^{-2 } \langle f( v + \eta ) - f(v) - 3 |v|^4 \eta - 2 |v|^2 v^2 \overline{\eta } , \Lambda v \rangle ,
\end{equation*} 
where
\begin{equation*} 
\mathcal H_4 = - 2 \lambda ^{-4} \langle f( V + \varepsilon ) - f(V) - 3 |V|^4 \varepsilon - 2 |V|^2 V^2 \overline{\varepsilon } , \partial _s V \rangle .
\end{equation*} 
We have 
\begin{equation} \label{eq:TY1} 
| f(V+\varepsilon ) - f(V) - 3 |V|^4 \varepsilon - 2 |V|^2 V^2 \overline{\varepsilon } | \lesssim (|\varepsilon |^3 + |V|^3) |\varepsilon |^2 ,
\end{equation} 
for all $V, \varepsilon \in \C$. 
(see the proof of~\eqref{fspl1}.) Thus by~\eqref{festV1} and $ \|\varepsilon \| _{ L^\infty }^3\lesssim K^{\frac {3} {4}} s^{-6} \lesssim s^{-\frac {3} {2}}$,
\begin{equation} \label{ftaylor11} 
| f(V+\varepsilon ) - f(V) - 3 |V|^4 \varepsilon - 2 |V|^2 V^2 \overline{\varepsilon } | \lesssim (e^{- \frac { |y|} {2}} + s^{-\frac 32} ) |\varepsilon |^2 .
\end{equation} 
Thus, using \eqref{dsV}, it follows that
\begin{equation*} 
 | \mathcal H_4 | \lesssim s^4 s^{- \frac {3} {2}} \int [ (1+ |y| {\mathbf 1}_I ) + \delta^{-1} \mathbf{1}_{J(s)} ] (e^{- \frac { |y|} {2}} + s^{-\frac 32} ) |\varepsilon |^2 .
\end{equation*}
Since $ | [ (1+ |y| {\mathbf 1}_I ) + \delta^{-1} \mathbf{1}_{J(s)} ] (e^{- \frac { |y|} {2}} + s^{-\frac 32} ) | \lesssim 1$, we deduce that
\begin{equation*} 
 | \mathcal H_4 | \lesssim s^{ \frac {5} {2}} \| \varepsilon \| _{ L^2 }^2 \lesssim K^{\frac {1} {2}} s^{- \frac {3} {2}} \lesssim s^{-1} \lesssim |t|,
\end{equation*} 
where we used~\eqref{eq:BS}. 

Next, using~\eqref{eq:eta} and 
\begin{equation*} 
\langle i [ \partial _{ xx } \eta - \lambda ^{-2} \eta + f(v+\eta ) - f(v)], \partial _{ xx } \eta - \lambda ^{-2} \eta + f(v+\eta ) - f(v) \rangle =0,
\end{equation*}
we obtain 
\begin{equation*} 
\begin{aligned} 
\frac {1} {2} \mathcal H_2 = & \lambda ^{-2} m_\gamma \langle \partial _{ xx } \eta - \lambda ^{-2} \eta + f(v+\eta ) - f( v ) , i \eta \rangle \\
& - \langle \partial _{ xx } \eta - \lambda ^{-2} \eta + f(v+\eta ) - f( v ) , i \mathcal P +i \mathcal Q_0 \rangle \\
= & \mathcal H_5 + \mathcal H_6.
\end{aligned} 
\end{equation*} 
We have
\begin{equation*} 
\mathcal H_5= \lambda ^{-2} m_\gamma \langle f(v + \eta ) - f(v) , i \eta \rangle .
\end{equation*} 
Since 
\begin{equation*} 
 | f(v + \eta ) - f(v) | \lesssim ( | v|^4 + |\eta |^4 ) |\eta | \lesssim \lambda ^{-2} |\eta |
\end{equation*} 
by~\eqref{eq:BS3}, \eqref{festv2} and \eqref{eq:BS}, we see that
\begin{equation*} 
 | \mathcal H_5 | \lesssim \lambda ^{-4} |m_\gamma | \, \| \eta \| _{ L^2 }^2 \lesssim K^{\frac {3} {2}} |t|^3 \lesssim |t|. 
\end{equation*} 
Furthermore, integrating by parts once the term $ \partial _{ xx } \eta $ and using the Cauchy-Schwarz inequality, we obtain
\begin{equation*} 
 | \mathcal H_6 | \lesssim \| \partial _x \eta \| _{ L^2 } (\| \partial _x \mathcal P \| _{L^2}
 +\|\partial_x \mathcal Q_0\|_{ L^2 } )
 + \lambda ^{-2} \| \eta \| _{ L^2 } (\| \mathcal P \| _{ L^2 }+\|\mathcal Q_0\|_{L^2} )\lesssim K^{\frac {1} {4}} |t|,
\end{equation*} 
where we used~\eqref{eq:BS3}, \eqref{eq:RR} and \eqref{eq:QQ} in the last inequality. 
(This term happens to determine the size of the error term in \eqref{eenergy3:1}.)

Putting together the previous identities and estimates, we have proved that
\begin{equation*} 
\begin{aligned} 
\frac {d} {dt} \mathcal H   = & -2 \frac {\lambda _s} {\lambda ^5} \| \eta \| _{ L^2 }^2 + 2 \frac {\lambda _s} {\lambda } \lambda ^{-2 } \langle f( v + \eta ) - f(v) - 3 |v|^4 \eta - 2 |v|^2 v^2 \overline{\eta } , \Lambda v \rangle \\ & + O(K^{\frac {1} {4}} |t| ). 
\end{aligned} 
\end{equation*} 
Since 
\begin{equation} \label{eq:lslb} 
\frac {\lambda _s} {\lambda } + b = O (K |t|^3)
\end{equation} 
by~\eqref{eq:BS}, it follows from~\eqref{eq:BS3} that
\begin{equation*} 
 -2 \frac {\lambda _s} {\lambda ^5} \| \eta \| _{ L^2 }^2 = 2 \frac {b} {\lambda ^4} \| \eta \| _{ L^2 }^2+ O (K^{\frac {3} {2}} |t|^3).
\end{equation*} 
Next, using~\eqref{eq:TY1}, then~\eqref{eq:BS3}-\eqref{festv2}, 
\begin{equation*} 
\begin{aligned} 
 | \langle f( v + \eta ) &- f(v) - 3 |v|^4 \eta - 2 |v|^2 v^2 \overline{\eta } , \Lambda v \rangle |\\ & \lesssim \int |\eta|^2 (|\eta |^3+|v|^3)( |v| + |x| \, |\partial _x v| )
 \\ & \lesssim 
\|\eta\|_{L^2}^2 \left(\|\eta\|_{L^\infty}^3+ \|v\|_{L^\infty}^3\right)\left(\|v\|_{L^\infty} + \|x \partial _x v\|_{L^\infty}\right)
\\ & \lesssim 
|t|^{-2}\|\eta\|_{L^2}^2 \lesssim K^{\frac {1} {2} } |t|^2 .
\end{aligned} 
\end{equation*} 
Using again~\eqref{eq:lslb}, we arrive at~\eqref{eenergy3:1}. 
The estimate~\eqref{eenergy3:1b} also follows. 
\end{proof}

\begin{proof}[Proof of Lemma~$\ref{eenergy4}$]
We first observe that
\begin{equation*} 
 \frac {d} {dt} \mathcal K = - 2 \Im \int \Lambda \overline{\eta } \partial _t \eta .
\end{equation*} 
It now follows from~\eqref{eq:eta} that
\begin{equation*} 
\begin{aligned} 
\frac {d} {dt} \mathcal K & = - 2 \Re \int \Lambda \overline{\eta } \partial _x ^2 \eta 
+2 \lambda ^{-2} \Re \int \Lambda \overline{\eta } \eta   - 2\Re \int \Lambda \overline{\eta }\left[f( v + \eta ) - f(v) \right] \\
&\quad + 2 \lambda ^{-2} m_\gamma \Re \int \Lambda \overline{\eta } \eta - 2 \Re \int \Lambda \overline{\eta } \mathcal Q_0
-2 \Re \int\Lambda \overline{\eta }\mathcal P .
\end{aligned} 
\end{equation*} 
Using \eqref{fortho3} and \eqref{fortho3c}, we deduce that
\begin{equation*} 
\begin{aligned} 
 \frac {d} {dt} \mathcal K & = 2 \int |\partial _x \eta |^2 - 2 \Re \int \Lambda \overline{\eta} \left[f( v + \eta ) - f(v) \right]   - 2 \Re \int \Lambda \overline{\eta} \mathcal Q_0
- 2 \Re \int \Lambda \overline{\eta} \mathcal P .
\end{aligned} 
\end{equation*} 
Using~\eqref{fortho3b} we deduce that 
\begin{equation*} 
 \frac {d} {dt} \mathcal K = 2 \int |\partial _x \eta |^2 - 2 \Re \int \Lambda \overline{\eta} \left[f( v + \eta ) - f(v) \right] 
 + 2 \Re \int \overline{\eta} \Lambda \mathcal Q_0
+ 2 \Re \int \overline{\eta} \Lambda \mathcal P .
\end{equation*} 
We write
\begin{equation*} 
\begin{aligned} 
 - 2 \Re \int \Lambda \overline{\eta} \left[f( v + \eta ) - f(v) \right] & = - \Re \int \overline{\eta} \left[f( v + \eta ) - f(v) \right] \\ & \quad- 2 \Re \int x \partial _x \overline{\eta} \left[f( v + \eta ) - f(v) \right] .
\end{aligned} 
\end{equation*} 
Note that
\begin{multline*} 
\Re \left[\partial _x ( F(v+\eta ) - F(v) - f(v) \overline{\eta } )\right]= \Re (\partial _x \overline{\eta } \left[ f(v+ \eta ) - f(v)\right] ) \\
 + \Re \left( \partial _x \overline{v} \left[ f(v+ \eta ) - f(v) - 3 |v|^4 \eta - 2 |v|^2 v^2 \overline{\eta } \right] \right);
\end{multline*} 
and so, multiplying by $x$ and integrating by parts
\begin{multline*} 
- 2 \Re \int x \partial _x \overline{\eta} \left[f( v + \eta ) - f(v) \right]
= 2 \Re \int \left[F(v+\eta ) - F(v) - f(v) \overline{\eta }\right] \\
 + 2 \Re \int x \partial _x \overline{v} \left[ f(v+ \eta ) - f(v) - 3 |v|^4 \eta - 2 |v|^2 v^2 \overline{\eta } \right].
\end{multline*} 
We rewrite this last equation in the form
\begin{multline*} 
- 2 \Re \int x \partial _x \overline{\eta} \left[f( v + \eta ) - f(v)\right]
= 2 \Re \int \left[F(v+\eta ) - F(v) - f(v) \overline{\eta }\right] \\
 + \Re \int [2 \Lambda \overline{v}- \overline{v} ] \left[ f(v+ \eta ) - f(v) - 3 |v|^4 \eta - 2 |v|^2 v^2 \overline{\eta } \right].
\end{multline*} 
Thus we see that
\begin{equation*} 
\begin{aligned} 
 - 2 \Re \int \Lambda \overline{\eta} \left[f( v + \eta ) - f(v) \right] & = - \Re \int \overline{\eta} \left[f( v + \eta ) - f(v) \right] \\ & \quad + 2 \Re \int \left[F(v+\eta ) - F(v) - f(v) \overline{\eta }\right] \\
& \quad- \Re \int \overline{v} \left[ f(v+ \eta ) - f(v) - 3 |v|^4 \eta - 2 |v|^2 v^2 \overline{\eta } \right] \\
& \quad+ 2 \Re \int \Lambda \overline{v} \left[ f(v+ \eta ) - f(v) - 3 |v|^4 \eta - 2 |v|^2 v^2 \overline{\eta } \right],
\end{aligned} 
\end{equation*}
which we rewrite in the form
\begin{multline*} 
 - 2 \Re \int \Lambda \overline{\eta} \left[f( v + \eta ) - f(v) \right] = -4 \Re \int \left[F(v+\eta ) - F(v) - f(v) \overline{\eta }\right] \\
 + 2 \Re \int \Lambda \overline{v} \left[ f(v+ \eta ) - f(v) - 3 |v|^4 \eta - 2 |v|^2 v^2 \overline{\eta } \right].
\end{multline*}
Thus we arrive at the expression
\begin{equation*} 
\begin{aligned} 
 \frac {d} {dt} \mathcal K & = 2 \int |\partial _x \eta |^2 -4 \Re \int \left[F(v+\eta ) - F(v) - f(v) \overline{\eta }\right] \\
& \quad + 2 \Re \int \Lambda \overline{v} \left[ f(v+ \eta ) - f(v) - 3 |v|^4 \eta - 2 |v|^2 v^2 \overline{\eta } \right] \\
&\quad + 2 \Re \int \overline{\eta} \Lambda \mathcal Q_0
+ 2 \Re \int \overline{\eta} \Lambda \mathcal P .
\end{aligned} 
\end{equation*} 
On the other hand, using~\eqref{eq:BS3}, \eqref{festv2} and~\eqref{eq:RR}, we obtain 
by the Cauchy-Schwarz inequality
\begin{equation*} 
 2 \Re \int \overline{\eta} \Lambda \mathcal Q_0
+ 2 \Re \int \overline{\eta} \Lambda \mathcal P = O (K^{\frac {1} {4}} |t|^2), 
\end{equation*} 
which completes the proof. 
\end{proof}

\begin{proof}[End of the proof of Proposition~$\ref{eenergy1}$]
The estimate \eqref{eenergy1:2t} is a consequence of \eqref{eenergy3:0} and
the estimate \eqref{eenergy1:2} is already proved in \eqref{eenergy3:1b}.

We finally prove \eqref{eenergy1:2b}.
By the expression of $\mathcal G = \lambda ^2 \mathcal H + b \mathcal K + \frac {1} {4} \frac{b^2}{\lambda^2} \mathcal N$,
$dt = \lambda^{2} ds$,
and the definitions of $m_\lambda$ and $m_b$, we have
\begin{align*}
\frac{d\mathcal G}{dt} &= \lambda^2 \frac{d\mathcal H}{dt} + b \frac{d\mathcal K}{dt} + \frac 14\frac{b^2}{\lambda^2} \frac{d\mathcal N}{dt} - 2b \mathcal H - \frac{b^2}{\lambda^2} \mathcal{K} \\
&\quad +2 m_\lambda \mathcal H+\frac{m_b+a}{\lambda^2} \mathcal{K}
+ \frac 12 \frac1{\lambda^2}\Bigl( \frac{a}{\lambda}+ \frac {m_b} {\lambda }- m_\lambda \frac {b} {\lambda } \Bigr) \frac b\lambda \mathcal{N},
\end{align*}
where the last term comes from the identity
\begin{equation} \label{eq:boverl} 
\frac{d}{ds} \left( \frac {b}{\lambda} \right)=
\frac{b_s - \frac{\lambda_s}{\lambda} b}{\lambda} 
= \frac{a}{\lambda}+ \frac {m_b} {\lambda }- m_\lambda \frac {b} {\lambda },
\end{equation} 
see~\eqref{def:m}. 
Combining \eqref{eenergy3:0}, \eqref{eenergy3:1}, \eqref{eenergy3:2}
with the expressions of $\mathcal H$ and $\mathcal K$, 
the resulting cancellations imply that
\[
\lambda^2 \frac{d\mathcal H}{dt} + b \frac{d\mathcal K}{dt} + \frac 14\frac{b^2}{\lambda^2} \frac{d\mathcal N}{dt} - 2b \mathcal H - \frac{b^2}{\lambda^2} \mathcal{K}
= 0+ O(K^\frac14 |t|^3).
\]
By \eqref{eq:bounds} and \eqref{festm1},
\[
\left|m_\lambda\right| \left| \mathcal H\right|
+\left| \frac{m_b}{\lambda^2}\right| \left| \mathcal{K}\right|
\lesssim K^\frac12 |t|^4 ,
\]
and
\begin{equation*} 
\left| \frac {a} {\lambda ^2} \mathcal K \right| 
+\frac 1{\lambda^2} \left| \frac{a}{\lambda}+ \frac {m_b} {\lambda }- m_\lambda \frac {b} {\lambda } \right| \left| \frac b\lambda\right|\left|\mathcal{N}\right|\lesssim K^\frac12 |t|^{\frac {7} {2}}.
\end{equation*} 
Thus,  we obtain~\eqref{eenergy1:2b}.
\end{proof}

\section{Proof of the uniform estimates}\label{S:6}

We consider $n\geq s_0$, where $s_0$ will be chosen sufficiently large later.
Recall from \eqref{def:Tn} that
$T_n=-\frac 1n$ and $S_n = n$.
For $\beta \in (-1, 1)$, we set 
\begin{equation} \label{eq:at:S}
\begin{aligned}
&\varepsilon_n^{in} = 0 , \quad \gamma_n^{in} =n, \quad a_n^{in} = \alpha _1 n^{-\frac {5} {2}} \sin n, \\
& \lambda_n^{in} = b_n^{in} = n^{-1} - \alpha_1 n^{-\frac52}\cos n + \beta n^{- \frac {7} {4}} ,\\
& \Gamma_n^{in} = (\gamma_n^{in},\lambda_n^{in},b^{in},a_n^{in},0,0,0),
\end{aligned}
\end{equation} 
and
\begin{equation}\label{def:un:in}
u_n^{in} (x) = \lambda_n ^{-\frac 12} e^{i\gamma_n}V[\Gamma_n^{in}](y),\quad y = \frac {x}{\lambda_n },
\end{equation}
where $V[\Gamma_n^{in}]$ is the blow-up approximate solution defined in \eqref{eq:blowup}
with the parameter $\Gamma_n^{in}$ defined in \eqref{eq:at:S}
(recall that $V$ depends on the constant $\delta=K^{-2}$).
We consider the solution $u_n$ of \eqref{nls} with initial data $u_n(T_n)=u_n^{in}$.
Using Lemma \ref{le:CL} with $u(t,x)=u_n(T_n+t,x)$ and $S=S_n$, there exist a time $\bar s_n<S_n$ and a 
$\mathcal C^1$ function $\Gamma$ such that $u_n$ decomposes as in \eqref{def:u} 
with \eqref{fortho33}. Since $\varepsilon_n(S_n)=\varepsilon_n^{in} = 0$, integrating the first four equations in \eqref{fortho33}, we obtain the orthogonality relations \eqref{fortho2}.
Note that the bootstrap estimate \eqref{eq:BS} is satisfied with strict inequalities at $s=S_n$ (for $n$ large) and thus by continuity, it is also satisfied on $[\bar s_n,S_n]$ after possibly increasing $\bar s_n<S_n$. Note that there exists $\tau_K$ large depending only on $K$ such that if \eqref{eq:BS} is satisfied on a time interval $[\tau,S_n]$, for $\tau_K< \tau<S_n$, then
the asssumption
\[
|(\lambda(s),b(s),a(s),j_1(s),j_2(s),j_3(s))|+|d(u_n[\Gamma(s)];\Gamma(s))| < \omega,
\]
 of Lemma \ref{le:CL} is satisfied for all $s\in [\tau,S_n]$ so that the decomposition can be extended to $[\tau',S_n]$ for some $\tau_K<\tau'<\tau$. 
Thus, assuming $s_0\geq \tau_K$, we may define 
\[
s_*=s_*(\beta,n) = \inf \{ s\in [s_0,S_n]: \mbox{\eqref{eq:BS} holds for $u_n$ on $[s,S_n]$}\}.
\]
In particular, possibly taking $s_0$ larger, depending on $K$, we may apply the estimates proved in Sections \ref{S:3}, \ref{S:4} and
\ref{S:5} to the solution $u_n$ on the rescaled time interval $[s_*,S_n]$

The next proposition, which is the main result of this section, shows that for $K$ large, 
there exists at least a value of $\beta=\beta_n$ such that 
\eqref{eq:BS} holds on $[s_0,S_n]$ where $s_0$ is independent of $n$.

\begin{proposition}\label{pr:boot}
For all sufficiently large $K$, 
there exists $s_0\geq \tau_K$ such that for all $n>s_0$, it holds
$s_*(\beta_n,n)=s_0$ for some $\beta_n\in (-1,1)$.
\end{proposition}

The rest of this section is devoted to the proof of Proposition~\ref{pr:boot}.
Recall that the estimate of $\vec m$ in \eqref{eq:BS} was strictly improved for $K$ large 
by Lemma~\ref{eestm1}.
In Section \ref{S:6.1}, we improve the bootstrap estimate of $\varepsilon$ in \eqref{eq:BS}.
In Section \ref{S:6.2}, we improve the bootstrap estimates of $\gamma$, $a$ and $\frac b\lambda$ in \eqref{eq:BS}.
Last, in Section \ref{S:6.3}, we prove by a contradiction argument that there exists at least one value of
$\beta\in (-1,1)$ such that $s_*(\beta,n)=s_0$.

\subsection{Closing the bootstrap estimates on the error term}\label{S:6.1}

In this subsection, we prove the following result.
\begin{lemma}\label{le:onEPS}
For all $s\in [s_*,S_n]$,
\begin{equation} \label{eq:BS3improved} 
\begin{aligned} 
\|\varepsilon(s)\|_{H^1}& \lesssim K^\frac18 s^{-2}, \\
 \| y \varepsilon(s) \| _{ L^2 } & \lesssim K^{\frac {3} {16}} s^{-1}.
\end{aligned}
\end{equation} 
\end{lemma}
\begin{proof}
Define $t_*<0$ such that $T_n-t_*=\int_{s_*}^{S_n} \lambda^2$.
Let $t\in [t_*,T_n]$.
Integrating \eqref{eenergy1:2b} on $[t,T_n]$, using $\mathcal G (T_n)=0$, we obtain
\[
|\mathcal G(t)| \lesssim K^{\frac 14} |t|^4 .
\]
Thus, by \eqref{eenergy1:1},
\[ 
\|\eta(t)\|_{L^2}^2 \lesssim \mathcal G(t) \lesssim K^{\frac 14} |t|^4.
\]
Inserting this estimate in~\eqref{eenergy1:2}, we obtain
\[
\Bigl| \frac {d} {dt} \mathcal H (t) \Bigr|
 \lesssim K^{\frac {1} {4}} |t| .
\]
Integrating the latter estimate on $[t,T_n]$, using $\mathcal H (T_n)=0$, we obtain
\[
|\mathcal H(t)| \lesssim K^{\frac 14} |t|^2.
\]
By \eqref{eenergy2:1}, this gives
\[
\|\partial_x \eta(t)\|_{L^2}^2 \lesssim K^{\frac 14} |t|^2.
\]
Thus, by change of variable $\|\varepsilon(s)\|_{H^1}\lesssim K^\frac18 s^{-2}$
on $[s_*,S_n]$.
We continue using \eqref{eenergy1:2t}, the above estimate and \eqref{eq:BS3} for $\|x\eta\|_{L^2}$,
\[
\left|\frac {d} {dt} \mathcal N \right| 
\lesssim \|\partial_x \eta\|_{L^2} \|x\eta\|_{L^2} + K^{\frac {1} {4}} |t|^3
\lesssim K^{\frac 38} |t|^3.
\]
Integrating this estimate on $[t,T_n]$, using $\mathcal N (T_n)=0$, we obtain
\[
\mathcal N(t)=\|x \eta(t)\|_{L^2}^2  \lesssim K^{\frac 38} |t|^4.
\]
The proof of the lemma is complete.
\end{proof}

\subsection{Closing the parameter estimates}\label{S:6.2}
 
\begin{lemma}\label{le:onPARAM}
For all $s\in [s_*,S_n]$, 
\begin{equation} \label{eq:BSimproved}
\left\{\begin{aligned}
|\gamma(s) -s |&\lesssim s^{-2},\\
|a(s) - \alpha_1 s^{-\frac 52} \sin \gamma(s) |&\lesssim K^\frac12 s^{-3},\\
\left|\frac{b(s)}{\lambda(s)} - 1\right|&\lesssim K^\frac12 s^{-1}.\\
\end{aligned}\right.
\end{equation}
\end{lemma}

\begin{proof}
Integrating on $[s,S_n]$ the estimate $|\gamma_s-1|\lesssim s^{-3}$ in \eqref{festm1} and using $\gamma_n^{in}=n=S_n$, we obtain
\[
|\gamma(s) - s|\lesssim s^{-2},
\] 
which is the estimate for $\gamma$ in \eqref{eq:BSimproved}.

Now, by \eqref{eq:new}, we have
\[
\frac {d} {ds} [ a + \alpha _5 \lambda ^{\frac {3} {2}} \mathcal J]= \Omega + O( K^\frac12 s^{-4}).
\]
Integrating, using 
$\mathcal J(S_n)=0$, then~\eqref{eq:bdJ},
\begin{equation}\label{triplegal} 
 \Bigl| a({S_n}) - a(s) - \int _s^{S_n} \Omega \Bigr| \lesssim \lambda ^{\frac {3} {2}} | \mathcal J(s) | + K^\frac12 s^{-3} \lesssim K^\frac12 s^{-3}.
\end{equation} 
We now calculate $\int \Omega $ using the definition of $\Omega$ in~\eqref{eq:O}. 
First, by $1=\gamma_s - m_\gamma$
\begin{equation*} 
\begin{aligned} 
\int _s ^{S_n} b \lambda^\frac32 \cos \gamma &= - \int _s ^{S_n} m_\gamma b \lambda^\frac32 \cos \gamma + \int _s ^{S_n} b \lambda^\frac32 \frac {d} {ds} \sin \gamma \\ 
& = O(s^{- \frac {9} {2}} ) + \int _s ^{S_n} b \lambda^\frac32 \frac {d} {ds} \sin \gamma .
\end{aligned} 
\end{equation*} 
Next, 
\begin{equation*} 
\int _s ^{S_n} b \lambda^\frac32 \frac {d} {ds} \sin \gamma = [ b \lambda^\frac32 \sin \gamma] (S_n) - [ b \lambda^\frac32 \sin \gamma] (s) - 
\int _s ^{S_n} \frac {d} {ds} ( b \lambda^\frac32 ) \sin \gamma ;
\end{equation*} 
then, using~\eqref{festm1} and~\eqref{eq:BS} (in particular, $|a|\lesssim s^{-\frac 52}$)
\begin{equation*} 
\begin{aligned} 
\int _s ^{S_n} \frac {d} {ds} ( b \lambda^\frac32 ) \sin \gamma & = - \frac {5} {2} \int _s ^{S_n} b^2 \lambda^\frac32 \sin \gamma 
+ \int _s ^{S_n} \lambda^\frac32 \Bigl( a + m_b + \frac {3} {2} b m_\lambda \Bigr) \sin \gamma \\
& = - \frac {5} {2} \int _s ^{S_n} b^2 \lambda^\frac32 \sin \gamma + O(s^{-3}). 
\end{aligned} 
\end{equation*} 
We integrate by parts again
\begin{equation*} 
\begin{aligned} 
\int _s ^{S_n} b^2 \lambda^\frac32 \sin \gamma &= - \int _s ^{S_n} m_\gamma b^2 \lambda^\frac32 \sin \gamma - \int _s ^{S_n} b^2 \lambda^\frac32 \frac {d} {ds} \cos \gamma \\ 
& = O(s^{- \frac {11} {2}} ) - \int _s ^{S_n} b^2 \lambda^\frac32 \frac {d} {ds} \cos \gamma .
\end{aligned} 
\end{equation*} 
Next, 
\begin{equation*} 
\begin{aligned} 
\int _s ^{S_n} b^2 \lambda^\frac32 \frac {d} {ds} \cos \gamma & = [ b^2 \lambda^\frac32 \cos \gamma] ({S_n}) - [ b^2 \lambda^\frac32 \cos \gamma] (s) - \int _s ^{S_n} \frac {d} {ds} ( b^2 \lambda^\frac32 ) \cos \gamma \\
& = O (s^{- \frac {7} {2}}) - \int _s ^{S_n} \frac {d} {ds} ( b^2 \lambda^\frac32 ) \cos \gamma ;
\end{aligned} 
\end{equation*}
then, using~\eqref{eq:BS} and~~\eqref{festm1} 
\begin{equation*} 
\begin{aligned} 
\int _s ^{S_n} \frac {d} {ds} ( b^2 \lambda^\frac32 ) \cos \gamma & = - \frac {7} {2} \int _s ^{S_n} b^3 \lambda^\frac32 \cos \gamma 
+ \int _s ^{S_n} b \lambda^\frac32 \Bigl( 2 a + 2 m_b + \frac {3} {2} b m_\lambda \Bigr) \cos \gamma \\
& = O(s^{- \frac {7} {2}}). 
\end{aligned} 
\end{equation*} 
Collecting the above identities, we obtain
\begin{equation*} 
\int _s ^{S_n} b \lambda^\frac32 \cos \gamma = [ b \lambda^\frac32 \sin \gamma] ({S_n}) - [ b \lambda^\frac32 \sin \gamma] (s) + O (s^{-3} ). 
\end{equation*} 
The other three terms in~\eqref{eq:O} can also be calculated by integration by parts (once only) and they are all $O(s^{-3})$.
Hence, in the end we obtain
\begin{equation*} 
\begin{aligned} 
\int _s ^{S_n}\Omega &= \alpha _1 [ b \lambda^\frac32 \sin \gamma] ({S_n}) - \alpha _1 [ b \lambda^\frac32 \sin \gamma] (s) + O (s^{-3} )\\
&= \alpha _1 S_n^{ - \frac52 } \sin \gamma ({S_n}) - \alpha _1 s^{ - \frac52 } \sin \gamma (s) + O (s^{-3} ),
\end{aligned} 
\end{equation*} 
where we used~\eqref{eq:BS} in the last line.
Using now
\begin{equation*} 
a({S_n})= \alpha _1 S_n^{-\frac {5} {2}} \sin \gamma ({S_n}),
\end{equation*} 
as specified in \eqref{eq:at:S}, we finally obtain from \eqref{triplegal},
\begin{equation}\label{eq:a:improved}
|a(s) - \alpha_1 s^{-\frac 52} \sin \gamma(s) |\lesssim K^\frac12 s^{-3},
\end{equation}
which is the desired estimate related to $a$ in \eqref{eq:BSimproved}.

Last, we observe that by the calculation~\eqref{eq:boverl} and~\eqref{festm1} 
\[
\frac{d}{ds} \left( \frac {b}{\lambda} \right)= \frac{a}{\lambda}+O(s^{-2}),
\]
so that using~\eqref{eq:a:improved} 
\[
\left| \frac{d}{ds} \left( \frac {b}{\lambda} \right) - \alpha_1 s^{-\frac 32} \sin \gamma\right| \lesssim 
K^\frac12 s^{-2}.
\]
Integrating on $[s,{S_n}]$ and using $\lambda({S_n})= b({S_n})$ from \eqref{eq:at:S}, we obtain
\[
\left| \frac {b(s)}{\lambda(s)} -1 \right| \lesssim K^\frac12 s^{-1} + \Bigl|\int _s^{S_n}  (s')^{-\frac 32} \sin \gamma (s' ) \Bigr|.
\]
We observe that, using $1= \gamma_s - m_\gamma$ and then integrating by parts
\begin{equation*} 
\begin{aligned} 
\int _s^{S_n} (s') ^{-\frac 32} \sin \gamma (s' ) ds'& = - \int _s^{S_n} m_\gamma (s')^{-\frac 32} \sin \gamma (s' ) ds'\\&\quad + \int _s^{S_n} (s')^{-\frac 32} \gamma_s (s' ) \sin \gamma (s' ) ds'\\
& = - \int _s^{S_n} m_\gamma (s')^{-\frac 32} \sin \gamma (s' )ds'
- \frac {3} {2} \int _s^{S_n} (s')^{-\frac 52} \cos \gamma (s' ) ds'\\
&\quad - S_n^{-\frac 32} \cos\gamma({S_n}) + s^{-\frac 32} \cos\gamma(s);
\end{aligned} 
\end{equation*} 
and so, applying~\eqref{festm1} 
\begin{equation} \label{es:osc} 
 \Bigl| \int _s^{S_n}(s')^{-\frac 32} \sin \gamma (s' ) ds'\Bigr| \lesssim s^{-\frac {3} {2}}.
\end{equation} 
Thus, we have proved
\begin{equation*} 
\left| \frac {b(s)}{\lambda(s)} -1 \right| \lesssim K^\frac12 s^{-1} ,
\end{equation*} 
which is the estimate related to $\frac b\lambda$ in \eqref{eq:BSimproved}.
\end{proof}

Using Lemmas~\ref{eestm1}, \ref{le:onEPS} and \ref{le:onPARAM}, 
we see that if $K$ is sufficiently large (independent of $n$), then
on $[s_*,S_n]$,
\begin{equation} \label{eq:BSimproved2}
\left\{\begin{aligned}
|\gamma(s) -s |&\leq \frac 12 s^{-\frac 74},\\
\left|\frac{b(s)}{\lambda(s)} - 1\right|&\leq \frac 12 K  s^{-1},\\
|a(s) - \alpha_1 s^{-\frac 52} \sin \gamma(s) |&\leq \frac 12 K s^{-3},\\
|\vec m(s)| & \leq \frac 12 K s^{-3},\\
\|\varepsilon(s)\|_{H^1}& \leq \frac 12 K^\frac14 s^{-2}, \\
 \| y \varepsilon(s) \| _{ L^2 } & \leq \frac 12 K^{\frac {1} {4}} s^{-1},
\end{aligned}\right.
\end{equation}
which means that
all the estimates in \eqref{eq:BS} are strictly improved except the estimate of $b$.
We now consider a value of $K$ sufficiently large so that \eqref{eq:BSimproved2} holds
on $[s_*,S_n]$.

\subsection{Closing the bootstrap by a contradiction argument}\label{S:6.3}
Closing the 
estimate of $b$ in~\eqref{eq:BS}  requires a specific contradiction argument related to the choice of the free parameter $\beta\in (-1,1)$.

We set
\[
g(s)=g_\beta(s) = s^\frac74 \left( b(s) -s^{-1} + \alpha_1 s^{-\frac52}\cos \gamma(s)\right)
\]
so that
\[
|g(s)|\leq 1 \iff \bigl| b(s) -s^{-1} + \alpha_1 s^{-\frac52}\cos \gamma(s)\bigr| \leq s^{-\frac 74}.
\]
Hence, by \eqref{eq:at:S} and \eqref{eq:BS},
\[
g({S_n}) = \beta \quad \mbox{and}\quad |g(s)|\leq 1.
\]
Using \eqref{festm1} and the estimate on $a$ in \eqref{eq:BS}, we have
\begin{align*}
b_s+ b^2 - \alpha_1 s^{-\frac 52} \sin \gamma 
= m_b + a - \alpha_1 s^{-\frac 52} \sin \gamma= O( K s^{-3}).
\end{align*}
Thus, by direct computation, and \eqref{eq:BS}
\begin{align*}
g_s & = \frac 74 s^{-1} g + s^{\frac 74} \left( -b^2 + s^{-2} - \alpha_1  m_\gamma s^{-\frac 52} \sin\gamma
-\frac 52 \alpha_1 s^{-\frac 72} \cos \gamma + O(Ks^{-3})\right)\\
& = - \frac g {4s}+ O(K s^{- \frac 54}),
\end{align*}
where we have used, by \eqref{eq:BS},
\[
s^\frac74 (-b^2+s^{-2}) = - s^\frac74  (b+s^{-1})(b-s^{-1})= -\frac 2 s g+ O(s^{-\frac 74}).
\]
This estimate implies the following properties, for $s$ large,
\begin{equation}\label{eq:exit}
g(s)=1 \implies g_s(s)<0 \quad \mbox{and}\quad
g(s)=-1 \implies g_s(s)>0.
\end{equation}

Here, $s_0$ large is fixed so that all the previous estimates hold, and we work for any $n>s_0$. For the sake of contradiction,
assume that for any $\beta\in (-1,1)$, 
$s_*(\beta,n)=s_*(\beta)>s_0$.
We claim now that
\[
|g(s_*(\beta))|=1.
\]
Indeed, if $|g(s_*(\beta))|<1$, then by \eqref{eq:BSimproved2}, the bootstrap
estimate \eqref{eq:BS} is satisfied with strict inequalities at $s=s_*$,
which is a contradiction with the definition of $s_*$ by continuity.
By \eqref{eq:exit}, the map $s_*$ is continuous on $(-1,1)$.
We deduce that the map 
\[g\circ s_*:
\beta\in (-1,1) \mapsto g\circ s_*(\beta)\in \{-1,1\}\]
is continuous. Moreover, from \eqref{eq:exit}, for $\beta$ close to $1$, $g(s_*(\beta))=1$ and
for $\beta$ close to $-1$, $g(s_*(\beta))=-1$. Thus,
\[g\circ s_*((-1,1)) =\{-1,1\},\]
which is a contradiction since $(-1,1)$ is connected.
Therefore, for any $n$ large, there exists $\beta_n\in (-1,1)$ such that
for any $s\in [s_0,{S_n}]$, $|g(s)|< 1$, which completes the proof of 
Proposition~\ref{pr:boot}.

\section{Proof of the main result}\label{S:7}

Let 
\begin{equation}\label{def:r:s}
r_* (x) = \Theta(x) \left( |x|+i\kappa |x|^2\right).
\end{equation}
It follows that $r_*(x) = |x|+ i\kappa |x|^2$ on $(-\delta,\delta)$ and $\|r_*\|_{L^2}\leq C \delta^\frac32$, for some constant independent of $\delta$.
In this section, we consider $K$ sufficiently large to satisfy the assumptions of  Proposition~\ref{pr:boot}. This corresponds to choosing $\delta>0$ sufficiently small.

In the rest of this section, the implicit constants in inequalities may depend on $\delta$, but are independent of $n$.

\subsection{Construction of a sequence of solutions}
In the next proposition, we construct a sequence of solutions $(u_n)$ of \eqref{nls}, defined
on $[t_1,T_n]$, where $t_1<0$ is independent of $n$, satisfying uniform estimates related to the 
blow-up behavior expected in Theorem \ref{theorem1}.

\begin{proposition} \label{pr:unif}
 There exists $t_1<0$ such that, for all $n>1$ large,
there exists $\beta_n\in (-1,1)$ with the following property.
If $u_n$ be the solution of \eqref{nls} such that $u_n(T_n)=u_n^{in}$, 
where $u_n^{in}$ is defined in \eqref{eq:at:S}--\eqref{def:un:in},
then $u_n$ exists on $[t_1,T_n]$ and satisfies the following uniform estimates
\begin{align} 
\biggl\|u_n(t)-\frac{e^{\frac it}}{|t|^{\frac12}} Q\left(\frac \cdot t\right) - r_*\biggr\|_{L^2} &\lesssim |t|^{\frac 34}, \label{eq:unif:n}\\
\|u_n(t)\|_{H^1} & \lesssim |t|^{-1}.\label{eq:unif:2}\\
\|xu_n(t)\|_{L^2} & \lesssim 1.\label{eq:unif:3}
\end{align} 
\end{proposition}
\begin{remark}
We observe by change of variable that
\[
\lim_{t\uparrow 0} \biggl\{S(t)-\frac{e^{\frac it}}{|t|^{\frac12}} Q\left(\frac x t\right)\biggr\}=
0\quad \mbox{in $L^2(\R)$},
\]
where $S$ is defined in \eqref{def:S}, 
which justifies the equality of the limits in \eqref{eq:thm1}.
\end{remark}
\begin{proof}
We use Proposition~\ref{pr:boot}, passing from the variable $(s,y)$ to the variable $(t,x)$, and 
using \eqref{eq:BS}. Let $t_0=t_0(n)$ be defined by $T_n-t_0 = j_1(s_0)$. By \eqref{eq:t}, we have
the estimate $|t_0|\gtrsim s_0^{-1}$, and so there exists $t_1$ independent of $n$ such that $t_0(n)\leq t_1<0$.
We now consider any $t\in [t_1,T_n]$.
For the sake of readability, we do not mention the $n$ dependence in the following formulas.
Let
\begin{align*}
w(t,x) &= e^{i\gamma} \lambda^{-\frac 12} W(s,y),\\
z(t,x) &= e^{i\gamma} \lambda^{-\frac 12} Z(s,y).
\end{align*}
We recall from \eqref{def:uveta} and $V=W+Z$ that
\[
u= w+ z +  e^{i\gamma} \eta. 
\]
First, from \eqref{eq:BS3}, we know that
\[
\|\eta(t)\|_{L^2}\lesssim  |t|^2,
\quad \|\partial_x \eta(t)\|_{L^2} \lesssim  |t|,\quad
\|x \eta(t)\|_{L^2}\lesssim   |t|^2.
\]
Second, we estimate $z$. From \eqref{festZ5}, recall that
\[
|Z| \lesssim  s^{-\frac 32} \mathbf{1}_{J(s)} + s^{-\frac{9}2} (|y|^3 + 1) \mathbf{1}_{I(s)}
\lesssim  s^{-\frac 32} \mathbf{1}_{I(s)}
\]
and
\[
 |\partial _y Z| 
\lesssim (1 + |y|)^{-1} \left[  s^{-\frac 32} \mathbf{1}_{J(s)} + s^{-\frac{9}2} (|y|^3 + 1) \mathbf{1}_{I(s)}\right] 
\lesssim  s^{-\frac 52} \mathbf{1}_{I(s)}.
\]
Thus,
\[
\|Z(s)\|_{L^2}\lesssim  s^{-1},\quad
\|\partial_y Z(s)\|_{L^2} \lesssim  s^{-2},\quad
\|yZ(s)\|_{L^2}\lesssim 1.
\]
By change of variable and \eqref{eq:t}, we obtain
\[
\|z(t)\|_{H^1} \lesssim  |t|,\quad
\|x z(t)\|_{L^2} \lesssim |t|.
\]
Last, we  decompose $w$ as
\[
w=q+r,\quad q(t,x)=e^{i\gamma} \lambda^{-\frac 12}M_{-b} Q_a (y),\quad r(t,x)=e^{i\gamma} \lambda^{-\frac12} \theta(s,y) (A+iB)(y).
\]

By the change of variable $z=\frac xt$, we have
\[
\biggl\|q(t)-\frac{e^{\frac it}}{|t|^{\frac12}} Q\left(\frac x t\right)  \biggr\|_{L^2} 
= \biggl\|e^{i\tilde \gamma} \tilde\lambda^{-\frac 12} e^{-ib\tilde \lambda^{-2} \frac{z^2}{4}}
Q_a\left( \frac {z}{\tilde \lambda}\right)- Q (z) \biggr\|_{L^2} 
\]
where
\[
\tilde\gamma = \gamma-\frac 1t,\quad
\tilde \lambda = \frac\lambda{|t|}.
\]
By \eqref{eq:BS} and \eqref{eq:t}, we have 
\[
|\tilde\gamma |+ |\tilde \lambda -1 | + |b| + |a|\lesssim    |t|^{\frac 34}.
\]
Writing
\begin{align*}
\biggl\|e^{i\tilde \gamma} \tilde\lambda^{-\frac 12} e^{-ib\tilde \lambda^{-2} \frac{z^2}{4}}
Q_a\left( \frac {z}{\tilde \lambda}\right)- Q (z) \biggr\|_{L^2} 
&\lesssim \biggl\|e^{i\tilde \gamma} \tilde\lambda^{-\frac 12} e^{-ib\tilde \lambda^{-2} \frac{z^2}{4}}
Q_a\left( \frac {z}{\tilde \lambda}\right)- Q \left( \frac {z}{\tilde \lambda}\right) \biggr\|_{L^2} \\
&\quad +\biggl\| Q\left( \frac {z}{\tilde \lambda}\right)- Q (z) \biggr\|_{L^2} ,
\end{align*}
it follows from usual computations, using the decay properties of $Q$ for 
the second term on the right-hand side, that
\[
\biggl\|q(t)-\frac{e^{\frac it}}{|t|^{\frac12}} Q\left(\frac x t\right)  \biggr\|_{L^2} \lesssim |t|^{\frac 34}.
\]

Let
\begin{align*}
A_*(s,y) & =\lambda^{\frac 32}|y| \cos \gamma + \kappa \lambda^{\frac 52}|y|^2 \sin \gamma, \\
B_*(s,y) & = - \lambda^{\frac 32}|y| \sin \gamma + \kappa \lambda^{\frac 52}|y|^2 \cos \gamma,
\end{align*}
so that
\[
r_*(t,x)=e^{i\gamma} \lambda^{-\frac12} \theta(s,y) (A_*+iB_*)(y).
\]
Now, note that by change of variable $x=\lambda y$,
\[
\|r-r_*\|_{L^2} =
\|\theta ( (A-A_*) + i (B-B_*) ) \|_{L^2}.
\]
 By the asymptotics of the functions $(\varphi_1,\psi_1)$ and $(\varphi_2,\psi_2)$ in Lemmas \ref{lem1bis} and \ref{lem2}, we have
\begin{align*}
\|\theta ( (A-A_*) + i (B-B_*) ) \|_{L^2}
&\leq 
\lambda^{\frac 32} \|\theta (\varphi_1-|y|)\|_{L^2}
+\lambda^{\frac 32}\|\theta (\psi_1+|y|)\|_{L^2}\\
&\quad +\lambda^{\frac 52}\|\theta (\varphi_2-\kappa |y|^2)\|_{L^2}
+\lambda^{\frac 52}\|\theta (\psi_2+\kappa|y|^2)\|_{L^2}\\
&\lesssim \lambda^\frac32 \|\theta\|_{L^2} \lesssim \delta^{\frac 12} |t|.
\end{align*}
Thus
\begin{equation*}
\|r(t)-r_*\|_{L^2} \lesssim \delta^{\frac 12}|t|.
\end{equation*}
This estimate completes the proof of \eqref{eq:unif:n}.

Finally, by \eqref{festW0}, for $s$ large,
\[
\|W(s)\|_{H^1} \lesssim 1,\quad \|y W(s)\|_{L^2} \lesssim \delta^\frac52 s,
\]
and thus, by change of variable,
\[
\|w(t)\|_{H^1}\lesssim |t|^{-1},\quad
\|xw(t)\|_{L^2}\lesssim \delta^\frac52.
\]
which completes the proof of \eqref{eq:unif:2}--\eqref{eq:unif:3}.
\end{proof}

\subsection{Compactness argument}
Finally, passing to the limit as $n\to +\infty$ in the sequence of solutions construted in Proposition \ref{pr:unif}, we construct a solution $u$ of \eqref{nls}
satisfying the conclusions of Theorem \ref{theorem1}.

By the bound \eqref{eq:unif:2}--\eqref{eq:unif:3}, there exists a subsequence $u_{n_k}(t_1)$ of $u_n(t_1)$
and $u_0\in \Sigma$ such that
\begin{align*}
&u_{n_k}(t_1) \rightharpoonup u_0 \quad \mbox{weakly in $H^1$ as $k\to +\infty$,}\\
&u_{n_k} (t_1) \to u_0 \quad \mbox{strongly in $L^2$ as $k\to +\infty$.}
\end{align*}
Let $u$ be the $H^1$ solution of \eqref{nls} such that $u(t_1)=u_0$.
By the local wellposedness of the Cauchy problem in $L^2$ for \eqref{nls} (see e.g. \cite[Theorem 4.7.1]{CLN10}), the solution $u(t)$ exists on $[t_1,0)$ and for any $t\in [t_1,0)$,
\[
u_{n_k} (t) \to u(t) \quad \mbox{strongly in $L^2$ as $k\to +\infty$}.
\]
Passing to the limit in \eqref{eq:unif:n}, we obtain on $[t_1,0)$,
\begin{align*} 
\biggl\|u(t)-\frac{e^{\frac it}}{|t|^{\frac12}} Q\left(\frac x t\right) - r_*\biggr\|_{L^2} &\lesssim |t|^\frac34,\\
\|u(t)\|_{H^1} & \lesssim |t|^{-1}.
\end{align*}
It follows from the first estimate above that $\{u(t); t\in [t_1,0]\}$ is not compact in $L^2$.
Therefore, the maximal time of existence of the $L^2$ solution $u(t)$ is $0$. By persistence of regularity (see e.g. \cite[Theorem 4.7.1]{CLN10}), $u(t)$ blows up in $H^1$ at time $t=0$.
Next, by the Gagliardo-Nirenberg inequality $\|g\|_{L^6}^3 \lesssim \|\partial_x g\|_{L^2}
\|g\|_{L^2}^2$,  we
obtain
\[
\biggl\|u(t)-\frac{e^{\frac it}}{|t|^{\frac12}} Q\left(\frac x t\right) - r_*\biggr\|_{L^6}
\lesssim |t|^{\frac 16}.
\]
By $E(Q)=0$, we have for any $t\neq 0$,
\[
\biggl\|\frac{e^{\frac it}}{|t|^{\frac12}} Q\left(\frac x t\right)\biggr\|_{L^6}^6
= \frac {\|Q'\|_{L^2}^2} {t^2} \]
and thus
\[
\lim_{t\uparrow 0} t^2\|u(t)\|_{L^6}^6
= \|Q'\|_{L^2}^2.
\]
Therefore, by conservation of the energy $E(u(t))$, we obtain \eqref{eq:thm0}.
This completes the proof of Theorem \ref{theorem1}. 

\appendix
\section{}\label{appendix}

In this appendix, we prove Lemma~\ref{lem1bis}.
First, we prove the following ODE result.

\begin{lemma} \label{agmon2} 
Let $m >0$. Let $(P_j)_{ 1\le j\le 4 } \subset \mathcal \mathcal C( \R, \R)$ and suppose there exists $\delta >0$ such that
\begin{equation} \label{ffagmon1} 
 | P_j (x)| \le C e^{- (2m+ \delta ) x}, \quad x\ge 0
\end{equation} 
 If $u, v\in \mathcal C^2 (\R, \R)$ satisfy 
\begin{equation} \label{fagmon2} 
\begin{cases} 
-u '' + m^2 u+ P_1 u + P_2 v =0 \\
- v'' + P_3 u + P_4 v =0
\end{cases} 
\end{equation} 
then there exist constants $c,d_0, d_1 \in \R$ such that 
\begin{enumerate}[{\rm (i)}]
\item for $x\geq 0$, $|u(x) - ce^{mx }| \lesssim e^{-mx}$,
\item for $x\geq 0$, $| v(x) - d_0 - d_1 x | \lesssim e^{-mx }$.
\end{enumerate}
\end{lemma} 

\begin{proof} 
Setting $u'= mw$ and $v'= mz$, we have
\begin{equation*} 
\frac {d} {dx}(u^2 + w^2+ v^2+ z^2) = 4m uw + 2m vz + 2\frac {P_1} {m} uw + 2\frac {P_2} {m} vw + 2\frac {P_3} {m} uz + 2 \frac {P_4} {m} vz .
\end{equation*} 
It follows that for every $\nu >0$, 
\begin{equation*} 
\frac {d} {dx}(u^2 + w^2+ v^2+ z^2) \le 2 (m+ \nu) (u^2 + w^2+ v^2+ z^2)
\end{equation*} 
for all sufficiently large $x$; and so 
\begin{equation*} 
u^2 + w^2+ v^2+ z^2 \lesssim e^{2(m+\nu) x},\quad x\ge 0.
\end{equation*} 
It follows that 
\begin{equation} \label{fagmon2:0} 
 |P_j| ( |u| + |v|)\lesssim e^{ - (m+ \delta -\nu) x}, \quad x\ge 0,
\end{equation} 
for $ j=1,2,3, 4$.
Next, from the variation of the parameter formula, we see that there exist constants $u_1$ and $u_2$ such that
\begin{equation} \label{fagmon2:2} 
u(x)= u_1 e^{mx }+ u_2 e^{-m x} + \frac {1} {2m} \int_0^x ( e^{m (x-s) } - e^{ - m (x-s) } ) (P_1 u + P_2 v )(s) \, ds. 
\end{equation} 
Applying~\eqref{fagmon2:0} with $\nu < \delta $ and setting 
\begin{equation} \label{fagmon2:3} 
c = u_1 + \frac {1} {2m} \int_0^\infty e^{-m s} (P_1 u + P_2 v )(s) \, ds ,
\end{equation} 
we deduce from~\eqref{fagmon2:2} that
\begin{equation} \label{fagmon2:4} 
\begin{aligned} 
u(x)- c e^{mx} & = u_2 e^{-mx} - \frac {1} {2m} \int_0^x e^{-m(x-s) } (P_1 u + P_2 v )(s) \,ds \\ &\quad - \frac {1} {2m} \int_x^\infty e^{ m(x-s) } (P_1 u + P_2 v )(s) \,ds .
\end{aligned} 
\end{equation} 
It follows easily from~\eqref{fagmon2:0} (with $\nu < \delta $) and~\eqref{fagmon2:4} that
\begin{equation} \label{fagmon2:5} 
 |u(x)- c e^{mx}| \lesssim e^{-m x}, \quad x\ge 0 .
\end{equation} 
This proves the first part of the statement.

Set now 
\begin{equation*} 
d_1 = v'(0) + \int_0^\infty (P_3 u + P_4 v) (\tau )\, d\tau ,
\end{equation*} 
which is well defined by~\eqref{fagmon2:0} (with $\nu < \delta $). It follows from the second equation in~\eqref{fagmon2} that 
\begin{equation*} 
v'(s)- d_1= - \int_s^\infty (P_3 u + P_4 v) (\tau )\, d\tau .
\end{equation*} 
Using again~\eqref{fagmon2:0}, we see that
\begin{equation*} 
 \Bigl| \int_s^\infty (P_3 u + P_4 v) (\tau )\, d\tau \Bigr| \lesssim e^{-ms},
\end{equation*} 
and we define
\begin{equation*} 
d_0= v(0) - \int_0^\infty \int_s^\infty (P_3 u + P_4 v) (\tau )\, d\tau ds.
\end{equation*} 
It follows that 
\begin{equation*} 
v(x)- d_0 -d_1 x = \int_x^\infty \int_s^\infty (P_3 u + P_4 v) (\tau )\, d\tau ds ,
\end{equation*} 
so that $ | v(x)- d_0 -d_1 x | \lesssim e^{-mx}$. This completes the proof.
\end{proof} 

We turn to the proof of Lemma~\ref{lem1bis}.
Set
\begin{equation*} 
 \widetilde{\varphi} = \varphi+\psi ,\quad \widetilde{\psi} = \varphi-\psi ,
 \quad \widetilde{g} = g+h ,\quad \widetilde{h} = g-h .
\end{equation*} 
We see by the expressions of $L_+$ and $L_-$ and direct computations that
the system~\eqref{ellipbis} is equivalent to
\begin{equation} 
\begin{cases} \label{ellip2bis} 
- \widetilde{\varphi} '' + 2 \widetilde{\varphi} - 3Q^4 \widetilde{\varphi} - 2 Q^4 \widetilde{\psi} = \widetilde{g} \\
- \widetilde{\psi} '' - 3Q^4 \widetilde{\psi} - 2 Q^4 \widetilde{\varphi} = \widetilde{h} . 
\end{cases} 
\end{equation} 
Recall that we consider only even solutions.

\begin{lemma} \label{lem1}
The following properties hold.

\begin{enumerate}[{\rm (i)}] 

\item \label{lem1:1} 
There exists a unique pair of even functions $( \widetilde{\varphi }_1, \widetilde{\psi }_1 ) \in \mathcal Y \times \mathcal Z_1$ 
such that
\begin{equation} \label{flem1:1} 
\begin{cases} 
- \widetilde{\varphi }_1 '' + 2 \widetilde{\varphi }_1 - 3 Q^4 \widetilde{\varphi }_1 - 2 Q^4 \widetilde{\psi }_1 = 0 \\
- \widetilde{\psi }_1 '' - 3 Q^4 \widetilde{\psi }_1 - 2 Q^4 \widetilde{\varphi }_1 = 0
\end{cases} 
\end{equation} 
and satisfying
\begin{equation}\label{flem1:1bis}
\widetilde \psi_1(x) = \mu_1(x) + c_1 + \widetilde w_1(x)
\end{equation}
where $c_1\in \R$ and $\widetilde w_1\in \mathcal Y$.
\item \label{lem1:2} 
Given $( \widetilde{g}, \widetilde{h} )\in \mathcal Y \times \mathcal Y$, there exists a unique even solution $( \widetilde{\varphi} , \widetilde{\psi} ) \in \mathcal Y \times  \mathcal Z_0$ of~\eqref{ellip2bis}. 
\end{enumerate} 
\end{lemma} 

\begin{proof} 
Proof of part~\eqref{lem1:1}. \quad 
Define the solutions $(f_1, g_1)$ and $(f_2, g_2)$ of system~\eqref{flem1:1} corresponding to the initial data
$f_1 (0)= 1, f'_1(0)= 0$, $g_1(0)=0$, $g_1'(0)=0$ and $f_2 (0)= 0$, $f'_2(0)= 0$, $g_2(0)=1$, $g_2'(0)=0$.
From (i) of Lemma~\ref{agmon2} and standard ODE arguments for the regularity and the decay of the derivatives, 
it follows that there exist $a_1, a_2 \in \R$ and $w_1, w_2\in \mathcal Y$ such that 
\[
f_1(x) = a_1 \mu_0(x) e^{|x| \sqrt 2} + w_1(x) \quad \mbox{and}\quad  f_2(x) = a_2 \mu_0(x) e^{|x| \sqrt 2 } + w_2(x).
\]
Therefore, there exists a linear combination of the pairs $(f_1, g_1)$ and $(f_2, g_2)$, denoted by $(f,g)$, 
such that $f \in \mathcal Y$. Since the solutions $(f_1, g_1)$ and $(f_2, g_2)$ are independent, the pair $(f,g)$ is non zero and satisfies the system~\eqref{flem1:1}.
By (i) of Lemma~\ref{agmon2},  
there exist $b_0, b_1\in \R$ and $z\in \mathcal Y$ such that 
\[ g(x) = b_1 \mu_1(x) +b_0 +  z(x).\]
Recall from~\cite[Proposition~2.1.4]{Perelman} that
there exist no nonzero bounded solution of~\eqref{flem1:1}, which means exactly that $b_1\neq 0$.
Thus, $ ( \widetilde{\varphi }_1, \widetilde{\psi }_1 )= \frac {1} {b_1} (f,g) $
satisfies \eqref{flem1:1} and \eqref{flem1:1bis}.

The uniqueness part of (i) also follows from the result of non existence of a non zero bounded solution
of \eqref{flem1:1} in~\cite[Proposition~2.1.4]{Perelman}.

Proof of part~\eqref{lem1:2}. \quad 
Let $(\widetilde g,\widetilde h)\in \mathcal Y \times \mathcal Y$ be a pair of even functions.
We first construct a solution of~\eqref{ellip2bis} in $ \mathcal Y \times \mathcal Z_1$,
using a reformulation of the problem.

Indeed, integrating twice the second equation in~\eqref{ellip2bis}, and  using that $\widetilde\psi$ is even, we obtain
\begin{equation*} 
 \widetilde{\psi} + 3 \int_0^x \int_0^y Q^4 \widetilde{\psi} = - \int_0^x \int_0^y (2Q^4 \widetilde{\varphi} + \widetilde{h} ) + C 
\end{equation*} 
for some constant $C$.
Now, we will solve the system
\begin{equation} \label{flem1:1b} 
\begin{cases} 
- \widetilde{\varphi} '' + 2 \widetilde{\varphi} - 3Q^4 \widetilde{\varphi} - 2 Q^4 \widetilde{\psi} = \widetilde{g} \\
\displaystyle \widetilde{\psi} + 3 \int_0^x \int_0^y Q^4 \widetilde{\psi} = - \int_0^x \int_0^y (2Q^4 \widetilde{\varphi} + \widetilde{h} ) ,
\end{cases} 
\end{equation} 
where for simplicity
the constant $C$ was chosen equal to $0$.
We rewrite the second equation in~\eqref{flem1:1b} in the equivalent form
\begin{equation}\label{flem1:2} 
Q^2 \widetilde{\psi} - T ( Q^2 \widetilde{\psi} ) = - Q^2 \int_0^x \int_0 ^y (2 Q^4 \widetilde{\varphi} + \widetilde{h} ) ,
\end{equation} 
where we have defined, for any $u\in L^2(\R)$
\begin{equation*} 
T u = - 3 Q^2 \int_0^x \int_0^y Q^2 u.
\end{equation*} 
It is clear that $T\in \mathcal L (L^2 (\R) )$ and that $T $ is a compact operator $L^2 (\R ) \to L^2 (\R ) $. 
We also observe that if the function $u$ is even, then $Tu$ is also even.
Moreover, the adjoint of $T$ is given by (use Fubini's theorem)
\begin{equation} \label{flem1:5b} 
T^\star u= - 3 Q^2 \int_x^\infty \int_y^\infty Q^2 u.
\end{equation} 
We claim that $I- T^\star$ has a trivial kernel. Indeed, suppose $u\in L^2 (\R ) $ and $T^\star u= u$.
By the definition of $T_*$, given $R\ge 0$ and $x>R$, we have by the Cauchy-Schwarz inequality and \eqref{def:Q}
\begin{equation*} 
\begin{aligned} 
 |T^\star u(x)| &\le 3Q^2 \int_x^\infty \| u\|_{ L^2(\{ z > y\}) } \Bigl( \int_y ^\infty Q^4\Bigr)^{\frac {1} {2}} 
\\ & \lesssim e^{-2x} \| u\|_{ L^2(\{ z>R\}) } \int_x^\infty e^{-2y} dy \lesssim e^{-4 x} \| u\|_{ L^2(\{z>R\}) } . 
\end{aligned} 
\end{equation*} 
 By the above estimate and $T^\star u= u$, we deduce that $u(x) =0$ for $x$ large. The same vanishing property holds for $v= Q^{-2} u$, which satisfies the ODE $v'' = - 3Q^4 v$. By uniqueness, $v\equiv 0$ and so $u\equiv 0$ on $\R$. 
By Fredholm's alternative, $I-T$ is invertible on $L^2 (\R ) $. 
Therefore, we may write equation~\eqref{flem1:2} in the form
\begin{equation} \label{flem1:3} 
 \widetilde{\psi} = - Q^{-2} (I - T)^{-1} \Bigl( Q^2 \int_0^x \int_0 ^y (2 Q^4 \widetilde{\varphi} + \widetilde{h} ) \Bigr) .
\end{equation} 
We let 
\begin{equation*} 
v= - \widetilde{\varphi}'' + 2 \widetilde{\varphi} ,
\end{equation*} 
i.e.
\begin{equation} \label{flem1:3c} 
 \widetilde{\varphi} = (-\partial_{ xx } +2 )^{-1} v,
\end{equation} 
where the inverse is in the sense of $L^2 (\R) $. 
The first equation in~\eqref{flem1:1b} becomes
\begin{equation*} 
v - 3 Q^4 (- \partial_{ xx }+ 2 )^{-1} v = \widetilde{g} + 2 Q^4 \widetilde{\psi}. 
\end{equation*} 
Using~\eqref{flem1:3}, we rewrite this last equation in the form
\begin{equation*} 
\begin{aligned} 
v - 3 Q^4 (- \partial_{ xx }+ 2 )^{-1} v = & \widetilde{g} - 2 Q^2 (I-T)^{-1} \Bigl( Q^2 \int_0^x \int_0 ^y \widetilde{h} \Bigr) 
\\ & + \frac {4} {3} Q^2 (I-T)^{-1} T ( Q^2 (- \partial_{ xx }+ 2 )^{-1} v) .
\end{aligned} 
\end{equation*} 
Set
\begin{equation*} 
\begin{aligned} 
T_1 v & = 3Q^4 (-\partial_{ xx }+ 2)^{-1} v + \frac {4} {3} Q^2 (I-T)^{-1} T ( Q^2 (- \partial_{ xx }+ 2 )^{-1} v ) \\
& = \frac {5} {3} Q^4 (-\partial_{ xx }+ 2)^{-1} v + \frac {4} {3} Q^2 (I-T)^{-1} ( Q^2 (- \partial_{ xx }+ 2 )^{-1} v ) \\
& = \frac {1} {3} Q^2 [ 5 + 4 (I-T)^{-1} ] ( Q^2 (- \partial_{ xx }+ 2 )^{-1} v ).
\end{aligned} 
\end{equation*} 
The above equation is equivalent to
\begin{equation} \label{flem1:3b} 
(I - T_1 ) v= \widetilde{g} - 2 Q^2 (I-T)^{-1} \Bigl( Q^2 \int_0^x \int_0 ^y \widetilde{h} \Bigr) .
\end{equation} 
On $L^2 (\R)$, the operator $v\mapsto Q^2 (- \partial_{ xx }+ 2 )^{-1} v$ is compact  and the operator $v \mapsto Q^2 [5 + 4 (I-T)^{-1} ]v$ is continuous. Thus, $T_1$ is a compact operator. 
Moreover,
\begin{equation*} 
T_1^\star u= \frac {1} {3} (- \partial_{ xx }+ 2 )^{-1} [ Q^2 (5 + 4 (I- T^\star)^{-1} ) Q^2 u]. 
\end{equation*} 
We claim that $I- T_1^\star$ has a trivial kernel. Indeed, suppose $u\in L^2 (\R ) $ and $T_1^\star u= u$. 
In particular, $u\in H^2 (\R ) $ and 
\begin{equation*} 
- u'' + 2u = \frac {1} {3} [ Q^2 (5 + 4 (I- T^\star)^{-1} ) Q^2 u] .
\end{equation*} 
Setting 
\begin{equation} \label{flem1:4} 
u_1 = (I- T^\star)^{-1} Q^2 u,
\end{equation} 
we obtain
\begin{equation} \label{flem1:5} 
- u'' + 2u = \frac {4} {3} Q^2 u_1 + \frac {5} {3} Q^4 u .
\end{equation} 
We rewrite equation~\eqref{flem1:4} in the form
\begin{equation*} 
- \frac {T^\star u_1} {Q^2} = u - \frac {u_1} {Q^2}.
\end{equation*} 
Applying~\eqref{flem1:5b} and differentiating twice, we obtain 
\begin{equation*} 
3 Q^2 u_1= \Bigl( u - \frac {u_1} {Q^2} \Bigr) '' .
\end{equation*} 
Last, we set
\begin{equation} \label{flem1:6} 
w= - \frac {2} {3} \Bigl( u- \frac {u_1} {Q^2} \Bigr),
\end{equation} 
so that 
\begin{equation} \label{flem1:7} 
w'' = -2 Q^2 u_1 . 
\end{equation} 
Moreover, \eqref{flem1:6} implies
\begin{equation} \label{flem1:8} 
Q^2 u_1= Q^4 u + \frac {3} {2} Q^4 w. 
\end{equation} 
Substituting~\eqref{flem1:8} in~\eqref{flem1:5} and~\eqref{flem1:7}, we obtain the system
\begin{equation*} 
\begin{cases} 
- u'' + 2 u - 3 Q^4 u - 2 Q^4 w= 0, \\
- w'' - 2 Q^4 u - 3 Q^4 w=0. 
\end{cases} 
\end{equation*} 
This system does not have any non trivial $L^2$ solution by~\cite[Corollary~2.1.3]{Perelman}, so we conclude that $u \equiv 0$. This proves that $I-T_1^\star$ has a trivial kernel as claimed.
By Fredholm's alternative, $I - T_1$ is invertible and 
we define the even function
\begin{equation} \label{flem1:9} 
 v= (I - T_1 )^{-1} \Bigl[ \widetilde{g} - 2 Q^2 (I-T)^{-1} \Bigl( Q^2 \int_0^x \int_0 ^y \widetilde{h} \Bigr) \Bigr],
\end{equation} 
so that equation~\eqref{flem1:3b} is satisfied. 
Let $ \widetilde{\varphi} $ be given by~\eqref{flem1:3c} and let $ \widetilde{\psi} $ be given by~\eqref{flem1:3}. 
Observe that $ \widetilde{\varphi} \in H^2 (\R) $ and $Q^2 \widetilde{\psi} \in L^2 (\R) $. 
Moreover, the pair $ ( \widetilde{\varphi}, \widetilde{\psi} )$ satisfies~\eqref{flem1:1b}. 
By the second equation in~\eqref{flem1:1b}, we see that $ \widetilde{\psi} \in H^2_{\mathrm{loc}} (\R )$ and $ \widetilde{\psi} ''\in L^2 (\R) $. From the above, it follows that $ ( \widetilde{\varphi}, \widetilde{\psi} )$ satisfy~\eqref{ellip2bis}.
Integrating twice the second equation in~\eqref{ellip2bis}, we deduce that 
$\widetilde{\psi}(x) = d_1 \mu_1(x)+ d_0 + w(x)$ with $d_1,d_0\in \R$ and $w\in \mathcal Y$, and it follows easily that $ \widetilde{\varphi} \in \mathcal Y$. 
Therefore, the pair $( \widetilde{\varphi}- d_1 \widetilde{\varphi }_1, \widetilde{\psi}- d_1 \widetilde{\psi }_1 )$ is a solution of~\eqref{ellip2bis} in $ \mathcal Y \times \mathcal Z_0$. 
Finally, uniqueness follows again from~\cite[Proposition~2.1.4]{Perelman}.
\end{proof} 

Let $( \widetilde{\varphi }_1, \widetilde{\psi }_1 )$ given by part (i) of Lemma~\ref{lem1}. It follows that 
that the pair
\[(\varphi _1, \psi_1)= ( \widetilde{\varphi }_1+ \widetilde{\psi }_1 , \widetilde{\varphi }_1- \widetilde{\psi }_1)\]
is an even solution of the system
\begin{equation} \label{fedefphi1psi1} 
\begin{cases} 
\psi_1 + L_+ \varphi _1 = 0 \\
\varphi _1 + L_- \psi_1 = 0 
\end{cases} 
\end{equation} 
such that
\begin{equation} \label{bfp1} 
\varphi _1 - \mu_1 - c_1 \in \mathcal Y, \quad \psi_1 + \mu_1 + c_1 \in \mathcal Y,
\end{equation}
which proves part (i) of Lemma~\ref{lem1bis}.

Part (ii) of Lemma~\ref{lem1bis} follows similarly from part (ii) of Lemma~\ref{lem1}.

\section{}\label{appendixB}

In this appendix, we justify the identity~\eqref{eq:W}.
Recall that
\begin{equation*}
W(s,y) = M_{ -b} Q_{a} + \theta (A + i B )
\end{equation*}
and that
\begin{equation*}
m_\gamma = \gamma_s - 1,\quad m_\lambda = \frac{\lambda_s}{\lambda}+b,
\quad m_b = b_s+b^2-a,\quad
m_a = a_s - \Omega.
\end{equation*}
Some useful calculations
\begin{equation} \label{fusef1} 
\begin{cases} 
 \partial _s M_{ -b }  = - i M_{ -b } b_s \frac {y^2} {4} \\
 \partial _y M_{ -b }  = - i M_{ -b } b \frac {y} {2} \\
 \partial ^2 _y M_{ -b }  = M_{ -b } ( -i \frac {b} {2} - b^2 \frac {y^2} {4} ) \\
 \partial _s Q_a = a_s \rho \\
 \partial ^2 _y Q_a = Q_a - a \frac{y^2}{4} Q -Q^5-5aQ^4\rho \\
 \partial _s \theta  = \lambda _s y \Theta ' (\lambda y) = y \frac {\lambda _s} {\lambda } \partial _y \theta 
\end{cases} 
\end{equation} 
We first rewrite $\mathcal E(W) $. We replace $- W - (\gamma_s - 1) W$ by $- \gamma _s W$, so
\begin{equation*} 
 \mathcal E(W) - f( W) = i \partial_s W + \partial_y^2 W 
- i \frac{\lambda_s}{\lambda} \Lambda W- \gamma_s W 
\end{equation*} 
Next, using $W = M_{ -b} Q_{a} + \theta (A + i B )$, we obtain
\begin{equation*} 
\begin{aligned} 
& \mathcal E(W) - f( W) = i \partial_s (M_{ -b} Q_{a} ) + \partial_y^2 (M_{ -b} Q_{a} ) 
- i \frac{\lambda_s}{\lambda} \Lambda (M_{ -b} Q_{a} )- \gamma_s M_{ -b} Q_{a} \\ 
& + i \partial_s (\theta (A + i B )) + \partial_y^2 (\theta (A + i B )) 
- i \frac{\lambda_s}{\lambda} \Lambda (\theta (A + i B ))- \gamma_s (\theta (A + i B )) 
\end{aligned} 
\end{equation*} 
We next use $\partial _s (uv) = u \partial _s v + \partial _s u v$, $\partial _{ yy } (uv) = u \partial _{ yy }v + 2 \partial _y u \partial _y v + \partial _{ yy } u v $ and $\Lambda (uv) = u \Lambda v + y v \partial _y u$, and we obtain
\begin{equation*} 
\begin{aligned} 
  \mathcal E(W) - f( W) &= i ( \partial_s M_{ -b} ) Q_{a} + i M_{ -b} \partial _s Q_a \\
& \quad+ \partial _{ yy } M_{ -b} Q_{a} + 2 \partial _y M _{ -b } \partial _y Q_a + M _{ -b } \partial _{ yy } Q_a \\
& \quad- i \frac{\lambda_s}{\lambda} M_{ -b} \Lambda Q_{a} - i \frac{\lambda_s}{\lambda} y \partial _y M_{ -b} Q_{a} - \gamma_s M_{ -b} Q_{a} \\ 
& \quad+ i \partial_s \theta (A + i B ) + i \theta (\partial _s A + i \partial _s B) \\
&\quad+ \theta _{ yy }(A + i B ) + 2 \theta _y (A_y + i B_y ) + \theta (A _{ yy } + i B _{ yy }) \\
&\quad - i \frac{\lambda_s}{\lambda} \theta ( \Lambda A + i \Lambda B ) - i \frac{\lambda_s}{\lambda} y \theta _y (A + i B ) - \gamma_s \theta (A + i B ) 
\end{aligned} 
\end{equation*} 
Using now the calculations~\eqref{fusef1}: 
\begin{equation*} 
\begin{aligned} 
 \mathcal E(W) - f( W) & = M_{ -b } b_s \frac {y^2} {4} Q_{a} + i M_{ -b} a_s \rho \\
&\quad + M_{ -b } \Bigl( -i \frac {b} {2} - b^2 \frac {y^2} {4} \Bigr) Q_{a} -2 i M_{ -b } b \frac {y} {2} \partial _y Q_a \\
&\quad + M _{ -b } (Q_a - a \frac{y^2}{4} Q -Q^5-5aQ^4\rho ) \\
& \quad- i \frac{\lambda_s}{\lambda} M_{ -b} \Lambda Q_{a} - \frac{\lambda_s}{\lambda} M_{ -b } b \frac {y^2} {2} Q_{a} - \gamma_s M_{ -b} Q_{a} \\ 
&\quad + i \frac {\lambda _s} {\lambda } y \theta _y (A + i B ) + i \theta (\partial _s A + i \partial _s B) \\
&\quad+ \theta _{ yy }(A + i B ) + 2 \theta _y (A_y + i B_y ) + \theta (A _{ yy } + i B _{ yy }) \\
&\quad - i \frac{\lambda_s}{\lambda} \theta ( \Lambda A + i \Lambda B )) - i \frac{\lambda_s}{\lambda} y \theta _y (A + i B )) - \gamma_s (\theta (A + i B )) 
\end{aligned} 
\end{equation*} 
We cancel the terms $ i \frac {\lambda _s} {\lambda } y \theta _y (A + i B ) - i \frac {\lambda _s} {\lambda } y \theta _y (A + i B )$ and factorize $M _{ -b }$:
\begin{equation*} 
\begin{aligned} 
& \mathcal E(W) - f( W) = M_{ -b } \Bigl( b_s \frac {y^2} {4} Q_{a} + i a_s \rho \\
& -i \frac {b} {2} Q_a - b^2 \frac {y^2} {4} Q_{a} - i b y \partial _y Q_a + Q_a - a \frac{y^2}{4} Q -Q^5-5aQ^4\rho \\
& - i \frac{\lambda_s}{\lambda} \Lambda Q_{a} - \frac{\lambda_s}{\lambda} b \frac {y^2} {2} Q_{a} - \gamma_s Q_{a} \Bigr) \\ 
& + i \theta (\partial _s A + i \partial _s B) + \theta _{ yy }(A + i B ) + 2 \theta _y (A_y + i B_y ) + \theta (A _{ yy } + i B _{ yy }) \\
& - i \frac{\lambda_s}{\lambda} \theta ( \Lambda A + i \Lambda B ) - \gamma_s (\theta (A + i B ) 
\end{aligned} 
\end{equation*} 
We use $ Q_a - \gamma_s Q_{a} = - m_\gamma Q_a$; $-i \frac {b} {2} Q_a - i b y \partial _y Q_a = -i b \Lambda Q_a$; $b + \frac {\lambda _s} {\lambda }= m_\lambda $ and $- a \frac{y^2}{4} Q= - a \frac{y^2}{4} Q_a + a^2 \frac{y^2}{4} \rho $: 
\begin{equation*} 
\begin{aligned} 
& \mathcal E(W) - f( W) = M_{ -b } \Bigl( b_s \frac {y^2} {4} Q_{a} + i a_s \rho \\
& -i m_\lambda \Lambda Q_a - b^2 \frac {y^2} {4} Q_{a} - a \frac{y^2}{4} Q_a -Q^5-5aQ^4\rho + a^2 \frac{y^2}{4} \rho - \frac{\lambda_s}{\lambda} b \frac {y^2} {2} Q_{a} - m_\gamma Q_{a} \Bigr) \\ 
& + i \theta (\partial _s A + i \partial _s B) + \theta _{ yy }(A + i B ) + 2 \theta _y (A_y + i B_y ) + \theta (A _{ yy } + i B _{ yy }) \\
& - i \frac{\lambda_s}{\lambda} \theta ( \Lambda A + i \Lambda B ) - \gamma_s (\theta (A + i B ) 
\end{aligned} 
\end{equation*} 
We reorder the terms in $y^2 Q_a$ and use $b_s - b^2 -a - 2b \frac {\lambda _s} {\lambda } = m_b - 2b m_\lambda $: 
\begin{equation*} 
\begin{aligned} 
& \mathcal E(W) = M_{ -b } \Bigl( (m_b - 2b m_\lambda ) \frac {y^2} {4} Q_{a} + i a_s \rho \\
& -i m_\lambda \Lambda Q_a -Q^5-5aQ^4\rho + a^2 \frac{y^2}{4} \rho - m_\gamma Q_{a} \Bigr) \\ 
& + i \theta (\partial _s A + i \partial _s B) + \theta _{ yy }(A + i B ) + 2 \theta _y (A_y + i B_y ) + \theta (A _{ yy } + i B _{ yy }) \\
& - i \frac{\lambda_s}{\lambda} \theta ( \Lambda A + i \Lambda B ) - \gamma_s (\theta (A + i B ) + f( W) 
\end{aligned} 
\end{equation*} 
Using now $a_s= m_a + \Omega $ and reordering 
\begin{equation*} 
\begin{aligned} 
& \mathcal E(W) = M_{ -b } \Bigl( - m_\gamma Q_{a} + (m_b - 2b m_\lambda ) \frac {y^2} {4} Q_{a} + i m_a \rho -i m_\lambda \Lambda Q_a \\
& -Q^5-5aQ^4\rho + a^2 \frac{y^2}{4} \rho + i \Omega \rho \Bigr) \\ 
& + i \theta (\partial _s A + i \partial _s B) + \theta _{ yy }(A + i B ) + 2 \theta _y (A_y + i B_y ) + \theta (A _{ yy } + i B _{ yy }) \\
& - i \frac{\lambda_s}{\lambda} \theta ( \Lambda A + i \Lambda B ) - \gamma_s (\theta (A + i B ) + f( W) 
\end{aligned} 
\end{equation*} 
Using $M _{ -b } ( - Q^5 -5aQ^4\rho + a^2 \frac{y^2}{4} \rho) = \mathcal R_4 - f(M _{ -b } Q_a)$, we find:
\begin{equation*} 
\begin{aligned} 
& \mathcal E(W) = M_{ -b } \Bigl( - m_\gamma Q_{a} + (m_b - 2b m_\lambda ) \frac {y^2} {4} Q_{a} + i m_a \rho -i m_\lambda \Lambda Q_a \Bigr) \\
& + i M_{ -b } \Omega \rho + \mathcal R_4 + f( W) - f(M _{ -b } Q_a) \\ 
& + i \theta (\partial _s A + i \partial _s B) + \theta _{ yy }(A + i B ) + 2 \theta _y (A_y + i B_y ) + \theta (A _{ yy } + i B _{ yy }) \\
& - i \frac{\lambda_s}{\lambda} \theta ( \Lambda A + i \Lambda B ) - \gamma_s \theta (A + i B ) 
\end{aligned} 
\end{equation*} 
Taking into account the definitions of $\mathcal S_0$, $\mathcal S_1$ and $\mathcal R_5$, this is equivalent to
\begin{equation} \label{fewv1} 
\begin{aligned} 
& \mathcal E(W) = \mathcal S_0 + \mathcal S_1 + \mathcal R_4 + \mathcal R_5 + i \theta \Omega \rho + f( W) - f(M _{ -b } Q_a) \\ 
& + i \theta (\partial _s A + i \partial _s B) + \theta (A _{ yy } + i B _{ yy }) - i \frac{\lambda_s}{\lambda} \theta ( \Lambda A + i \Lambda B ) - \gamma_s \theta (A + i B ) 
\end{aligned} 
\end{equation} 
Next, using the definitions of $L_+$ and $L_-$, $m_\gamma = \gamma _s -1$ and $m_\lambda -b= \frac {\lambda _s} {\lambda }$, we rewrite 
\begin{align*}
G & = - \partial_s B + A _{ yy } + 5 Q^4 A
-b y^2 Q^4 B- \gamma_s A + \frac {\lambda _s} {\lambda } \Lambda B , \\
H & = \partial_s A + B _{ yy } + Q^4 B -b y^2 Q^4 A- \gamma_s B - \frac {\lambda _s} {\lambda } \Lambda A\\
&\quad +\bigl(\alpha_1 b \lambda^\frac32 \cos \gamma + \alpha_2 b\lambda^\frac52\sin \gamma\bigr)\rho.
\end{align*}
It follows that
\begin{equation*} 
\begin{aligned} 
{\mathcal R}_3& = i \theta ( \partial _s A+ i \partial_s B) + \theta ( A _{ yy } +i B _{ yy }) - i \frac {\lambda _s} {\lambda } \theta (\Lambda A + i \Lambda B) - \gamma _s \theta (A+ i B) \\
& \quad+ \theta ( 5 Q^4 A -b y^2 Q^4 B ) + i\theta ( Q^4 B -b y^2 Q^4 A ) \\
& \quad+ i \theta \bigl(\alpha_1 b \lambda^\frac32 \cos \gamma + \alpha_2 b\lambda^\frac52\sin \gamma\bigr)\rho
\end{aligned} 
\end{equation*} 
Using the definition of ${\mathcal S}_2$ and $\Omega $, we deduce that
\begin{equation} \label{fewv2} 
\begin{aligned} 
  {\mathcal S}_2+ {\mathcal R}_3  
& = i \theta ( \partial _s A+ i \partial_s B) + \theta ( A _{ yy } +i B _{ yy })\\
&\quad  - i \frac {\lambda _s} {\lambda } \theta (\Lambda A + i \Lambda B) - \gamma _s \theta (A+ i B) \\
&\quad+ f(W)-f(M_{-b}Q_a) +i \theta \Omega \rho
\end{aligned} 
\end{equation} 
Identities~\eqref{fewv1} and~\eqref{fewv2} immediately yield~\eqref{eq:W}.


\begin{thebibliography}{99}

\bibitem{BourgainW}{Bourgain J. and Wang W.} {Construction of blowup solutions for the nonlinear Schr\"o\-din\-ger equation with critical nonlinearity. Dedicated to Ennio De Giorgi. Ann. Sc. Norm. Super. Pisa Cl. Sci. (4) {\bf 25} (1997), no. 1-2, 197--215.} 
\MScN{MR1655515} \LINK{http://www.numdam.org/item/?id=ASNSP_1997_4_25_1-2_197_0}

\bibitem{CLN10}{Cazenave T.} {{\it Semilinear Schr\"o\-din\-ger 
equations}, Courant Lecture Notes in Mathematics, {\bf 10}. New York
University, Courant Institute of Mathematical Sciences, New York; American
Mathematical Society, Providence, RI, 2003.}

\bibitem{Jendrej} {Jendrej J.}
{Construction of type II blow-up solutions for the energy-critical wave equation in dimension 5,
J. Funct. Anal. {\bf 272} (2017), no. 3, 866--917.}

\bibitem{Jendrej2016} {Jendrej J.}
{Bounds on the speed of type II blow-up for the energy critical wave equation in the radial case,
Int. Math. Res. Not.  (2016), no. 21, 6656--6688.}

\bibitem{JendrejLR}  {Jendrej J., Lawrie A. and Rodriguez C.}
{Dynamics of bubbling wave maps with prescribed radiation, 
Ann. Sci. Éc. Norm. Supér. \textbf{55} (2022), no. 4, 1135--1198.}

\bibitem{KimK} {Kim K., Kwon S.}
{On pseudoconformal blow-up solutions to the self-dual Chern-Simons-Schrödinger equation: existence, uniqueness, and instability.}
Preprint arXiv:1909.01055. To appear in Mem. Amer. Math. Soc.

\bibitem{KimK2} {Kim K., Kwon S.}
{Construction of blow-up manifolds to the equivariant self-dual Chern-Simons-Schrödinger equation.
Preprint arXiv:2009.0294.}

\bibitem{Kim} {Kim K.}
{Rigidity of smooth finite-time blow-up for equivariant self-dual Chern-Simons-Schrödinger equation.
Preprint arXiv:2210.05412.}


\bibitem{KriegerS}{Krieger J. and Schlag W.} {Non-generic blow-up solutions for the critical focusing NLS in 1D. J. Eur. Math. Soc. (JEMS) {\bf 11} (2009), no. 1, 1--125.}
\MScN{MR2471133} \DOI{10.4171/JEMS/143}

\bibitem{MartelMR1}{Martel Y., Merle F. and Rapha\"el P.}
{Blow-up for the critical generalized Korteweg-de Vries equation I: dynamics near the soliton,} 
Acta Math. \textbf{212} (2014), 59--140.

\bibitem{MartelMR3}{Martel Y., Merle F. and Rapha\"el P.}
{Blow-up for the critical generalized Korteweg-de Vries equation III: exotic regimes,
Ann. Sc. Norm. Super. Pisa Cl. Sci. (5) \textbf{14} (2015), 575--631.}

\bibitem{MartelP}{Martel Y. and Pilod D.} {Finite point blowup for the critical generalized Korteweg-de Vries equation.
To appear in Ann. Sc. Norm. Super. Pisa Cl. Sci.}

\bibitem{MartelR}{Martel Y. and Rapha\"el P.} {Strongly interacting blow up bubbles for the mass critical NLS,
Ann. Sci. \'Ec. Norm. Sup\'er. \textbf{51} (2018), 701--737.}

\bibitem{Me90}
{Merle F.}
{Construction of solutions with exactly $k$ blow-up points for the Schr\"odinger equation with critical nonlinearity,} 
Comm. Math. Phys., \textbf{129} (1990), 223--240.

\bibitem{Me93} Merle F.
{Determination of blow-up solutions with minimal mass for nonlinear Schr\"odinger equations with critical power,
Duke Math. J. \textbf{9} (1993), 427--454.}

\bibitem{MerleR1}{Merle F. and Rapha\"el P.} {Profiles and quantization of the blow
up mass for the critical Schr\"odinger equation, Comm. Math. Phys. {\bf 253} (2005), no. 3,
675--704.}
\MScN{MR2116733} \DOI{10.1007/s00220-004-1198-0}

\bibitem{MerleR2}{Merle F. and Rapha\"el P.} {On a sharp lower bound on the blow-up
rate for the $L\sp 2$ critical nonlinear Schr\"odinger equation, J. Amer.
Math. Soc. {\bf 19} (2006), no. 1, 37--90.}
\MScN{MR2169042} \DOI{10.1090/S0894-0347-05-00499-6 }

\bibitem{MerleRS}{Merle F., Rapha\"el P. and Szeftel J.} {The instability of Bourgain-Wang solutions for the $L^2$ critical NLS. Amer. J. Math. {\bf 135} (2013), no. 4, 967--1017.}
\MScN{MR3086066} \DOI{10.1353/ajm.2013.0033}

\bibitem{Perelman}{Perelman G.}{ On the formation of singularities in solutions of the critical nonlinear Schr\"o\-din\-ger equation. Ann. Henri Poin\-ca\-r\'e {\bf 2} (2001), no. 4, 605--673.}
\MScN{MR1852922} \DOI{10.1007/PL00001048}

\bibitem{Raphael}{Rapha\"el P.} {Stability of the log-log bound for blow up
solutions to the critical non linear Schr\"o\-din\-ger equation, Math. Ann. {\bf
331} (2005), no. 3, 577--609.}
\MScN{MR2122541} \DOI{10.1007/s00208-004-0596-0}

\bibitem{RaphaelS}{Rapha\"el P. and Szeftel J.}
{Existence and uniqueness of minimal blow-up solutions to an inhomogeneous mass critical NLS, 
J. Am. Math. Soc. {\bf 24} (2011), No. 2, 471--546 .}

\bibitem{SulemSulem}
Sulem C. and Sulem P.-L.,
{\it The nonlinear Schr\"odinger equation. Self-focusing and wave collapse.}
Applied Mathematical Sciences, 139. Springer-Verlag, New York, 1999.

\bibitem{Weinstein83}
{Weinstein M.I.}
{Nonlinear Schr\"odinger equations and sharp interpolation estimates,
Comm. Math. Phys., \textbf{87} (1983), 567--576.}

\bibitem{Weinstein}{Weinstein M.I.} {Modulational stability of ground states of nonlinear
Schr\"o\-din\-ger equations. SIAM J. Math. Anal. {\bf 16} (1985), no. 3, 567--576.}
\MScN{MR0783974} \DOI{10.1137/0516034}

\end{thebibliography}
\end{document}